\documentclass[preprint,12pt]{elsarticle}

\usepackage[pdftex, pdfpagelabels, bookmarks, hyperindex, hyperfigures,
colorlinks=red, pagebackref]{hyperref}
\usepackage{diagbox}
\usepackage{amssymb}
\usepackage{amsmath,bm}
\usepackage{multirow}
\usepackage{makecell}
\usepackage{graphicx}
\usepackage{subfigure}
\usepackage{float}
\usepackage[titletoc, title]{appendix}
\usepackage{color}
\usepackage{lipsum}
\usepackage{amsfonts}
\usepackage{graphicx}
\usepackage{epstopdf}
\usepackage{algorithmic}
\usepackage{amsopn}
\usepackage{amsthm} 
\usepackage{cleveref}

\newtheorem{lemma}{Lemma}[section]
\newtheorem{theorem}{Theorem}[section]
\newtheorem{remark}{Remark}[section]
\newtheorem{corollary}{Corollary}[section]

\usepackage{amssymb}

\journal{XXX}

\allowdisplaybreaks

\begin{document}

\begin{frontmatter}



\title{
	Analysis of  a class of globally divergence-free HDG methods for stationary Navier-Stokes equations}


\author[label2]{Gang Chen\fnref{label3}}
\ead{cglwdm@scu.edu.cn}
\author[label2]{Xiaoping Xie\corref{cor1}\fnref{label4}}
\ead{xpxie@scu.edu.cn}
\cortext[cor1]{Corresponding author}
\address[label2]{School of Mathematics, Sichuan University, Chengdu 610064, China.}

\fntext[label3]{The first author's work was supported by  National Natural Science Foundation of China  (no. 12171341 and no. 11801063),
	the Fundamental Research Funds for the Central Universities (no. YJ202030). 
}
\fntext[label4]{The second author was supported by National Natural Science Foundation of China (no. 12171340 and no. 11771312). 
}

\begin{abstract}
This paper   analyzes a class of
globally divergence-free (and therefore pressure-robust) hybridizable discontinuous Galerkin (HDG) finite element methods for stationary Navier-Stokes equations. The methods use the $\mathcal{P}_{k}/\mathcal{P}_{k-1}$ $(k\geq1)$ discontinuous finite element combination for the velocity and pressure approximations in the interior of elements, and piecewise $\mathcal{P}_k/\mathcal{P}_{k}$ for the trace approximations of  the  velocity and pressure on   the inter-element boundaries. 
 It is shown that the uniqueness condition for  the discrete solution is guaranteed by  that for the continuous solution together with a sufficiently small mesh size.  Based on the derived discrete HDG Sobolev embedding properties, 
optimal error estimates are obtained.
Numerical experiments are performed to verify the theoretical analysis.
\end{abstract}


\end{frontmatter}


\color{black}
\section{Introduction}
Let $\Omega\subset \mathbb{R}^d$ $(d=2,3)$ \color{black} be a Lipschitz  polygonal/polyhedral domain.  We consider the following stationary Navier-Stokes equations:  seek \color{black}   \color{black} the velocity $\bm{u}$ and the pressure $p$ such that
\begin{eqnarray}
\left\{
\begin{aligned}
-\nu \Delta \bm{u}+\nabla\cdot(\bm{u}\otimes\bm{u})+\nabla p&=\bm{f}&\text{in} \ \Omega,\label{Or-NS}\\
\nabla\cdot\bm{u}&=0&\text{in} \ \Omega,\\
\bm{u}&=\bm{0}&\text{on} \ \partial \Omega. \\
\end{aligned}
\right.
\end{eqnarray}
\color{black}Here \color{black} $\nu=\text{Re}^{-1}>0$ is the fluid viscosity coefficient with $\text{Re}$ denoting the \text{Re}ynolds number, and $\bm{f}\in [L^2(\Omega)]^d$ is the given body force.

 A Galerkin mixed method  for the problem \eqref{Or-NS} (or even Stokes equations)  requires the pair of  finite element spaces for
	the velocity and pressure   to  satisfy  an  inf-sup stability condition (see, e.g. \cite{LBB1, LBB2, SV1,mixed1,mixed2, NS-Siam} and books \cite{Brezzi;Fortin1991, Girault.V;Raviart.P1986, mixedbook, mixedbook2, mixedbook4}). To circumvent the inf-sup difficulty,  many methods of stabilization  have been developed   for the Stokes or Navier-Stokes equations (cf. \cite{LS0,LS1,LS2,LS3, PP1,PP2,PP3,PGP1,PGP2, LPS1,LPS2,LPS3,CX2010, XXX2008,NS-Siam} and the references therein).

It is very important to preserve the mass conservation in the numerical solution of the incompressible fluid flows, since  the finite element methods with poor mass conservation may lead to instabilities for more complex problems (e.g. unsteady Naiver-Stokes equations) 
	(cf. \cite{bookin1,bookin2,bookin3,add1, add2, add4}).
	The exactly divergenve free discretizations automatically lead to pressure-robustness.	
	As  pointed out in \cite{MR3564690, MR3683678}, classical mixed methods for the Stokes equations, 
	constructed by satisfying the discrete inf-sup condition, usually  lead to the lack of pressure-robustness, i.e. the velocity error of the methods depends on the best approximation error of the pressure  scaled with the inverse of the viscosity. 
	Therefore, it is desirable to design globally divergence-free  methods for the Stokes equations to ensure mass conservation and pressure-robustness.
	
	\color{black}
	
	

%



In recent years the DG framework \cite{Cock-2000,Cock-2001} has become increasingly popular
due to its attractive features like local conservation of physical quantities and flexibility in
meshing. In \cite{Cock-equal,dg1999,dg2007}   DG methods were applied to solve the Navier-Stokes model \eqref{Or-NS}.
  A local discontinuous Galerkin (LDG) method for \eqref{Or-NS} was proposed in \cite{LDG},
where a globally divergence-free approximation for the velocity is obtained
by post-processing.
The hybridizable discontinuous Galerkin (HDG) framework, presented in \cite{hdg2009} 
 for diffusion problems, provides a unifying strategy for hybridization of finite element methods. The resultant HDG methods preserve the advantages of standard DG methods and lead to discrete systems of  significantly reduced sizes.  \color{black} We refer to \cite{ hdg1,hdg2,hdg3, NS-unsteady,MR3556409,AIAA1,AIAA2,TH1,MR3511719,MR3194122,IMA,MR3833698} for some HDG methods for the Stokes   equations, 
the Navier-Stokes equations and Stokes-like equations.
\color{black}
In particular,  in \cite{MR3833698} a technique  that  introduces the pressure trace on  the inter-element boundaries as a Lagrangian multiplier so as to derive a divergence-free velocity approximation,{   as same as our earlier work in \cite{MR3549196}}, was used to construct a
globally divergence-free HDG method for the unsteady Navier-Stokes 
equations.

	 It is well-known the existence and uniqueness of the weak solution to \eqref{Or-NS} is under the  smallness condition
	\color{black}
	\begin{eqnarray}
	(\mathcal{N}/\nu^2)\|\bm{f}\|_{*}<1, \label{unique}
	\end{eqnarray}
	where the constant $\mathcal{N}$ and the norm $\|\cdot\|_*$ are defined in Section 2. As for numerical schemes, the corresponding discrete smallness condition is
	\begin{eqnarray}
	(\mathcal{N}_h/\nu^2)\|\bm{f}\|_{*,h}<1,\label{unique-dis}
	\end{eqnarray}
\color{black}	
where $\mathcal{N}_h$ and   $\|\cdot\|_{*,h}$ are defined similarly to the continuous case but with specific finite element spaces (cf. Section 2).  
For  standard finite element approaches, it has been shown in \cite[]{MR548867} that
\begin{align*}
\lim_{h\to 0}\|\bm{f}\|_{*,h}&=\|\bm{f}\|_{*},\qquad
\lim_{h\to 0}\mathcal{N}_h=\mathcal{N}. 
\end{align*}
Thus, the continuous condition \eqref{unique}, together with a sufficiently small mesh size $h$,
	guarantees the discrete smallness condition \eqref{unique-dis} for the conforming methods.
In \cite{MR2136994,TH2,MR3626531} a fixed point argument was used to establish uniqueness for discontinuous Galerkin (DG)  or HDG methods under the condition
	\begin{align}
	(C_0/\nu^{2})\|\bm f\|_0< 1,
	\label{unique-dis2}
	\end{align} 
	where the constant $C_0$ depends on the underlying numerical method.
	Note that  \eqref{unique-dis2} is equivalent to
	\begin{align*}
	(\mathcal N/\nu^2)\|\bm f\|_*< \frac{\mathcal N\|\bm f\|_*}{C_0\|\bm f\|_0}
	\end{align*} 
 when $\bm f\neq\bm 0$, 
	which means that  
	 the condition \eqref{unique-dis2} can be 
	guaranteed by \eqref{unique}  if 
	\begin{align}\label{cond0}
	 \frac{\mathcal N\|\bm f\|_*}{C_0\|\bm f\|_0}\geq 1.
	\end{align} 
 However, the condition \eqref{cond0}  does not hold in general.

 	{ 
	So far, to our best knowledge,  there is no proof of \eqref{unique-dis} 
	under the condition \eqref{unique} for  nonstandard approaches.	
	We   emphasize that it is true,  like in
	\cite{MR2136994,TH2,MR3626531}, that if $\nu^{-2}\|\bm f\|_0$ is small enough, then \eqref{unique-dis2}  holds true and therefore the underlying discretization  has a unique solution, but there is  still a chance that \eqref{unique} holds true but \eqref{unique-dis2} does not! 
	So  it is possible that the continuous problem has a unique solution, while the corresponding  discrete scheme does not!  This theoretical gap will be filled by our analysis. 
	}
	
	In this paper, we shall analyze a class of  HDG methods for the  Navier-Stokes problem \eqref{Or-NS}. The methods use the $\mathcal{P}_{k}/\mathcal{P}_{k-1}$ $(k\geq1)$ discontinuous finite element combination for the velocity and pressure approximations in the interior of elements, and piecewise $\mathcal{P}_k/\mathcal{P}_{k}$ for the trace approximations of  the  velocity and pressure on   the inter-element boundaries. We note that the  finite element combinations used in our  methods {  are inherited from our previous paper \cite{MR3549196} for the Stokes equations, and later appears} in \cite{MR3833698} for the unsteady Navier-Stokes  equations.  
Our analysis is of the following features. 
\begin{itemize}
	
\item  { The discrete smallness  condition  \eqref{unique-dis}  for  the proposed HDG discritization is shown to be guaranteed by  the  continuous smallness condition \eqref{unique} together with a sufficiently small mesh size.  To the authors' knowledge, this is the first proof of such a conclusion for HDG methods, and it is expectable to extend the proof  to   other nonstandard methods. The key to   the proof is that we define two switch operators, one from the continuous spaces  to the HDG spaces and the other one from the HDG spaces to the continuous spaces, which satisfy certain strict  inequalities (cf. \eqref{bound1} and \eqref{bound2}).}

\item The methods are shown to yield globally divergence-free velocity approximations, and therefore are pressure-robust.	
	
\item Our treatment of the nonlinear term is different from that in  \cite{MR3833698}.  In our discretization, 
we design the nonlinear term as an antisymmetry form which makes the numerical scheme to be of stability independent of the choice of the stabilization parameter.  In \cite{MR3833698}   the nonlinear term  is discretized   directly,  and the stabilization parameter needs to be sufficiently large to ensure the stability.

\item Optimal error estimates are established based on the derived discrete HDG Sobolev embedding properties. We note that \cite{MR3833698} does not provide   convergence analysis for the methods therein.

%

\end{itemize}

The rest of this paper is organized as follows.
Section 2  introduces notations and the HDG formulations.
Section 3 gives some interpolation properties. 
Section 4  discusses  stability and continuous conditions.
Section 5 is devoted to  the a priori error analysis.
Section 6 derives   $L^2$ error estimates for the velocity.
Section 7 describes the local elimination property of the  HDG methods.
Finally, Section 8 provides numerical experiments.  
\color{black}

\section{HDG formulations}
\subsection{Notation}
\color{black}For \color{black}any bounded domain $\Lambda \subset \mathbb{R}^s$ $(\color{black}s=d,d-1\color{black})$, let $H^{m}(\Lambda)$ and $H^m_0(\Lambda)$  denote the usual \color{black}$m^{th}\color{black}$-order Sobolev spaces on $\Lambda$, and $\|\cdot\|_{m, \Lambda}$, $|\cdot|_{m,\Lambda}$  denote the norm and semi-norm on these spaces, respectively.
Let $(\cdot,\cdot)_{m,\Lambda}$ be the inner product of $H^m(\Lambda)$, with $(\cdot,\cdot)_{\Lambda}:=(\cdot,\cdot)_{0,\Lambda}$.
When $\Lambda=\Omega$, we set $\|\cdot\|_{m }:=\|\cdot\|_{m, \Omega}, |\cdot|_{m}:=|\cdot|_{m,\Omega}$, $(\cdot,\cdot):=(\cdot,\cdot)_{\Omega}$.  In particular, when $\color{black}\Lambda\in \mathbb{R}^{d-1}\color{black}$, we use $\langle\cdot,\cdot\rangle_{\Lambda}$ to replace $(\cdot,\cdot)_{\Lambda}$.
We note that   bold face fonts will be used for vector analogues of the Sobolev spaces along with vector-valued functions. For an integer $k\ge 0$, $\mathcal{P}_k(\Lambda)$  denotes the set of all polynomials defined on $\Lambda$ with degree not greater than $k$.
We also need the following spaces:
\begin{eqnarray*}
	L^2_0(\Omega):=\{q\in L^2(\Omega):(q,1)=0\},
\end{eqnarray*}
$$ \bm{H}({\rm div};\Lambda):=\left\{\bm v\in [L^2(\Lambda)]^s: \color{black}\nabla\cdot\color{black}\bm{v}\in L^2(\Lambda)\right\}.$$

Let \color{black} $\mathcal{T}_h=\bigcup\{T\}$ be a shape-regular simplical decomposition of the domain $\Omega$ with mesh size $h=\max_{T\in\mathcal{T}_h}h_T$, where $h_T$ is the diameter of $T$. Let $\mathcal{E}_h=\bigcup\{E\}$ be the union of all edges (faces) of $T\in\mathcal{T}_h$. For any simplex $T\in\mathcal{T}_h$ and $E\in\mathcal{E}_h$, we denote by $\color{black}\bm{n}_T\color{black}$ and $\bm{n}_E$   the \color{black}outward unit normal vectors along $\partial T$ and  $E$, respectively.  Let $h_E$ denote the diameter of  $E$.  

We use $\nabla_h$ and $\nabla_h\cdot$ to denote the piecewise-defined gradient and divergence with respect to the decomposition $\mathcal{T}_h$.  We also introduce the following mesh-dependent inner product and mesh-dependent  norm:
\begin{align*}
\langle u,v \rangle_{\partial\mathcal{T}_h}:=&\sum_{T\in\mathcal{T}_h}\langle u,v \rangle_{\partial T},\ \ \ \ \
\|u\|^2_{\partial \mathcal{T}_h}:=\sum_{T\in\mathcal{T}_h}\|u\|^2_{0,\partial T}.
\end{align*}

For  simplicity, throughout this paper we use $a\lesssim b$ ($a\gtrsim b$)
to denote $a\le Cb$ ($a\ge Cb$), where   $C$  is a positive constant independent of \color{black} mesh sizes \color{black}$h$, $h_T$, $h_E$ and \color{black}the fluid viscosity coefficient \color{black}$\nu$. In addition, $a\sim b$ simplifies  $a\lesssim b\lesssim a.$

\subsection{Basic results for Naiver-Stokes equations}
Introduce the spaces
\begin{eqnarray}
\bm{V}:=[H^1_0(\Omega)]^d,\quad
Q:=L^2_0(\Omega),\quad
\bm{W}:=\{\bm{v}\in\bm{V}:\nabla
\cdot \bm{v}=0\},\nonumber
\end{eqnarray}
and define the following trilinear form: for  any $(\bm{w},\bm{u},\bm{v})\in [H^1(\Omega)]^d\times[H^1(\Omega)]^d\times[H^1(\Omega)]^d$,
\begin{eqnarray}
b(\bm{w};\bm{u},\bm{v}):=\frac{1}{2}(\nabla\cdot(\bm{u}\otimes\bm{w}),\bm{v})
-\frac{1}{2}(\nabla\cdot(\bm{v}\otimes\bm{w}),\bm{u}).\nonumber
\end{eqnarray}
\color{black} By integration by parts, it \color{black} is easy to see
\begin{eqnarray}
b(\bm{w};\bm{u},\bm{v})=(\bm{w}\cdot\nabla\bm{u},\bm{v})&\forall \bm{w}\in\bm{W},\bm{u},\bm{v}\in\bm{V}.\nonumber
\end{eqnarray}
We   set
\begin{align}
\mathcal{N}:=&\sup_{\bm 0\neq\bm{u},\bm{v},\bm{w}\in \bm{W}}\frac{(\bm{w}\cdot\nabla\bm{u},\bm{v})}
{|\bm{w}|_1|\bm{u}|_1|\bm{v}|_1},\label{26}\\
\|\bm{f}\|_*:=&\sup_{\bm{0}\neq\bm{v}\in\bm{W}}\frac{(\bm{f},\bm{v})}
{|\bm{v}|_1}.\label{27}
\end{align}

\begin{theorem}[cf. \cite{Girault.V;Raviart.P1986}] \label{TH21} Let $\Omega$ be a bounded domain with a Lipschitz continuous boundary $\partial\Omega$. Given $\bm{f}\in [H^{-1}(\Omega)]^d$, the problem  \eqref{Or-NS} admits   at least one weak solution  $(\bm{u},p)\in \bm{W}\times Q$.
\end{theorem}

\begin{theorem}[cf. \cite{Girault.V;Raviart.P1986}]\label{TH22} Under the hypotheses of Theorem \ref{TH21}, let $(\bm{u},p)\in \bm{W}\times Q$ be a weak solution of  \eqref{Or-NS}, then $\bm{u}$ satisfies
	\begin{eqnarray}
	\|\nabla\bm{u}\|_0\le\nu^{-1}\|\bm{f}\|_{*}.\label{sta0}
	\end{eqnarray}
	In addition, if
	\begin{eqnarray}
	(\mathcal{N}/\nu^2)\|\bm{f}\|_{*}<1
	\end{eqnarray}
	holds, then
	the problem \eqref{Or-NS} admits a unique solution $(\bm{u},p)\in\bm{W}\times Q$.
\end{theorem}

%
In the rest of this paper, the solution $(\bm{u},p)$ is supposed to be unique and, more precisely, we assume that
\begin{eqnarray}
(\mathcal{N}/\nu^2)\|\bm{f}\|_{*}\le 1-\delta\text{ for some }\delta> 0. \label{uni}
\end{eqnarray}

\subsection{HDG  scheme}


For any integer $k\ge 1$, and integer $m\in\{k-1,k\}$, 
we introduce the following finite element spaces:
\begin{align*}
\mathbb{K}_h:=&\{\bm{\tau}_h:\bm{\tau}_h|_T\in [\mathcal{P}_m(T)]^{d\times d},\forall T\in\mathcal{T}_h\},  \\
\bm{V}_h:=&\{\bm{v}_h: \bm{v}_{h}|_T\in [\mathcal{P}_k(T)]^d,\forall T\in\mathcal{T}_h\},\\
\widehat{\bm{V}}_h:=&\{\widehat{\bm{v}}_{h}:\widehat{\bm{v}}_{h}|_E\in [\mathcal{P}_{k}(E)]^d,\forall E\in\mathcal{E}_h,\text{ and } \widehat{\bm{v}}_{h}|_{\partial \Omega}=\bm{0}\},\\
Q_h:=&\{q_h\in L^2_0(\Omega): q_{h}|_T\in \mathcal{P}_{k-1}(T),\forall T\in\mathcal{T}_h\},\\
\widehat{Q}_h:=&\{\widehat{q}_h|_E\in \mathcal{P}_{k}(E),\forall E\in\mathcal{E}_h\}.
\end{align*}
Introducing $\mathbb L=\nu\nabla \bm u$ in \eqref{Or-NS}, we can rewrite it as
\begin{eqnarray}
\left\{
\begin{aligned}
\nu^{-1}\mathbb L-\nabla \bm u&=0\quad \ \text{in} \ \Omega,\\
-\nabla\cdot\mathbb L+\nabla\cdot(\bm{u}\otimes\bm{u})+\nabla p&=\bm{f}\quad \ \text{in} \ \Omega,\label{mixed}\\
\nabla\cdot\bm{u}&=0\quad \ \text{in} \ \Omega,\\
\bm{u}&=\bm{0}\quad \ \text{on} \ \partial \Omega. \\
\end{aligned}
\right.
\end{eqnarray}

Then our HDG finite element \color{black} scheme for \eqref{mixed} is given as follows:
find
$(\mathbb{L}_h,\bm{u}_h,\widehat{\bm{u}}_h,$ $p_h,\widehat{p}_h)\in \mathbb{K}_h\times\bm{V}_h\times\widehat{\bm{V}}_h\times Q_h\times Q_h$ such that, for all
$(\mathbb{G}_h,\bm{v}_h,\widehat{\bm{v}}_h,q_h,\widehat{q}_h )\in \mathbb{K}_h\times\bm{V}_h\times\widehat{\bm{V}}_h\times Q_h\times Q_h$, 
\begin{subequations}
	\begin{align}
	\nu^{-1}(\mathbb{L}_h,\mathbb{G}_h)+(\bm{u}_h,\nabla_h\cdot\mathbb{G}_h)
	-\langle\widehat{\bm{u}}_h,\mathbb{G}_h\bm{n} \rangle_{\partial\mathcal{T}_h}=&0,\label{HDG01}\\
	(\bm{v}_h,\nabla_h\cdot\mathbb{L}_h)-\langle\widehat{\bm{v}}_h,\mathbb L_h\bm{n} \rangle_{\partial\mathcal{T}_h}+(\nabla_h\cdot\bm{v}_h,p_h)+\langle \bm{v}_h\cdot\bm{n},\widehat{p}_h \rangle_{\partial\mathcal{T}_h} &\nonumber\\
	\quad\quad-\frac{1}{2}( \bm{v}_h\otimes\bm{u}_h,\nabla_h\bm{u}_h  )+\frac{1}{2}( \widehat{\bm{v}}_h\cdot\widehat{\bm{u}}_h\bm{n},\bm{u}_h  ) &\nonumber\\
	\quad\quad+\frac{1}{2}( \bm{u}_h\otimes\bm{u}_h,\nabla_h\bm{v}_h  )-\frac{1}{2}( \widehat{\bm{u}}_h\cdot\widehat{\bm{u}}_h\bm{n},\bm{v}_h  ) &\nonumber\\
	\quad\quad-\nu\langle \tau( \bm{u}_{h}-\widehat{\bm{u}}_h), \bm{v}_{h}-\widehat{\bm{v}}_{h} \rangle_{\partial\mathcal{T}_h}=&-(\bm{f},\bm{v}_h),\label{HDG02}\\
	(\nabla_h\cdot\bm{u}_h,q_h)-\langle \bm{u}_h\cdot\bm{n},\widehat{q}_h \rangle_{\partial\mathcal{T}_h}=&0,\label{HDG03}
	\end{align}
\end{subequations}
where $\tau|_{E}=h_E^{-1} $ for all $E\in\mathcal{E}_h$. To simplify notation, we set 
$$\mathcal{W}_h:=( \bm{w}_h,\widehat{\bm{w}}_h),\quad \mathcal{U}_h:=( \bm{u}_h,\widehat{\bm{u}}_h),\quad \mathcal{V}_h=( \bm{v}_h,\widehat{\bm{v}}_h),$$
$$\mathcal{P}_h:=( p_h,\widehat{p}_h),\quad\mathcal{Q}_h=( q_h,\widehat{q}_h),$$
and define
\begin{align*}
a_h(\mathbb{L}_h,\mathbb{G}_h)&:=\nu^{-1}(\mathbb{L}_h,\mathbb{G}_h), \nonumber\\
c_h(\mathcal{U}_h,\mathbb{G}_h)&:=(\bm{u}_h,\nabla_h\cdot\mathbb{G}_h)
-\langle\widehat{\bm{u}}_h,\mathbb{G}_h\bm{n} \rangle_{\partial\mathcal{T}_h},\nonumber\\
d_h(\mathcal{V}_{h},\mathcal{P}_h)&:=(\nabla_h\cdot\bm{v}_h,p_h)+\langle \bm{v}_h\cdot\bm{n},\widehat{p}_h \rangle_{\partial\mathcal{T}_h}, \nonumber\\
s_h(\mathcal{U}_h,\mathcal{V}_h)&:=\nu\langle \tau( \bm{u}_{h}-\widehat{\bm{u}}_h), \bm{v}_{h}-\widehat{\bm{v}}_{h} \rangle_{\partial\mathcal{T}_h},\nonumber\\
b_h(\mathcal{W}_h;\mathcal{U}_h,\mathcal{V}_h)
&:=\frac{1}{2}( \bm{v}_h\otimes\bm{w}_h,\nabla_h\bm{u}_h  )-\frac{1}{2}( \widehat{\bm{v}}_h\cdot\widehat{\bm{w}}_h\bm{n},\bm{u}_h  ) \nonumber\\
&\quad-\frac{1}{2}( \bm{u}_h\otimes\bm{w}_h,\nabla_h\bm{v}_h  )+\frac{1}{2}( \widehat{\bm{u}}_h\cdot\widehat{\bm{w}}_h\bm{n},\bm{v}_h  ).\nonumber
\end{align*}

Then \eqref{HDG01}-\eqref{HDG03} can be rewritten as \color{black} a compact form\color{black}
:
find $(\mathbb{L}_h,\mathcal{U}_h,\mathcal{P}_h)\in  \mathbb{K}_h\times[\bm{V}_h\times\widehat{\bm{V}}_h]\times [Q_h\times Q_h]$ such that, for all $(\mathbb{G}_h,\mathcal{V}_h,\mathcal{Q}_h)\in  \mathbb{K}_h\times[\bm{V}_h\times\widehat{\bm{V}}_h]\times [Q_h\times Q_h]$,
{
	\begin{subequations}
		\begin{align}
		a_h(\mathbb{L}_h,\mathbb{G}_h)+c_h(\mathcal{U}_h,\mathbb{G}_h)=&0,\label{HDG1}\\
		c_h(\mathcal{V}_h,\mathbb{L}_h)+d_h(\mathcal{V}_{h},\mathcal{P}_h)-s_h(\mathcal{U}_h,\mathcal{V}_h)
		-  b_h(\mathcal{U}_h;\mathcal{U}_h,\mathcal{V}_h)=&-(\bm{f},\bm{v}_h),\label{HDG2}\\
		d_h(\mathcal{U}_h,\mathcal{Q}_h)=&0.\label{HDG3}
		\end{align}
	\end{subequations}
}
Introduce an operator $K_h: \bm{V}_h\times\widehat{\bm{V}}_h\to \mathbb{K}_h$ defined by
\begin{eqnarray}\label{K_h}
(K_h\mathcal{V}_h,\mathbb{G}_h)=-c_h(\mathcal{V}_h,\mathbb{G}_h)\quad \forall \mathcal{V}_h\in \bm{V}_h\times\widehat{\bm{V}}_h, \mathbb{G}_h\in \mathbb{K}_h.
\end{eqnarray}
It is easy to see that  $K_h$ is \color{black} well defined. 
\color{black}  From \eqref{HDG1} we immediately have
\begin{equation*}\label{Lh}
\mathbb{L}_h=\nu K_h\mathcal{U}_h.
\end{equation*}
Hence, by eliminating $\mathbb{L}_h$ in \eqref{HDG1} and \eqref{HDG2}, we can rewrite \eqref{HDG1}-\eqref{HDG3} as the following system:  

Find $(\mathbb L_h,\mathcal{U}_h,\mathcal{P}_h)\in \mathbb K_h\times[\bm{V}_h\times\widehat{\bm{V}}_h]\times [Q_h\times \widehat{Q}_h]$ such that
{
	\begin{subequations}
		\begin{align}
		\mathbb{L}_h-\nu K_h\mathcal{U}_h=&0,\label{HDG11}\\
		\nu(K_h\mathcal{U}_h,K_h\mathcal{V}_h)-d_h(\mathcal{V}_{h},\mathcal{P}_h) 
		+s_h(\mathcal{U}_h,\mathcal{V}_h)+  b_h(\mathcal{U}_h;\mathcal{U}_h,\mathcal{V}_h)=&(\bm{f},\bm{v}_h),
		\label{HDG22}\\
		d_h(\mathcal{U}_h,\mathcal{Q}_h)=&0,\label{HDG33}
		\end{align}
	\end{subequations}
holds for all $(\mathcal{V}_h,\mathcal{Q}_h)\in [\bm{V}_h\times\widehat{\bm{V}}_h]\times [Q_h\times \widehat{Q}_h]$.
}


	\section{Stability results}

We devote  this section to   the analysis of stability of scheme $(\ref{HDG1})$. To this end, we introduce the following two semi-norms:
\color{black}
\begin{align*}
\|\mathcal{V}_h\|_V^2&:= \|K_h(\bm{v}_{h},\widehat{\bm{v}}_h)\|^2_0+\|\tau^{\frac 1 2}(\bm{v}_{h}-\widehat{\bm{v}}_{h})\|^2_{\partial \mathcal{T}_h}, \\
\|\mathcal{Q}_h\|_Q^2&:=\|{q}_{h}\|^2_0+\|\tau^{-\frac 1 2}(q_h-\widehat{q}_h)\|_{\partial \mathcal{T}_h}^2.
\end{align*}
\color{black}
Here we recall that $K_h$ is given by \eqref{K_h}, and  $\tau|_{E}=h_E^{-1}$, for all $E\in\mathcal{E}_h$.

It is easy to see that $\|\cdot\|_V$ and $\|\cdot\|_Q$ are  norms on   $\bm{V}_h\times\widehat{\bm{V}}_h$ and $Q_h\times  \widehat Q_h$, respectively. In fact,
if $\|\mathcal{V}_h\|_V=0$, then
\begin{eqnarray*}
	\nabla_h\bm{v}_h=\bm{0}, \quad
	\bm{v}_h-\widehat{\bm{v}}_h=\bm{0},
\end{eqnarray*}
which mean that $\bm{v}_h$ is piecewise constant with respect to $\mathcal{T}_h$,
and $\bm{v}_h=\widehat{\bm{v}}_h$ on $\mathcal{E}_h$. Since
$\widehat{\bm{v}}_h=\bm{0}$ on $\partial \Omega$, then $\bm{v}_h=\widehat{\bm{v}}_h=\bm{0}$. Thus, $\mathcal{V}_h=\bm{0}$.
Similarly, we can show  $\|\mathcal{Q}_h\|_Q=0$ leads to $\mathcal{Q}_h=\bm{0}$.

\subsection{Basic results}

\begin{lemma} \label{lem-sim}For any $\mathcal{V}_h=(\bm{v}_h,\widehat{\bm{v}}_h)\in\bm{V}_h\times\widehat{\bm{V}}_h$, it holds
	\begin{eqnarray}\label{sim}
	\|\mathcal{V}_h\|_V\sim \|\nabla_h\bm{v}_h\|_0+\|\tau^{\frac 1 2}(\bm{v}_h-\widehat{\bm{v}}_h)\|_{\partial \mathcal{T}_h} .
	\end{eqnarray}
\end{lemma}
\begin{proof}
	For $\mathcal{V}_h\in\bm{V}_h\times\widehat{\bm{V}}_h$,
	by the definition of $K_h$, it holds
	\begin{align*}
	(K_h\mathcal{V}_h,\mathbb{G}_h)=(\nabla_h\bm{v}_h,\mathbb{G}_h )
	+\langle\widehat{\bm{v}}_h-\bm{v}_h,\mathbb{G}_h\bm{n} \rangle_{\mathcal{T}_h}
	\end{align*}
	for all $\mathbb{G}_h\in \mathbb{K}_h$.
	By taking $\mathbb{G}_h=K_h\mathcal{V}_h$
	and $\mathbb{G}_h=\nabla_h\bm{v}_h$ in this equality, respectively,
	the estimate \eqref{sim} follows from  Holder's inequality and the inverse inequality.
\end{proof}



The following embedding relationships are standard (cf. \cite{adma, book2}).

%

Before proving this theorem, we begin with a well-known discrete Sobolev embedding inequality in \cite[Theorem 2.1]{MR2629994} and a continuous Sobolev embedding inequality. 

\color{black}

\begin{lemma} [cf. {\cite[Theorem 2.1]{MR2629994}}] \label{old-embed} There is a constant $C>0$, such that
	\begin{align*}
	\|\bm v_h\|_{0,\mu}\le
	C\left[
	\sum_{K\in\mathcal{T}_h}
	\|\nabla\bm v_h\|^2_{0,T}
	+\sum_{F\in\mathcal{E}_h}h_E^{-1}\|[\![\bm v_h]\!]\|^2_{0,E}
	\right]^{\frac 1 2},
	\end{align*}
	for  $\mu$ satisfying
	\begin{eqnarray*}
		\left\{
		\begin{aligned}
			&1\le \mu< \infty,&\text{ if } \ d=2,\\
			&1\le \mu\le 6 ,&\text{ if } \ d=3,
		\end{aligned}
		\right.
	\end{eqnarray*}
	and for all $\bm v_h\in \bm V_h$.
\end{lemma}

\color{black}

Based on this lemma, we can obtain a discrete HDG Sobolev embedding relation, which will be used to derive the continuity results in Lemma \ref{conti}.

\begin{lemma}[HDG Sobolev embedding] \label{discrete-soblev} It holds
	\begin{eqnarray*}
		\|\bm{v}_{h}\|_{0,\mu}\lesssim \|\mathcal{V}_h\|_V &\forall \mathcal{V}_h=(\bm{v}_h,\widehat{\bm{v}}_h)\in\bm{V}_h\times\widehat{\bm{V}}_h
	\end{eqnarray*}
	for  $\mu$ satisfying
	\begin{eqnarray*}
		\left\{
		\begin{aligned}
			&1\le \mu< \infty,&\text{ if } \ d=2,\\
			&1\le \mu\le 6 ,&\text{ if } \ d=3.
		\end{aligned}
		\right.
	\end{eqnarray*}
\end{lemma}

\subsection{Stability  conditions}

For the trilinear form $b_h(\cdot;\cdot,\cdot)$, we have the following boundedness results.

\begin{lemma} \label{conti} For all  $\mathcal{U}_h,\mathcal{V}_h,\mathcal{W}_h\in\bm{V}_h\times\widehat{\bm{V}}_h$ with $\mathcal{U}_h=(\bm{u}_h,\widehat{\bm{u}}_h), \mathcal{V}_h=(\bm{v}_h,\widehat{\bm{v}}_h)$, and $ \mathcal{W}_h=(\bm{w}_h,\widehat{\bm{w}}_h)$,
	it holds
	\begin{align}
	|b_h(\mathcal{W}_h;\mathcal{U}_h,\mathcal{V}_h)|\lesssim& \|\mathcal{W}_h\|_V  \|\mathcal{U}_h\|_V  \|\mathcal{V}_h\|_V,\label{B1}\\
	|b_h(\mathcal{W}_h;\mathcal{U}_h,\mathcal{V}_h)|\lesssim& (\|\bm{w}_{h}\|_{0,3}+h^{1-\frac d 6}\|\mathcal{W}_h\|_V ) \|\mathcal{U}_h\|_V  \|\mathcal{V}_h\|_V. \label{B2}
	\end{align}
	Moreover, if $\bm{w}_{h}\in\bm{H}({\rm div},\Omega)\cap \bm{V}_h$
	\color{black} 
	with $\nabla\cdot \bm w_h=0,$ 
	\color{black}
	then 
	\begin{align}
	|b_h(\mathcal{W}_h;\mathcal{U}_h,\mathcal{V}_h)|\lesssim& (\|\bm{u}_{h}\|_{0,3}+h^{1-\frac d 6}\|\mathcal{U}_h\|_V ) \|\mathcal{W}_h\|_V  \|\mathcal{V}_h\|_V,\label{B3}\\
	|b_h(\mathcal{W}_h;\mathcal{U}_h,\mathcal{V}_h)|\lesssim& (\|\bm{v}_{h}\|_{0,3}+h^{1-\frac d 6}\|\mathcal{V}_h\|_V) \|\mathcal{W}_h\|_V  \|\mathcal{U}_h\|_V .\label{B4}
	\end{align}
\end{lemma}
\begin{proof}
	We first show \eqref{B2}.
	For all $\mathcal{W}_h,\mathcal{U}_h,\mathcal{V}_h\in\bm{V}_h\times\widehat{\bm{V}}_h$, by the definition of $b_h(\cdot;\cdot,\cdot)$ we have
	\begin{align} \label{B417}
	2b_h(\mathcal{W}_h;\mathcal{U}_h,\mathcal{V}_h)
	=&\left[( \bm{v}_h\otimes\bm{w}_h,\nabla_h\bm{u}_h  )-( \bm{u}_h\otimes\bm{w}_h,\nabla_h\bm{v}_h  )\right] \nonumber\\
	&\color{black}+\left[( \widehat{\bm{w}}_h\cdot\bm{n},\widehat{\bm{u}}_h\cdot\bm{v}_h  )-( \widehat{\bm{w}}_h\cdot\bm{n},\widehat{\bm{v}}_h\cdot\bm{u}_h  )\right]\nonumber\\
	=:& R_1+R_2.
	\end{align}
	We use the H\"{o}lder's inequality, Lemma \ref{discrete-soblev} with $\mu=6$, and the relationship \eqref{sim} to get
	\begin{align}
	|R_1|\le& \sum_{T\in\mathcal{T}_h}\|\bm{u}_{h}\|_{0,6,T}\|\bm{w}_{h}\|_{0,3,T}\|\nabla\bm{v}_{h}\|_{0,T}
	+\sum_{T\in\mathcal{T}_h}\|\bm{v}_{h}\|_{0,6,T}\|\bm{w}_{h}\|_{0,3,T}\|\nabla\bm{u}_{h}\|_{0,T} \nonumber\\
	\lesssim&\|\bm{w}_{h}\|_{0,3}\cdot\|\mathcal{U}_h\|_V \cdot\|\mathcal{V}_h\|_V.
	\end{align}
	From triangle inequality we have
	\begin{align}
	|R_2|\le&|\langle(\bm{w}_{h}-\widehat{\bm{w}}_{h})\bm{n},(\bm{u}_{h}-\widehat{\bm{u}}_h)\bm{v}_{h} \rangle_{\partial\mathcal{T}_h} | +|\langle\bm{w}_{h}\bm{n},(\bm{u}_{h}-\widehat{\bm{u}}_h)\bm{v}_{h} \rangle_{\partial\mathcal{T}_h} | \nonumber\\
	&+|\langle(\bm{w}_{h}-\widehat{\bm{w}}_{h})\bm{n},(\bm{v}_{h}-\widehat{\bm{v}}_{h})\bm{u}_{h} \rangle_{\partial\mathcal{T}_h} | +|\langle\bm{w}_{h}\bm{n},(\bm{v}_{h}-\widehat{\bm{v}}_{h})\bm{u}_{h} \rangle_{\partial\mathcal{T}_h} | \nonumber\\
	=:&\sum_{i=1}^4T_{i}.
	\end{align}
	By the H\"{o}lder's inequality, the inverse inequality, and \eqref{partial2}, we obtain
	\begin{align}
	T_{1}\le&
	\sum_{T\in\mathcal{T}_h}
	\|\bm{w}_{h}-\widehat{\bm{w}}_{h}\|_{0,\partial T}
	\|\bm{u}_{h}-\widehat{\bm{u}}_h\|_{0,3,\partial T} \|\bm{v}_{h}\|_{0,6,\partial T}\nonumber\\
	\lesssim&\sum_{T\in\mathcal{T}_h}
	\|\bm{w}_{h}-\widehat{\bm{w}}_{h}\|_{0,\partial T}
	h_T^{-(d-1)/6}\|\bm{u}_{h}-\widehat{\bm{u}}_h\|_{0,\partial T} h_T^{-\frac 1 6}\|\bm{v}_{h}\|_{0,6, T}\nonumber\\
	=&\sum_{T\in\mathcal{T}_h}
	h_T^{-\frac 1 2}\|\bm{w}_{h}-\widehat{\bm{w}}_{h}\|_{0,\partial T}
	h_T^{-\frac 1 2}\|\bm{u}_{h}-\widehat{\bm{u}}_h\|_{0,\partial T} h_T^{1-\frac d 6}\|\bm{v}_{h}\|_{0,6, T}\nonumber\\
	\lesssim&h^{1-\frac d 6}\|\bm{v}_{h}\|_{0,6}\sum_{T\in\mathcal{T}_h}
	h_T^{-\frac 1 2}\|\bm{w}_{h}-\widehat{\bm{w}}_{h}\|_{0,\partial T}
	h_T^{-\frac 1 2}\|\bm{u}_{h}-\widehat{\bm{u}}_h\|_{0,\partial T}\nonumber\\
	\lesssim& h^{1-\frac d 6}\|\mathcal{W}_h\|_V \cdot \|\mathcal{U}_h\|_V \cdot\|\mathcal{V}_h\|_V,\\
	T_{2}\le&
	\sum_{T\in\mathcal{T}_h}
	\|\bm{w}_{h}\|_{0,3,\partial T}
	\|\bm{u}_{h}-\widehat{\bm{u}}_h\|_{0,\partial T}
	\|\bm{v}_{h}\|_{0,6,\partial T}
	\nonumber\\
	\lesssim&
	\sum_{T\in\mathcal{T}_h}
	h_T^{-\frac 1 3}\|\bm{w}_{h}\|_{0,3,T}
	\|\bm{u}_{h}-\widehat{\bm{u}}_h\|_{0,\partial T}
	h_T^{-\frac 1 6}\|\bm{v}_{h}\|_{0,6,T}\nonumber\\
	=&\sum_{T\in\mathcal{T}_h}
	\|\bm{w}_{h}\|_{0,3,T}
	h_T^{-\frac 1 2}\|\bm{u}_{h}-\widehat{\bm{u}}_h\|_{0,\partial T}
	\|\bm{v}_{h}\|_{0,6,T}\nonumber\\
	\lesssim&\|\bm{w}_{h}\|_{0,3}\cdot\|\mathcal{U}_h\|_V \cdot\|\mathcal{V}_h\|_V.
	\end{align}
	Similarly, we have
	\begin{align}
	T_{3}\lesssim& h^{1-\frac d 6}\|\mathcal{W}_h\|_V \cdot \|\mathcal{U}_h\|_V \cdot\|\mathcal{V}_h\|_V,\\
	T_4\lesssim& \|\bm{w}_{h}\|_{0,3}\cdot\|\mathcal{U}_h\|_V \cdot\|\mathcal{V}_h\|_V.\label{B423}
	\end{align}
	As a result, the desired inequality   \eqref{B2} follows from \eqref{B417}-\eqref{B423}.

	In light of  the facts that $h\lesssim 1$, $1-\frac d 6\ge \frac 1 2$, and $\|\bm{w}_{h}\|_{0,3}\lesssim\|\mathcal{W}_h\|_V$,  the inequality \eqref{B2} indicates  \eqref{B1}.

	%
	%
	%
	%
	%

	Now we turn  to show   \eqref{B3}.
	By   integration by parts and the fact $\color{black}\langle\bm{w}_{h}\cdot \bm{n},\widehat{\bm{u}}_h\cdot \widehat{\bm{v}}_{h} \rangle_{\partial\mathcal{T}_h}$ $=0$ we  get
	\begin{align}\label{b424}
	2b_h(\mathcal{W}_h;\mathcal{U}_h,\mathcal{V}_h)=&(\bm{v}_{h}\otimes\bm{w}_{h},\nabla_h\bm{u}_{h})-(\bm{u}_{h}\otimes\bm{w}_{h},\nabla_h\bm{v}_{h})\nonumber\\
	&+\langle \widehat{\bm{u}}_h\otimes\widehat{\bm{w}}_{h}\bm{n},\bm{v}_{h} \rangle_{\partial\mathcal{T}_h}
	-\langle \widehat{\bm{v}}_{h}\otimes\widehat{\bm{w}}_{h}\bm{n},\bm{u}_{h} \rangle_{\partial\mathcal{T}_h}\nonumber\\
	=&-2(\bm{u}_{h}\otimes\bm{w}_{h},\nabla_h\bm{v}_{h})+(\nabla\cdot\bm{w}_{h},\bm{u}_{h}\bm{v}_{h})\nonumber\\
	&+\langle\bm{w}_{h} \bm{n},\bm{u}_{h} \bm{v}_{h}-\widehat{\bm{u}}_h \widehat{\bm{v}}_{h} \rangle_{\partial\mathcal{T}_h}
	+\langle\widehat{\bm{w}}_{h}\bm{n}, \widehat{\bm{u}}_h\bm{v}_{h}-\bm{u}_{h}\widehat{\bm{v}}_{h} \rangle_{\partial\mathcal{T}_h}\nonumber \\
	=&\big(-2(\bm{u}_{h}\otimes\bm{w}_{h},\nabla_h\bm{v}_{h})+(\nabla\cdot\bm{w}_{h},\bm{u}_{h}\bm{v}_{h})\big)\nonumber\\
	&+2\langle\bm{w}_{h} \bm{n},\bm{u}_{h} (\bm{v}_{h}-\widehat{\bm{v}}_{h}) \rangle_{\partial\mathcal{T}_h}\nonumber\\
	&+\langle\bm{w}_{h}\cdot \bm{n},(\bm{u}_{h}-\widehat{\bm{u}}_h) \cdot(\widehat{\bm{v}}_{h}-\bm{v}_{h}) \rangle_{\partial\mathcal{T}_h}\nonumber\\
	&+\langle(\widehat{\bm{w}}_{h}-\bm{w}_{h}) \cdot\bm{n},(\widehat{\bm{u}}_h-\bm{u}_{h})\cdot\bm{v}_{h}  \rangle_{\partial\mathcal{T}_h}\nonumber\\
	&+\langle(\widehat{\bm{w}}_{h}-\bm{w}_{h})\cdot \bm{n},\bm{u}_{h}\cdot(\bm{v}_{h}-\widehat{\bm{v}}_{h})  \rangle_{\partial\mathcal{T}_h}\nonumber\\
	=:&\sum_{i=1}^5S_i
	\end{align}
	In what follows we estimate $S_i$ term by term. It is easy to see
	\begin{align}
	S_1\le& 2\|\bm{u}_{h}\|_{0,3}\|\bm{w}_{h}\|_{0,6}\|\nabla_h\bm{v}_{h}\|_{0}
	+d\|\bm{u}_{h}\|_{0,3}\|\bm{u}_{h}\|_{0,6}\|\nabla_h\bm{w}_{h}\|_{0}\nonumber\\
	\lesssim&\|\bm{u}\|_{0,3}\|\mathcal{W}_h\|_V\| \mathcal{V}_h\|_V,\\
	S_2\le&2\sum_{T\in\mathcal{T}_h}\|\bm{w}_{h}\|_{0,6,\partial T}
	\|\bm{u}_{h}\|_{0,3,\partial T}
	\|\bm{v}_{h}-\bm{v}_{h}\|_{0,\partial T}\nonumber\\
	\lesssim&\sum_{T\in\mathcal{T}_h}\|\bm{w}_{h}\|_{0,6, T}
	\|\bm{u}_{h}\|_{0,3,T}
	h_T^{-\frac 1 2}\|\bm{v}_{h}-\bm{v}_{h}\|_{0,\partial T}\nonumber\\
	\lesssim&\|\bm{u}_{h}\|_{0,3}\|\mathcal{W}_h\|_V \| \mathcal{V}_h\|_V,\\
	S_3\le&\sum_{T\in\mathcal{T}_h}\|\bm{w}_{h}\|_{0,6,\partial T}
	\|\bm{u}_{h}-\widehat{\bm{u}}_h\|_{0,3,\partial T}
	\|\bm{v}_{h}-\widehat{\bm{v}}_{h}\|_{0,\partial T}\nonumber\\
	\lesssim&\sum_{T\in\mathcal{T}_h}h_T^{-\frac 1 6}\|\bm{w}_{h}\|_{0,6, T}
	h_T^{-(d-1)/6}\|\bm{u}_{h}-\widehat{\bm{u}}_h\|_{0,\partial T}
	\|\bm{v}_{h}-\bm{v}_{h}\|_{0,\partial T}\nonumber\\
	\lesssim&h^{1-\frac d 6}\|\mathcal{W}_h\|_V \|\mathcal{U}_h\|_V\| \mathcal{V}_h\|_V.
	\end{align}
	Similarly, we have
	\begin{align}
	S_4\lesssim& h^{1-\frac d 6}\|\mathcal{W}_h\|_V \|\mathcal{U}_h\|_V \| \mathcal{V}_h\|_V, \\
	S_5\lesssim& h^{1-\frac d 6}\|\mathcal{W}_h\|_V \|\mathcal{U}_h\|_V \| \mathcal{V}_h\|_V.\label{b430}
	\end{align}
	Therefore,  the inequality \eqref{B3} follows from  \eqref{b424}-\eqref{b430}.
	
	The inequality  \eqref{B4} follows similarly.
\end{proof}


%
%


For the bilinear form $d_h(\cdot,\cdot)$, we have the following inf-sup inequality.

\begin{theorem}[Inf-sup Stability] \label{theoremLBB}  For all $\mathcal{P}_h\in Q_h\times \widehat{Q}_h$, it holds
	\begin{eqnarray*}
		\sup_{\bm 0\neq \mathcal{V}_h\in \bm{V}_h\times\widehat{\bm{V}}_h}\frac{d_h(\mathcal{V}_{h},\mathcal{P}_h)}{\|\mathcal{V}_h\|_V }\gtrsim\|\mathcal{P}_h\|_Q.
	\end{eqnarray*}
\end{theorem}
\begin{proof} We use the Fortin technique to prove the discrete inf-sup condition.
	
	\textbf{Step 1.}
	Since $p_{h}\in Q$, by the continuous inf-sup condition \cite[Chapter 1, Corollary 2.4]{MR851383}, there exists \color{black} a function \color{black}$\bm{v}\in H^1_0(\Omega)$ such that
	\begin{align}\label{327}
	&\nabla\cdot\bm{v}=-p_{h},\qquad\qquad
	|\bm{v}|_1\lesssim \|p_{h}\|_0.
	\end{align}
	\color{black} 
    Take
	$$\mathcal{R}_h:=(\bm{r}_{h},\widehat{\bm{r}}_h)= (\bm{P}^{BDM}_{k}\bm{v},\bm{\Pi}_{k}^{\partial}\bm{v})\in\bm{V}_h\times\widehat{\bm{V}}_h, $$
	then  \color{black} from the definition of $d_h(\cdot,\cdot)$ and \color{black} integration by parts \color{black} it follows \color{black}
	\begin{align}
	d_h(\mathcal{R}_h,\mathcal{P}_h )
	=&- (p_{h},\nabla\cdot\bm{P}^{BDM}_{k}\bm{v})+ \langle \widehat{p}_h,\bm{P}^{BDM}_{k}\bm{v}\cdot\bm{n} \rangle_{\partial\mathcal{T}_h}\nonumber\\
	=& (\color{black}\nabla_h\color{black} p_{h},\bm{P}^{BDM}_{k}\bm{v})+ \langle \widehat{p}_h-p_{h},\bm{P}^{BDM}_{k}\bm{v}\cdot\bm{n} \rangle_{\partial\mathcal{T}_h}.\nonumber
	\end{align}
	\color{black}In view of the properties of $\bm{P}^{BDM}_{k}$ in \color{black} \eqref{BDM1} and \eqref{BDM2},  \color{black} integration by parts \color{black}, the fact $\langle \widehat{p}_h,\bm{v}\cdot\bm{n} \rangle_{\partial\mathcal{T}_h}=0$, and \eqref{327}, we have
	\begin{align*}
	d_h(\mathcal{R}_h,\mathcal{P}_h ) =& (\color{black}\nabla_h\color{black} p_{h},\bm{v})+ \langle \widehat{p}_h-p_{h},\bm{v}\cdot\bm{n} \rangle_{\partial\mathcal{T}_h},\nonumber\\
	=&- (p_{h},\nabla\cdot\bm{v})+ \langle \widehat{p}_h,\bm{v}\cdot\bm{n} \rangle_{\partial\mathcal{T}_h}\nonumber\\
	=&- (p_{h},\nabla\cdot\bm{v})\nonumber\\
	=& \|p_{h}\|^2_0.
	\end{align*}
	Since $\bm{r}_{h}= \bm{P}^{BDM}_{k}\bm{v},\widehat{\bm{r}}_h=\bm{\Pi}_{k}^{\partial}\bm{v}$, by    \eqref{BDM31}, \eqref{327}, and the approximation properties of  $\bm{P}^{BDM}_{k}$ and $\bm{\Pi}_{k}^{\partial}$, we  get
	\begin{align*}
	\|\nabla_h\bm{r}_{h}\|_{0}+\|\tau^{\frac 1 2}(\bm{r}_{h}-\widehat{\bm{r}}_{h})\|_{0,\partial \mathcal{T}_h}&=\|\nabla_h\bm{P}^{BDM}_{k}\bm{v}\|_{0}
	+\|\tau^{\frac 1 2}(\bm{P}^{BDM}_{k}\bm{v}-\bm{v})\|_{0,\partial \mathcal{T}_h}\nonumber\\
	&\lesssim |\bm{v}|_{1}\lesssim  \|p_{h}\|_{0}.
	\end{align*}
	Thus, by Lemma \ref{lem-sim} and  the definitions of  $\|\cdot\|_V$ and $\|\cdot\|_Q$,  we obtain
	\begin{align*}
	\|\mathcal{R}_h\|_V\lesssim  \|p_{h}\|_0\le \|\mathcal{P}_h\|_Q.
	\end{align*}
	
	\textbf{Step 2.}
	Let  $\bm{w}_h\in \bm{V}_h$, with $ \bm{w}_h|_T\in [\mathcal{P}_{k}(T)]^d$,  be determined by  
	\begin{align}
	(\bm{w}_h,\nabla\bm{\tau}_{k-1})_T&=0&\forall \bm{\tau}_{k-1}\in \mathcal{P}_{k-1}(T),\label{434}\\
	\langle\bm{w}_h\cdot\bm{n},\tau_{k} \rangle_E&=h_E\langle p_h-\widehat{p}_h,\tau_{k} \rangle_E&\forall \tau_k\in \mathcal P_{k}(E), E\in\partial T. \label{435}
	\end{align}
	for any $T\in\mathcal{T}_h$.     Standard scaling arguments show
	\begin{align*}
	\|\bm{w}_h\|_{0,T}\lesssim h_T^{3/2}\|(p_h-\widehat{p}_h)\|_{0,\partial T},
	\end{align*}
	which combines with an inverse inequality, yields
	\begin{equation}\label{wp437}
	\|\mathcal{W}_h\|_V\lesssim \|\mathcal{P}_h\|_V
	\end{equation}
	with   $\mathcal{W}_h:=(\bm{w}_h,\bm{0})$. 
	From  \eqref{434}-\eqref{435} and the definition of $d_h(\cdot,\cdot)$, it  follows
	\begin{align}\label{d-rp}
	d_h(\mathcal{W}_h,\mathcal{P}_h )
	&=- (p_{h},\nabla\cdot\bm{w}_h)+ \langle \widehat{p}_h,\bm{w}_h\cdot\bm{n} \rangle_{\partial\mathcal{T}_h}\nonumber\\
	&= (\color{black}\nabla_h\color{black} p_{h},\bm{w}_h)+ \langle \widehat{p}_h-p_{h},\bm{w}_h\cdot\bm{n} \rangle_{\partial\mathcal{T}_h}.\nonumber\\
	&=\|\tau^{-\frac 1 2}(p_h-\widehat{p}_h)\|^2_{\partial \mathcal{T}_h}.
	\end{align}
	\textbf{Step 3.}
	Take $\mathcal{V}_h=\mathcal{R}_h+\mathcal{W}_h\in\bm{V}_h\times\widehat{\bm{V}}^0$, and  \eqref{d-rp}, then we get
	\begin{eqnarray}\label{439}
	d_h(\mathcal{V}_{h},\mathcal{P}_h)=d_h(\mathcal{R}_{h},\mathcal{P}_h)+d_h(\mathcal{W}_{h},\mathcal{P}_h)=\|\mathcal{P}_h\|^2_Q.
	\end{eqnarray}
	By \eqref{434} and \eqref{wp437} we obtain
	\begin{eqnarray*}
		\|\mathcal{V}_h\|_V\le \|\mathcal{R}_h\|_V +\|\mathcal{W}_{h}\|_V \lesssim \|\mathcal{P}_h\|_Q.
	\end{eqnarray*}
	This result, together with \eqref{439}, implies
	\begin{eqnarray*}
		d_h(\mathcal{V}_{h},\mathcal{P}_h)\gtrsim\|\mathcal{P}_h\|_Q\|\mathcal{V}_h\|_Q,
	\end{eqnarray*}
	which finishes the proof.
\end{proof}

\section{Existence and uniqueness of discrete solution}

\subsection{Globally divergence-free velocity approximation}\label{divergence}

\begin{theorem} \label{Th5.1} If $\mathcal{U}_h:=(\bm u_h,\widehat{\bm u}_h)\in\bm{V}_h\times\widehat{\bm{V}}_h$ satisfies \eqref{HDG3}, then we have
	\begin{align*}
	&\bm{u}_{h}\in \bm{H}({\rm div};\Omega),\quad  \nabla\cdot\bm{u}_{h}=0.
	\end{align*}
\end{theorem}
\begin{proof} Introduce  a function $\widehat{r}_{h}\in L^2(\mathcal{E}_h)$ with 
	\begin{eqnarray}
	\widehat{r}_h|_E: =\left\{
	\begin{aligned}
	-(\bm{u}_{h}\cdot\bm{n}\color{black})|_{\partial T_1\cap E}-(\bm{u}_{h}\cdot\bm{n}\color{black})|_{\partial T_2\cap E}\color{black},\quad \text{if } E\in\mathcal{E}_h/\partial\Omega,\\
	{0},\quad \text{if } E\in \partial\Omega
	\end{aligned}
	\right.
	\end{eqnarray}
	for any $E\in \mathcal{E}_h$, 
	where   $T_1,T_2\in\mathcal{T}_h$ are the adjacent elements sharing the common edge (face) $E$. \color{black}
	Then we take $\mathcal{Q}_h=(\nabla_h\cdot\bm{u}_{h}-c_0,\widehat{r}_{h}-c_0)$ in $(\ref{HDG3})$, with $c_0=(\nabla_h\cdot\bm{u}_{h},1)/|\Omega|$, to get
	\begin{eqnarray}
	\|\nabla_h\cdot\bm{u}_{h}\|^2_0+\sum_{E\in \mathcal{E}_h/\partial\Omega,E=\partial T_1\cap \partial T_2}\|(\bm{u}_{h}\cdot\bm{n}\color{black})|_{\partial T_1\cap E}+(\bm{u}_{h}\cdot\bm{n}\color{black})|_{\partial T_2\cap E}\|^2_{0,E}=0,
	\end{eqnarray}
	which implies the desired conclusion. 
\end{proof}

\subsection{Existence result}
Define
\begin{eqnarray*}
	\bm{W}_h:=\{\mathcal{W}_h\in\bm{V}_h\times\widehat{\bm{V}}_h:d_h(\mathcal{W}_{h};\mathcal{Q}_h)=0,\forall \mathcal{Q}_h\in Q_h\times\widehat{Q}_h\}.
\end{eqnarray*}
Notice that $\bm{W}_h$ is non-trivial due to Theorem \ref{theoremLBB}.
From Theorem \ref{Th5.1},   for  $\color{black}\mathcal{V}_h=(\bm{v}_h,\widehat{\bm{v}}_h)\in\bm{W}_h$,  we have $\bm{v}_{h}\in \bm{H}({\rm div},\Omega)$ and $\nabla\cdot\bm{v}_{h}=0$.

By Lemma \ref{conti},  we can define 
\begin{align}
\mathcal{N}_h:=\sup_{\bm 0\neq\mathcal{W}_h,\mathcal{U}_h,\mathcal{V}_h\in[\bm{W}_h]^3}
\frac{b_h(\mathcal{W}_h;\mathcal{U}_h,\mathcal{V}_h)}{\|\mathcal{W}_h\|_V \|\mathcal{U}_h\|_V \|\mathcal{V}_h\|_V }.\label{Nh}
\end{align}
Similarly, since $\|\bm{v}_{h}\|_0\lesssim \|\mathcal{V}_h\|_V$, we can  define a norm of $\bm f$ by
\begin{eqnarray}
\|\bm{f}\|_{*,h}:=\color{black}\sup_{\bm 0\neq \mathcal{V}_h\in\bm{W}_h}\frac{(\bm f,\bm v_h)}{\|\mathcal{V}_h\|_V}\label{fh}.
\end{eqnarray}

\begin{theorem} The HDG scheme \eqref{HDG1}-\eqref{HDG3} admits at least one solution   $(\mathbb{L}_h,$
	$\mathcal{U}_h,\mathcal{P}_h)\in \mathbb{K}_h\times[\bm{V}_h\times\widehat{\bm{V}}_h]\times [Q_h\times\widehat{{Q}}_h]$ for a sufficiently small mesh size $h$.
\end{theorem}
\begin{proof} Introduce the  trilinear form
	\begin{align*}
	A_h(\mathcal{W}_h;\mathcal{U}_h,\mathcal{V}_h ):=&\nu(K_h\mathcal{U}_h; K_h\mathcal{V}_h )+s_h(\mathcal{U}_h;\mathcal{V}_h )+b_h(\mathcal{W}_h;\mathcal{U}_h,\mathcal{V}_h).
	\end{align*}
	Then the  following two results hold:
	
	(i) $A_h(\mathcal{V}_h;\mathcal{V}_h;\mathcal{V}_h )=\nu\|\mathcal{V}_h\|_V^2,\quad \forall \mathcal{V}_h\in \bm{W}_h$;
	
	(ii) $\bm{W}_h$ is separable,  and the relation  $\lim_{n\to\infty}\mathcal{U}_h^n=\mathcal{U}_h$ (weakly in $\bm{W}_h$) implies
	$$\lim_{n\to\infty}A_h(\mathcal{U}^n_h;\mathcal{U}^n_h,\mathcal{V}_h )
	=A_h(\mathcal{U}_h;\mathcal{U}_h,\mathcal{V}_h ),\quad
	\forall\mathcal{V}_h\in \bm{W}_h.$$
	
	Notice that (i) is obvious. We only need to show (ii). Let $\{\mathcal{U}_h^n\}_{n=1}^\infty$ be a sequence in $\bm{W}_h$ such that
	\begin{eqnarray}
	\color{black}\mathcal{U}_h^n\to\mathcal{U}_h \text{ weakly in } \bm{W}_h \text{ as } n\to\infty.
	\end{eqnarray}
	Since $\bm{W}_h$ is finite dimensional space,  we know that $\bm{W}_h$ is separable and the weak convergence and strong convergence are equivalent on $\bm{W}_h$. Then we have
	\begin{eqnarray}
	\lim_{n\to\infty}\| \mathcal{U}_{h}^n-\mathcal{U}_{h}\|_V =0.\label{45}
	\end{eqnarray}
	Hence,
	\begin{align}
	&|A_h(\mathcal{U}^n_h;\mathcal{U}^n_h,\mathcal{V}_h )
	-A_h(\mathcal{U}_h;\mathcal{U}_h,\mathcal{V}_h )|\nonumber\\
	=&|\nu(K_h(\mathcal{U}^n_h-\mathcal{U}_h); K_h\mathcal{V}_h )+s_h(\mathcal{U}^n_h-\mathcal{U}_h;\mathcal{V}_h ) +b_h(\mathcal{U}^n_h-\mathcal{U}_h;
	\mathcal{U}^n_h-\mathcal{U}_h,\mathcal{V}_h)\nonumber\\
	&\qquad
	+ b_h(\mathcal{U}_{h};\mathcal{U}^n_h-\mathcal{U}_h,
	\mathcal{V}_h)  +b_h(\mathcal{U}^n_h-\mathcal{U}_h;
	\mathcal{U}_{h},\mathcal{V}_h)|\nonumber\\
	\le& \nu\|\mathcal{U}^n_h-\mathcal{U}_h\|_V
	\|\mathcal{V}_h\|_V+\mathcal{N}_h\|\mathcal{U}^n_h-\mathcal{U}_h\|_V^2
	\|\mathcal{V}_h\|_V
	+2\mathcal{N}_h\|\mathcal{U}^n_h-\mathcal{U}_h\|_V\|\mathcal{U}_h\|_V \|\mathcal{V}_h\|_V.\nonumber
	\end{align}
	which, together with \eqref{45} and the fact that   $\mathcal{N}_h$ can be bounded from above by a positive constant  for a sufficiently small $h$,   yields
	$$\lim_{n\to\infty}A_h(\mathcal{U}^n_h;\mathcal{U}^n_h,\mathcal{V}_h )
	=A_h(\mathcal{U}_h;\mathcal{U}_h,\mathcal{V}_h ),\
	\forall\mathcal{V}_h\in \bm{W}_h,$$
	i.e. (ii) holds.
	
	In light of  \cite[Page 280, Theorem 1.2]{Girault.V;Raviart.P1986}, the results (i)-(ii) imply that there exists at least one $\mathcal{U}_h\in\bm{W}_h$ satisfying
	\begin{align*}
	A_h(\mathcal{U}_h;\mathcal{U}_h,\mathcal{V}_h )&=(\bm{f},\bm{v}_{h}) &\forall\mathcal{V}_h\in \bm{W}_h.
	\end{align*}
	Given such a $\mathcal{U}_h\in \bm{W}_h$, by Theorem \ref{theoremLBB}   there  exists a unique $\mathcal{P}_h\in Q_h\times \widehat{Q}_h$ satisfying
	{
		\begin{align*}
		d_h(\mathcal{V}_{h},\mathcal{P}_h)=-(\bm{f},\bm{v}_{h})+\nu(K_h\mathcal{U}_h; K_h\mathcal{V}_h )
		+s_h(\mathcal{U}_h;\mathcal{V}_h )+b_h(\mathcal{W}_h;\mathcal{U}_h,\mathcal{V}_h).
		\end{align*}
	}
	As a result,  the triple $(\mathbb{L}_h=\nu K_h\mathcal{U}_h,\mathcal{U}_h,\mathcal{P}_h)\in\mathbb{K}_h\times[\bm{V}_h\times\widehat{\bm{V}}_h]\times [Q_h\times \widehat{Q}_h]$ 
	is a solution to the scheme \eqref{HDG1}-\eqref{HDG3}.
\end{proof}

\subsection{Uniqueness result}
\begin{theorem} Let  $(\mathbb{L}_h,\mathcal{U}_h,\mathcal{P}_h)\in\mathbb{K}_h\times[\bm{V}_h\times\widehat{\bm{V}}_h]\times [Q_h\times \widehat{Q}_h]$ be a solution to the problem \eqref{HDG1}-\eqref{HDG3}, then $\mathcal{U}_h$ satisfies
	\begin{eqnarray}
	\|\mathcal{U}_h\|_V \le\nu^{-1}\|\bm{f}\|_{*,h}.\label{sta}
	\end{eqnarray}
	In addition, if
	\begin{eqnarray}\label{513}
	(\mathcal{N}_h/\nu^2)\|\bm{f}\|_{*,h}<1, \label{uni-condi}
	\end{eqnarray}
	then the scheme \eqref{HDG1}-\eqref{HDG3} admits a unique solution. 
\end{theorem}
\begin{proof} 
	Take $\mathcal{V}_h=\mathcal{U}_h,\mathcal{Q}_h=\mathcal{P}_h$ in \eqref{HDG22} and \eqref{HDG33}, and add the two equations to get
	\begin{eqnarray*}
		\nu\|\mathcal{U}_h\|_V ^2=(\bm{f},\bm{u}_{h})\le \|\bm{f}\|_{*,h}\|\mathcal{U}_h\|_V ,
	\end{eqnarray*}
	which leads to \eqref{sta}. 
	
	Since $\mathbb{L}_h=\nu K_h \mathcal{U}_{h}$, we only need to show the uniqueness of $\mathcal{U}_h$ and $\mathcal{P}_h$. Let $(\mathcal{U}_{h}^1,\mathcal{P}_{h}^1)$ and $(\mathcal{U}_{h}^2,\mathcal{P}_{h}^2)$ be two solutions to \eqref{HDG22}-\eqref{HDG33}, then we have, for $i=1,2$ and $(\mathcal{V}_h,\mathcal{Q}_h)\in [\bm{V}_h\times\widehat{\bm{V}}_h]\times [Q_h\times Q_h]$,
	{
		\begin{align*}
		\nu(K_h\mathcal{U}_h^i,K_h\mathcal{V}_h)-d_h(\mathcal{V}_{h},\mathcal{P}^i_h)
		+s_h(\mathcal{U}_h^i;\mathcal{V}_h)+  b_h(\mathcal{U}_h^i;\mathcal{U}^i_h,\mathcal{V}_h)&=(\bm{f},\bm{v}_h),
		\\
		d_h(\mathcal{U}_h^i;\mathcal{Q}_h)&=0.
		\end{align*}
	}
	These two equations imply
	\begin{subequations}
		\begin{align}
		\nu(K_h(\mathcal{U}_h^1-\mathcal{U}_h^2),K_h\mathcal{V}_h)-d_h(\mathcal{V}_{h},\mathcal{P}^1_h-\mathcal{P}^2_h)&\nonumber\\
		+s_h(\mathcal{U}_h^1-\mathcal{U}_h^2;\mathcal{V}_h)&= b_h(\mathcal{U}_h^2;\mathcal{U}^2_h,\mathcal{V}_h)- b_h(\mathcal{U}_h^1;\mathcal{U}^1_h,\mathcal{V}_h),\label{515a}\\
		d_h(\mathcal{U}_h^1-\mathcal{U}_h^2;\mathcal{Q}_h)&=0.
		\end{align}
	\end{subequations}
	Take $\mathcal{V}_h=\mathcal{U}_h^1-\mathcal{U}_h^2,\mathcal{Q}_h=\mathcal{P}^1_h-\mathcal{P}^2_h$ in the above two equations,  and add them together to get
	\begin{align*}
	\nu\|\mathcal{U}_h^1-\mathcal{U}_h^2\|_V^2=&b_h(\mathcal{U}_{h}^2;\mathcal{U}_{h}^2,\mathcal{U}_{h}^1-\mathcal{U}_{h}^2)
	-b_h(\mathcal{U}_{h}^1;\mathcal{U}_{h}^1,\mathcal{U}_{h}^1-\mathcal{U}_{h}^2) \nonumber\\
	=&b_h(\mathcal{U}_{h}^2;\mathcal{U}_{h}^2-\mathcal{U}_{h}^1,\mathcal{U}_{h}^1-\mathcal{U}_{h}^2)
	+b_h(\mathcal{U}_{h}^2-\mathcal{U}_{h}^1;\mathcal{U}_{h}^1,\mathcal{U}_{h}^1-\mathcal{U}_{h}^2) \nonumber\\
	=&b_h(\mathcal{U}_{h}^2-\mathcal{U}_{h}^1;\mathcal{U}_{h}^1,\mathcal{U}_{h}^1-\mathcal{U}_{h}^2)\nonumber\\
	\le&\mathcal{N}_h\| \mathcal{U}_{h}^1-\mathcal{U}_{h}^2\|_V^2\| \mathcal{U}_{h}^1\|_V  \nonumber\\
	\le&\nu^{-1}\mathcal{N}_h\|\bm{f}\|_{*,h}\| \mathcal{U}_{h}^1-\mathcal{U}_{h}^2\|_V^2,
	\end{align*}
	where we have used the stability result \eqref{sta}.
	Thus, it holds
	\begin{eqnarray}
	\nu(1-(\mathcal{N}_h/\nu^2)\|\bm{f}\|_{*,h})\| \mathcal{U}_{h}^1-\mathcal{U}_{h}^2\|_V^2\le 0.
	\end{eqnarray}
	As a result, by \eqref{513}  we get $\mathcal{U}_{h}^1=\mathcal{U}_{h}^2$. And  $\mathcal{P}_{h}^1=\,\mathcal{P}_{h}^2$ follows from   \eqref{515a} and  Theorem \ref{theoremLBB}.
\end{proof}
\subsection{Relationship between  discrete and continuous  conditions of uniqu-\\eness}

In this subsection, we shall show that the uniqueness condition \eqref{uni-condi} for the discrete solution is  consistent with the the continuous condition  \eqref{uni}.

\begin{lemma} For all $\bm{w},\bm{u},\bm{v}\in \bm{W}$, it holds
	\begin{eqnarray}
	|b_h( \mathcal{W}_I;\mathcal{U}_I,\mathcal{V}_I)-b(\bm{w};\bm{u},\bm{v})|\lesssim h^{1-\frac d 6}\|\nabla\bm{u}\|_{0}\|\nabla\bm{w}\|_{0}\|\nabla\bm{v}\|_{0},\label{B5}
	\end{eqnarray}
where $\mathcal W_I$, $\mathcal U_I$ and $\mathcal V_I$ are defined as 
	\begin{align*}
	\mathcal{W}_I:=(\bm{P}^{RT}_k{\bm{w}},\bm{\Pi}^{\partial}_k\bm{w}),\quad
	\mathcal{U}_I:=(\bm{P}^{RT}_k{\bm{u}},\bm{\Pi}^{\partial}_k\bm{u}),\quad
	\mathcal{V}_I:=(\bm{P}^{RT}_k{\bm{v}},\bm{\Pi}^{\partial}_k\bm{v}).
	\end{align*}
\end{lemma}

\begin{proof} Set
	\begin{align*}\footnotesize
	&R_1:= 
	-(\bm{P}^{RT}_k\bm{u}\otimes\bm{P}^{RT}_k\bm{w},\nabla_h\bm{P}^{RT}_k\bm{v})+\langle \bm{\Pi}_{k}^{\partial}\bm{u}\otimes\bm{\Pi}_{k}^{\partial}\bm{w}\bm{n},\bm{P}^{RT}_k\bm{v} \rangle_{\partial\mathcal{T}_h}
	-(\nabla\cdot(\bm{u}\otimes\bm{w}),\bm{v}),\\
	&R_2:=
	(\bm{P}^{RT}_k\bm{v}\otimes\bm{P}^{RT}_k\bm{w},\nabla_h\bm{P}^{RT}_k\bm{u}) -\langle \bm{\Pi}_{k}^{\partial}\bm{v}\otimes\bm{\Pi}_{k}^{\partial}\bm{w}\bm{n},\bm{P}^{RT}_k\bm{u} \rangle_{\partial\mathcal{T}_h}
	+(\nabla\cdot(\bm{v}\otimes\bm{w}),\bm{u}).
	\end{align*}
	Then we have
	\begin{eqnarray}
	b_h( \mathcal{W}_I;\mathcal{U}_I,\mathcal{V}_I)-b(\bm{w};\bm{u},\bm{v})
	=\frac{1}{2}(R_1-R_2).\label{R0}
	\end{eqnarray}
	Now we are going to estimate $R_1$ and $R_2$.
	From  integration by parts \color{black} and the fact $\langle \bm{\Pi}_{k}^{\partial}\bm{u}\otimes\bm{\Pi}_{k}^{\partial}\bm{w}\bm{n},\bm{v} \rangle_{\partial\mathcal{T}_h}=0$ it follows
	\begin{align*}
	R_1=&-((\bm{P}^{RT}_k\bm{u}-\bm{u})\otimes\bm{P}^{RT}_k\bm{w}
	,\nabla_h\bm{P}^{RT}_k\bm{v})-(\bm{u}\otimes(\bm{P}^{RT}_k\bm{w}-\bm{w})
	,\nabla_h\bm{P}^{RT}_k\bm{v})\nonumber\\
	&-(\bm{u}\otimes\bm{w}
	,\nabla_h(\bm{P}^{RT}_k\bm{v}-\bm{v}))+\langle \bm{\Pi}_{k}^{\partial}\bm{u}\otimes\bm{\Pi}_{k}^{\partial}\bm{w}\bm{n},\bm{P}^{RT}_k\bm{v} -\bm{v} \rangle_{\partial\mathcal{T}_h}\\
	=&-((\bm{P}^{RT}_k\bm{u}-\bm{u})\otimes\bm{P}^{RT}_k\bm{w}
	,\nabla_h\bm{P}^{RT}_k\bm{v})-(\bm{u}\otimes(\bm{P}^{RT}_k\bm{w}-\bm{w})
	,\nabla_h\bm{P}^{RT}_k\bm{v})\nonumber\\
	&+(\nabla\cdot(\bm{u}\otimes\bm{w})
	,\bm{P}^{RT}_k\bm{v}-\bm{v})+\langle \bm{\Pi}_{k}^{\partial}\bm{u}\otimes\bm{\Pi}_{k}^{\partial}\bm{w}\bm{n}-\bm{u}\otimes\bm{w}\bm{n},\bm{P}^{RT}_k\bm{v} -\bm{v} \rangle_{\partial\mathcal{T}_h}\nonumber\\
	=&-((\bm{P}^{RT}_k\bm{u}-\bm{u})\otimes\bm{P}^{RT}_k\bm{w}
	,\nabla_h\bm{P}^{RT}_k\bm{v})-(\bm{u}\otimes(\bm{P}^{RT}_k\bm{w}-\bm{w})
	,\nabla_h\bm{P}^{RT}_k\bm{v})\nonumber\\
	&+(\nabla\cdot(\bm{u}\otimes\bm{w})
	,\bm{P}^{RT}_k\bm{v}-\bm{v})+\langle (\bm{\Pi}_{k}^{\partial}\bm{u}-\bm{u})\otimes\bm{\Pi}_{k}^{\partial}\bm{w}\bm{n},\bm{P}^{RT}_k\bm{v} -\bm{v} \rangle_{\partial\mathcal{T}_h}\nonumber\\
	&+\langle \bm{u}\otimes(\bm{\Pi}_{k}^{\partial}\bm{w}-\bm{w})\bm{n},\bm{P}^{RT}_k\bm{v} -\bm{v} \rangle_{\partial\mathcal{T}_h}\nonumber\\
	=:&\sum_{i=1}^5 R_{1,i}.
	\end{align*}
	Thanks to  the   H\"{o}lder's inequality and the approximation properties of $\bm{P}^{RT}_k$,  it holds
	\begin{align*}
	|R_{1,1}|\le& \sum_{T\in\mathcal{T}_h}\|\bm{P}^{RT}_k\bm{u}-\bm{u}\|_{0,3,T}(\|\bm{P}^{RT}_k\bm{w}-\bm{w}\|_{0,6,T}+\|\bm{w}\|_{0,6,T})
	\|\nabla_h\bm{P}^{RT}_k\bm{v}\|_{0,T}\nonumber\\
	\le&
	\left(\|\bm{P}^{RT}_k\bm{w}-\bm{w}\|_{0,6}+\|\bm{w}\|_{0,6}\right)
	\sum_{T\in\mathcal{T}_h}\|\bm{P}^{RT}_k\bm{u}-\bm{u}\|_{0,3,T}
	\|\nabla_h\bm{P}^{RT}_k\bm{v}\|_{0,T}\nonumber\\
	\lesssim&(h^{1-d/3}+1)\|\nabla\bm{w}\|_{0}\sum_{T\in\mathcal{T}_h}h_T^{1-\frac d 6}\|\nabla\bm{u}\|_{0,T}\|\nabla\bm{v}\|_{0,T}\nonumber\\
	\lesssim&h^{1-\frac d 6}\|\nabla\bm{u}\|_{0}\|\nabla\bm{w}\|_{0}\|\nabla\bm{v}\|_{0},\\
	|R_{1,2}|\le&\sum_{T\in\mathcal{T}_h}\|\bm{u}\|_{0,6,T}\|\bm{P}^{RT}_k\bm{w}-\bm{w}\|_{0,3,T}
	\|\nabla_h\bm{P}^{RT}_k\bm{v}\|_{0,T}\nonumber\\
	\lesssim&h^{1-\frac d 6}\|\nabla\bm{u}\|_{0}\|\nabla\bm{w}\|_{0}\|\nabla\bm{v}\|_{0},\\
	|R_{1,3}|\le&\sum_{T\in\mathcal{T}_h}\|\nabla\cdot\bm{w}\|_{0,T}
	\|\bm{u}\|_{0,6,T}\|\bm{P}^{RT}_k\bm{v}-\bm{v}\|_{0,3,T}\nonumber\\
	&+\sum_{T\in\mathcal{T}_h}\|\bm{w}\|_{0,6,T}
	\|\nabla\bm{u}\|_{0,T}\|\bm{P}^{RT}_k\bm{v}-\bm{v}\|_{0,3,T}\nonumber\\
	\lesssim&h^{1-\frac d 6}\|\nabla\bm{u}\|_{0}\|\nabla\bm{w}\|_{0}\|\nabla\bm{v}\|_{0},\\
	|R_{1,4}|\le&\sum_{T\in\mathcal{T}_h}
	\|\bm{\Pi}_{k}^{\partial}\bm{u}-\bm{u}\|_{0,\partial T}
	\|\bm{\Pi}_{k}^{\partial}\bm{w}\|_{0,6,\partial T}
	\|\bm{P}^{RT}_k\bm{v} -\bm{v}\|_{0,3,\partial T}\nonumber\\
	\lesssim&\sum_{T\in\mathcal{T}_h}
	h_T^{\frac 1 2}\|\nabla\bm{u}\|_{0,T}
	h_T^{-\frac 1 6}( \|\bm{w}\|_{1,T}+\|\bm{w}\|_{0,6,T}   )
	h_T^{2/3-d/6}\|\nabla\bm{v}\|_{0,T}\nonumber\\
	\le&( \|\bm{w}\|_{1}+\|\bm{w}\|_{0,6}   )
	\sum_{T\in\mathcal{T}_h}
	h_T^{1-\frac d 6}\|\nabla\bm{u}\|_{0,T}
	\|\nabla\bm{v}\|_{0,T}\nonumber\\
	\lesssim&h^{1-\frac d 6}\|\nabla\bm{u}\|_{0}\|\nabla\bm{w}\|_{0}\|\nabla\bm{v}\|_{0},
	\end{align*}
	and 
	\begin{align*}
	|R_{1,5}|\le&\sum_{T\in\mathcal{T}_h}
	\|\bm{u}\|_{0,6,\partial T}
	\|\bm{\Pi}_{k}^{\partial}\bm{w}-\bm{w}\|_{0,\partial T}
	\|\bm{P}^{RT}_k\bm{v} -\bm{v}\|_{0,3,\partial T}\nonumber\\
	\lesssim&\sum_{T\in\mathcal{T}_h}
	h_T^{-\frac 1 6}( \|\bm{u}\|_{1,T}+\|\bm{u}\|_{0,6,T}   )
	h_T^{\frac 1 2}\|\nabla\bm{w}\|_{0,T}
	h_T^{2/3-d/6}\|\nabla\bm{v}\|_{0,T}\nonumber\\
	\le&( \|\bm{u}\|_{1}+\|\bm{u}\|_{0,6}   )
	\sum_{T\in\mathcal{T}_h}
	h_T^{1-\frac d 6}\|\nabla\bm{w}\|_{0,T}
	\|\nabla\bm{v}\|_{0,T}\nonumber\\
	\lesssim&h^{1-\frac d 6}\|\nabla\bm{u}\|_{0}\|\nabla\bm{w}\|_{0}\|\nabla\bm{v}\|_{0}.
	\end{align*}
	As a result, we have
	\begin{eqnarray}
		|R_1|\lesssim h^{1-\frac d 6}\|\nabla\bm{u}\|_{0}\|\nabla\bm{w}\|_{0}\|\nabla\bm{v}\|_{0}.\label{R1}
	\end{eqnarray}
	Similarly, we can obtain
	\begin{eqnarray*}
		|R_2|\lesssim h^{1-\frac d 6}\|\nabla\bm{u}\|_{0}\|\nabla\bm{w}\|_{0}\|\nabla\bm{v}\|_{0},\label{R2}
	\end{eqnarray*}
	which, together with \eqref{R0} and \eqref{R1},  gives the desired result.
\end{proof}

In view of the definitions \eqref{26}, \eqref{27}, \eqref{Nh}, and \eqref{fh}, in what follows  we shall show the relationships between $\|\bm{f}\|_{*,h}$ and $\|\bm{f}\|_{*}$ and  between $\mathcal{N}_h$ and $\mathcal{N}$. To this end,   for any $\lambda\in (1, 2]$,
let $\bm{g}\in [W^{0,\lambda}(\Omega)]^d$ be a given function,  and let $(\bm{\Phi},\Psi)\in {\bm V}\times Q$ be the solution of the auxiliary Stokes problem
\begin{subequations}
	\begin{align}
	-\Delta\bm{\Phi}+\nabla\Psi&=\bm{g},\\
	\nabla\cdot\bm{\Phi}&=0
	\end{align}
\end{subequations}
with the   regularity assumption 
\begin{eqnarray}
\|\bm{\Phi}\|_{1+s,\lambda}+\|\Psi\|_{s,\lambda}\lesssim \|\bm{g}\|_{0,\lambda}
\label{regular}
\end{eqnarray}
 for some $s\in (\frac{1}{2},1]$. 

\begin{remark} As shown in  \cite{Girault.V;Raviart.P1986,regular1,regular2},
	when $\Omega$ is a bounded convex polygon in $\mathbb{R}^2$ or convex polyhedron in $\mathbb{R}^3$, the regularity estimate \eqref{regular} holds true with $s=1$.
\end{remark}

%
%

\begin{lemma}\label{limit} It holds that
	\begin{align*}
	\lim_{h\to 0}\|\bm{f}\|_{*,h}&=\|\bm{f}\|_{*},\qquad
	\lim_{h\to 0}\mathcal{N}_h=\mathcal{N}.
	\end{align*}
\end{lemma}

\begin{proof} We divide the proof into three steps.
	
	Step 1). Let $\bm{\Pi}:\bm{V}_h\times\widehat{\bm{V}}_h\to\bm{W}$ be a switch operator from the discrete space to the continuous space defined as follows: for any $\mathcal{V}_h=( \bm{v}_h,\widehat{\bm{v}}\color{black}_h\color{black})\in\color{black}\bm W_h$, $\bm{\Pi}\mathcal{V}_h\in\bm{W}$ is determined by
	\begin{eqnarray}\label{pi}
	(\nabla \bm{\Pi}\mathcal{V}_h,\nabla \bm{w})=(K_h\mathcal{V}_h,\nabla \bm{w}),\qquad\forall \bm{w}\in\bm{W}.
	\end{eqnarray}
	Here we recall that the operator $K_h$ is given in \eqref{K_h}. 
	By testing \eqref{pi} with $\bm{w}=\bm{\Pi}\mathcal{V}_h\in \bm{W}$ we have
	\begin{eqnarray}
	\|\nabla \bm{\Pi}\mathcal{V}_h\|_0\le \|K_h\mathcal{V}_h\|_0\leq \|\mathcal{V}_h\|_V.\label{bound1}
	\end{eqnarray}
	 Let $\mu\in [2,\infty)$ be the conjugate number of $\lambda$ with
	\begin{eqnarray}
	\frac{1}{\mu}+\frac{1}{\lambda}=1.\label{528}
	\end{eqnarray}
We consider the following problem: find $(\bm{\Phi},\Psi)\in \color{black}\bm{V}\times Q$ such that
	\begin{subequations}
		\begin{align}
		-\Delta\bm{\Phi}+\nabla \Psi&= (\bm{\Pi}\mathcal{V}_h-\bm{v}_{h})^{\mu-1},\label{526}\\
	\color{black}
		\nabla\cdot\bm{\Phi}&=0.
		\end{align}
	\end{subequations}
	Then from the regularity \eqref{regular} it follows
	\begin{eqnarray}
	\|\Phi\|_{1+s,\lambda}\lesssim \|(\bm{\Pi}\mathcal{V}_h-\bm{v}_{h})^{\mu-1}\|_{0,\lambda}
	=\|\bm{\Pi}\mathcal{V}_h-\bm{v}_{h}\|^{\mu-1}_{0,\mu}.\label{530}
	\end{eqnarray}
	We use the equation \eqref{526}, \color{black} integration by parts \color{black},
	the fact
	$\langle \bm{n}\cdot\nabla\bm{\Phi},\widehat{\bm{v}}_{h}\rangle_{\partial\mathcal{T}_h}=0$,
	the definitions of $ \bm{\Pi}_m^o$, \color{black}
	  the equality \eqref{pi},
	and the  definition of $K_h$
	to get
	\begin{align*}
	\|\bm{\Pi}\mathcal{V}_h-\bm{v}_{h}\|^{\mu}_{0,\mu}=&(-\Delta\bm{\Phi}\color{black}+\nabla\Psi\color{black},\bm{\Pi}\mathcal{V}_h-\bm{v}_{h})\nonumber\\
	=&(\nabla\bm{\Phi},\nabla\bm{\Pi}\mathcal{V}_h)
	-(\nabla\bm{\Phi},\nabla_h\bm{v}_{h})+\langle \bm{n}\cdot\nabla\bm{\Phi},\bm{v}_{h}-\widehat{\bm{v}}_{h}\rangle_{\partial\mathcal{T}_h}\nonumber\\
	=&(\nabla\bm{\Phi},\nabla\bm{\Pi}\mathcal{V}_h)
	-(\bm{\Pi}_{m}^o\nabla\bm{\Phi},\nabla_h\bm{v}_{h})
	+\langle \bm{n}\cdot\nabla\bm{\Phi},\bm{v}_{h}-\widehat{\bm{v}}_{h}\rangle_{\partial\mathcal{T}_h}\nonumber\\
	=&(\nabla\bm{\Phi},\nabla\bm{\Pi}\mathcal{V}_h-K_h\mathcal{V}_h)+\langle \bm{n}\cdot(\nabla\bm{\Phi}-\bm{\Pi}_{m}^o\nabla\bm{\Phi}),\bm{v}_{h}-\widehat{\bm{v}}_{h}\rangle_{\partial\mathcal{T}_h}\nonumber\\
	=&\langle \bm{n}\cdot(\nabla\bm{\Phi}-\bm{\Pi}_{m}^o\nabla\bm{\Phi}),\bm{v}_{h}-\widehat{\bm{v}}_{h}\rangle_{\partial\mathcal{T}_h}.
	\end{align*}
	In light of the \color{black} H\"{o}lder's inequality and the approximation properties  of interpolations we have
	\begin{align*}
	\|\bm{\Pi}\mathcal{V}_h-\bm{v}_{h}\|^{\mu}_{0,\mu}
	\lesssim&\sum_{T\in\mathcal{T}_h}\| \bm{n}\cdot(\nabla\bm{\Phi}-\bm{\Pi}_{k}^o\nabla\bm{\Phi})\|_{0,\lambda,\partial T}
	\|\bm{v}_{h}-\widehat{\bm{v}}_{h}\|_{0,\mu,\partial T}
	\nonumber\\
	\lesssim&\sum_{E\in\mathcal{E}_h}h_T^{s-\frac{1}{\lambda}}\|\bm{\Phi}\|_{1+s,\lambda,T}
	h_T^{-(d-1)(\frac{1}{2}-\frac{1}{\mu})}\|\bm{v}_{h}-\widehat{\bm{v}}_{h}\|_{0,\partial T}
	\nonumber\\
	\lesssim& h^{s-d(\frac{1}{2}-\frac{1}{\mu})}\|\bm{\Pi}\mathcal{V}_h-\bm{v}_{h}\|^{\mu-1}_{0,\mu}
	\|\tau^{\frac 1 2}(\bm{v}_h-\widehat{\bm{v}}_h )\|_{\partial \mathcal{T}_h},
	\end{align*}
	which leads to
	\begin{align}
	\|\bm{\Pi}\mathcal{V}_h-\bm{v}_{h}\|_{0,\mu}
	\lesssim h^{s-d(\frac{1}{2}-\frac{1}{\mu})}\|\tau^{\frac 1 2}(\bm{v}_{h}-\widehat{\bm{v}}_{h})\|_{\partial \mathcal{T}_h}.\label{est1}
	\end{align}
	With this estimate and \eqref{bound1} it is easy to get
	\begin{align}
	\|\bm{f}\|_{*,h}=&\sup_{\bm{v}_h\in\bm{V}_h\times\widehat{\bm{V}}_h}
	\frac{(\bm{f},\bm{v}_{h})}{\|\mathcal{V}_h\|_V}\nonumber\\
	=&\sup_{\bm 0\neq\bm{v}_h\in\bm{V}_h\times\widehat{\bm{V}}_h}\frac{(\bm{f},\bm{\Pi}\mathcal{V}_h)+(\bm{f},\bm{v}_{h}-\bm{\Pi}\mathcal{V}_h)}
	{\|\mathcal{V}_h\|_V}\nonumber\\
	\le&\sup_{\bm 0\neq\bm{v}_h\in\bm{V}_h\times\widehat{\bm{V}}_h}\frac{(\bm{f},\bm{\Pi}\mathcal{V}_h)+Ch^{s}\|\bm{f}\|_0\|\tau^{\frac 1 2}(\bm{v}_{h}-\widehat{\bm{v}}_{h})\|_{\partial \mathcal{T}_h}}
	{\|\mathcal{V}_h\|_V}
	\nonumber\\
	\le&\sup_{\bm 0\neq\bm{v}_h\in\bm{V}_h\times\widehat{\bm{V}}_h}\frac{(\bm{f},\bm{\Pi}\mathcal{V}_h)}
	{\|\nabla\bm{\Pi}\mathcal{V}_h\|_0}+Ch^{s}\|\bm{f}\|_0\nonumber\\
	\le&\sup_{\bm 0\neq\bm{v}\in\bm{W}}\frac{(\bm{f},\bm{v})}
	{\|\nabla\bm{v}\|_0}+Ch^{s}\|\bm{f}\|_0\nonumber\\
	=&\|\bm{f}\|_{*}+Ch^{s}\|\bm{f}\|_0.\label{fh1}
	\end{align}
	By the approximation properties of interpolations $\bm{P}^{RT}_k,\bm{\Pi}_k^{\partial} $ and \eqref{bound1} we obtain
	\begin{eqnarray}
	\|(\bm{\Pi}\mathcal{V}_h)_I\|_V
	\lesssim\|\mathcal{V}_h\|_V,\label{535}
	\end{eqnarray}
	where
	$$(\bm{\Pi}\mathcal{V}_h)_I:=(\bm{P}^{RT}_k\bm{\Pi}\mathcal{V}_h,\bm{\Pi}_k^{\partial}\bm{\Pi}\mathcal{V}_h  ) .$$
	Using \eqref{est1} and the approximation property of $\bm{P}_k^{RT}$ yields
	\begin{align}
	\|\bm{v}_{h}-\bm{P}^{RT}_k\bm{\Pi}\mathcal{V}_h\|_{0,3}&\le
	\|\bm{v}_{h}-\bm{\Pi}\mathcal{V}_h\|_{0,3}
	+\|\bm{\Pi}\mathcal{V}_h-\bm{P}^{RT}_k\bm{\Pi}\mathcal{V}_h\|_{0,3}\nonumber\\
	&\lesssim h^{s-d/6}(\|\tau^{\frac 1 2}(\bm{v}_{h}-\widehat{\bm{v}}_{h})\|_{\partial \mathcal{T}_h}
	+\|\nabla\bm{\Pi}\mathcal{V}_h\|_0)\nonumber\\
	&\lesssim h^{s-d/6}\|\mathcal{V}_h\|_V.\label{536}
	\end{align}
	Then, from \eqref{B2}, \eqref{B3}, \eqref{B4}, \eqref{535}, and \eqref{536} it follows
	\begin{eqnarray*}
		&&b_h(\mathcal{W}_h;\mathcal{U}_h,\mathcal{V}_h)-b_h((\bm{\Pi}\mathcal{W}_h)_I;(\bm{\Pi}\mathcal{U}_h)_I,(\bm{\Pi}\mathcal{V}_h)_I))\nonumber\\
		&&\quad\quad=
		b_h(\mathcal{W}_h-(\bm{\Pi}\mathcal{W}_h)_I;\mathcal{U}_h,\mathcal{V}_h)\nonumber\\
		&&\quad\quad\quad+b_h((\bm{\Pi}\mathcal{W}_h)_I;\mathcal{U}_h-(\bm{\Pi}\mathcal{U}_h)_I,\mathcal{V}_h)\nonumber\\
		&&\quad\quad\quad+b_h((\bm{\Pi}\mathcal{W}_h)_I;(\bm{\Pi}\mathcal{U}_h)_I,\mathcal{V}_h-(\bm{\Pi}\mathcal{V}_h)_I)\nonumber\\
		&&\quad\quad\lesssim h^{s-d/6}\|\mathcal{W}_h\|_V \cdot\|\mathcal{U}_h\|_V \cdot\|\mathcal{V}_h\|_V.
	\end{eqnarray*}
	In view of \eqref{B5} and \eqref{bound1}, we have
	\begin{eqnarray*}
		&&b_h((\bm{\Pi}\mathcal{W}_h)_I;(\bm{\Pi}\mathcal{U}_h)_I,(\bm{\Pi}\mathcal{V}_h)_I)
		-b(\bm{\Pi}\mathcal{W}_h;\bm{\Pi}\mathcal{U}_h,\bm{\Pi}\mathcal{V}_h)\nonumber\\
		&&\quad\quad\lesssim
		h^{1-\frac d 6}\|\nabla\bm{\Pi}\mathcal{W}_h\|_0
		\|\nabla\bm{\Pi}\mathcal{U}_h\|_0
		\|\nabla\bm{\Pi}\mathcal{V}_h\|_0\nonumber\\
		&&\quad\quad\le
		h^{1-\frac d 6}\|\mathcal{W}_h\|_V \cdot\|\mathcal{U}_h\|_V \cdot\| \bm{v}_h\|_V.
	\end{eqnarray*}
	Therefore,
	\begin{align}
	b_h(\mathcal{W}_h;\mathcal{U}_h,\mathcal{V}_h)
	=&
	b(\bm{\Pi}\mathcal{W}_h;\bm{\Pi}\mathcal{U}_h,\bm{\Pi}\mathcal{V}_h)\nonumber\\
	&
	+(b_h(\mathcal{W}_h;\mathcal{U}_h,\mathcal{V}_h)-b_h((\bm{\Pi}\mathcal{W}_h)_I;(\bm{\Pi}\mathcal{U}_h)_I,(\bm{\Pi}\mathcal{V}_h)_I)) \nonumber\\
	&+(b_h((\bm{\Pi}\mathcal{W}_h)_I;(\bm{\Pi}\mathcal{U}_h)_I,(\bm{\Pi}\mathcal{V}_h)_I)-b(\bm{\Pi}\mathcal{W}_h;\bm{\Pi}\mathcal{U}_h,\bm{\Pi}\mathcal{V}_h)) \nonumber\\
	\le&b(\bm{\Pi}\mathcal{W}_h;\bm{\Pi}\mathcal{U}_h,\bm{\Pi}\mathcal{V}_h)+C h^{s-d/6}\|\mathcal{W}_h\|_V \cdot\|\mathcal{U}_h\|_V \cdot\|\mathcal{V}_h\|_V.\nonumber
	\end{align}
	Hence, we get
	\begin{align}
	\mathcal{N}_h=&\sup_{\bm 0\neq\bm{w}_h,\bm{u}_h,\bm{v}_h\in\bm{W}_h}\frac{b_h(\mathcal{W}_h;\mathcal{U}_h,\mathcal{V}_h)}{\|\mathcal{W}_h\|_V \cdot\|\mathcal{U}_h\|_V \cdot\|\mathcal{V}_h\|_V}
	\nonumber\\
	\le&\sup_{\bm 0\neq\bm{w}_h,\bm{u}_h,\bm{v}_h\in\bm{W}_h}\frac{b(\bm{\Pi}\mathcal{W}_h;\bm{\Pi}\mathcal{U}_h,\bm{\Pi}\mathcal{V}_h)
		+Ch^{1-\frac d 6}\|\mathcal{W}_h\|_V \cdot\|\mathcal{U}_h\|_V \cdot\|\mathcal{V}_h\|_V}
	{\|\mathcal{W}_h\|_V \cdot\|\mathcal{U}_h\|_V \cdot\|\mathcal{V}_h\|_V}
	\nonumber\\
	\le&\sup_{\bm 0\neq\bm{w}_h,\bm{u}_h,\bm{v}_h\in\bm{W}_h}\frac{b(\bm{\Pi}\mathcal{W}_h;\bm{\Pi}\mathcal{U}_h,\bm{\Pi}\mathcal{V}_h)
	}
	{\|\nabla\bm{\Pi}\mathcal{W}_h\|_0\|\nabla\bm{\Pi}\mathcal{U}_h\|_0\|\nabla\bm{\Pi}\mathcal{V}_h\|_0}+Ch^{s-d/6}
	\nonumber\\
	\le&\sup_{\bm 0\neq\bm{w},\bm{u},\bm{v}\in\bm{W}}\frac{b(\bm{w};\bm{u},\bm{v})
	}
	{\|\nabla\bm{w}\|_0\|\nabla\bm{u}\|_0\|\nabla\bm{v}\|_0}+Ch^{s-d/6}
	\nonumber\\
	=&\mathcal{N}+Ch^{s-d/6}.\label{Nh1}
	\end{align}

	Step 2). Let $\bm{\Pi}_h:\bm{W}\to \bm{W}_h$
	 be the switch operator defined as follows: for any $\bm{v}\in\bm{W}$,   $\bm{\Pi}_h\bm{v}:=(\bm{\pi}_{h}\bm{v},\widehat{\bm{\pi}}_{h}\bm{v})\in\color{black}
	\bm{W}_h$ is determined by
	\begin{align*}
	\nu(K_h\bm{\Pi}_h\bm{v}; K_h\mathcal{W}_h )+s_h(\bm{\Pi}_h\bm{v};\mathcal{W}_h )&=\color{black}\nu\color{black}(\nabla\bm{v},K_h\mathcal{W}_h) &\forall \mathcal{W}_h\in\color{black}\bm{W}_h.
	\end{align*}
	Taking $\mathcal{W}_h=\bm{\Pi}_h\bm{v}$ in this equation, we have
	\begin{eqnarray}
	\|\bm{\Pi}_h\bm{v}\|_V\le \|\nabla\bm{v}\|_0.\label{bound2}
	\end{eqnarray}
	Let us consider the following problem: find $(\bm{\Phi},\Psi)\in \color{black}\bm{V}\times Q$ such that
	\begin{subequations}
		\begin{align}
		-\Delta\bm{\Phi}+\nabla\Psi&=(\bm{v}-\bm{\pi}_{h}\bm{v})^{\mu-1}, \label{543}\\
		\nabla\cdot\bm{\Phi}&=0.
		\end{align}
	\end{subequations}
	Then we apply the \color{black} regularity \color{black} \eqref{regular} to get
	\begin{eqnarray}
	\|\color{black}\bm\Phi\color{black}\|_{2,\lambda}\lesssim \|(\bm{v}-\bm{\pi}_{h}\bm{v})^{\mu-1}\|_{0,\lambda}
	=\|\bm{v}-\bm{\pi}_{h}\bm{v}\|^{\mu-1}_{0,\mu}.
	\end{eqnarray}
	From \eqref{543}, \color{black} integration by parts\color{black},   the \color{black} Holder's inequality, the approximation properties of interpolations it follows
	\begin{align}
	\|\bm{v}-\bm{\pi}_{h}\bm{v}\|^{\mu}_{0,\mu}=&(-\Delta\bm{\Phi}+\color{black}\nabla\Psi,\bm{v}-\bm{\pi}_{h}\bm{v})\nonumber\\
	=&(\nabla\bm{\Phi},\nabla\bm{v})
	-(\nabla\bm{\Phi},\nabla_h\bm{\pi}_{h}\bm{v})+\langle \bm{n}\cdot\nabla\bm{\Phi},\bm{\pi}_{h}\bm{v}-\widehat{\bm{\pi}}_{h}\bm{v}\rangle_{\partial\mathcal{T}_h}\nonumber\\
	=&(\nabla\bm{\Phi},\nabla\bm{v})
	-(\bm{\Pi}_{m}^o\nabla\bm{\Phi},\nabla_h\bm{\pi}_{h}\bm{v})
	+\langle \bm{n}\cdot\nabla\bm{\Phi},\bm{\pi}_{h}\bm{v}-\widehat{\bm{\pi}}_{h}\bm{v}\rangle_{\partial\mathcal{T}_h}\nonumber\\
	=&(K_h(\bm{\Phi})_I,\nabla\bm{v}-K_h(\bm{\pi}_{h}\bm{v},\widehat{\bm{\pi}}_{h}\bm{v}))
	+(\nabla\bm{\Phi}-K_h(\bm{\Phi})_I,\nabla\bm{v})\nonumber\\
	&+\langle \bm{n}\cdot(\nabla\bm{\Phi}-\bm{\Pi}_{m}^o\nabla\bm{\Phi}),\bm{\pi}_{h}\bm{v}-\widehat{\bm{\pi}}_{h}\bm{v}\rangle_{\partial\mathcal{T}_h}\nonumber\\
	=&((\nabla\bm{\Phi}-\bm{\Pi}_{m}^o\nabla\bm{\Phi}),\nabla\bm{v})
	\nonumber\\
	&+\langle \bm{n}\cdot(\nabla\bm{\Phi}-\bm{\Pi}_{m}^o\nabla\bm{\Phi}),\bm{\pi}_{h}\bm{v}-\widehat{\bm{\pi}}_{h}\bm{v}\rangle_{\partial\mathcal{T}_h}\nonumber\\
	&-\langle \tau(\bm{\Pi}_{k}^o\bm{\Phi}-\bm{\Pi}_k^{\partial}\bm{\Phi}),\bm{\pi}_{h}\bm{v}-\widehat{\bm{\pi}}_{h}\bm{v}\rangle_{\partial\mathcal{T}_h}\nonumber\\
	\le&\sum_{T\in\mathcal{T}_h}\|\nabla\bm{\Phi}-\bm{\Pi}_{m}^o\nabla\bm{\Phi}\|_{0,T}\|\nabla\bm{v}\|_{0,T}
	\nonumber\\
	&+\sum_{T\in\mathcal{T}_h}\|\nabla\bm{\Phi}-\bm{\Pi}_{m}^o\nabla\bm{\Phi}\|_{0,\lambda,\partial T}
	\|\bm{\pi}_{h}\bm{v}-\widehat{\bm{\pi}}_{h}\bm{v}\|_{0,\mu,\partial T}\nonumber\\
	&+\sum_{T\in\mathcal{T}_h}h_T^{-1}\|\bm{\Pi}_{k}^o\bm{\Phi}-\bm{\Phi}\|_{0,\lambda,\partial T}\|\bm{\pi}_{h}\bm{v}-\widehat{\bm{\pi}}_{h}\bm{v}\|_{0,\mu,\partial T}\nonumber\\
	\lesssim&h^{s+d(\frac{1}{2}-\frac{1}{\mu})}|\bm{\Phi}|_{1+s,\lambda}(\|\nabla\bm{v}\|_0
	+\|\tau^{\frac 1 2}(\bm{\pi}_{h}\bm{v}-\widehat{\bm{\pi}}_{h}\bm{v})\|_{\partial \mathcal{T}_h})\nonumber\\
	\lesssim&h^{s+d(\frac{1}{2}-\frac{1}{\mu})}\|\bm{v}-\bm{\Pi}_{h}\bm{v}\|^{\mu-1}_{0,\mu}\|\nabla\bm{v}\|_0, \nonumber
	\end{align}
	which leads to
	\begin{eqnarray}
	\|\bm{v}-\bm{\pi}_{h}\bm{v}\|_{0,\mu}\lesssim h^{s+d(\frac{1}{2}-\frac{1}{\mu})}\|\nabla\bm{v}\|_0.\label{est2}
	\end{eqnarray}
	By \eqref{est2} and \eqref{bound2} we have
	\begin{align}
	\|\bm{f}\|_{*}=&\sup_{\bm 0\neq\bm{v}\in\bm{W}}
	\frac{(\bm{f},\bm{v})}{\|\nabla\bm{v}\|_0}\nonumber\\
	=&\sup_{\bm 0\neq\bm{v}\in\bm{W}}\frac{(\bm{f},\bm{\pi}_{h}\bm{v})+(\bm{f},\bm{v}-\bm{\pi}_{h}\bm{v})}
	{\|\nabla\bm{v}\|_0}\nonumber\\
	\le&\sup_{\bm 0\neq\bm{v}\in\bm{W}}\frac{(\bm{f},\bm{\pi}_{h}\bm{v})+Ch\|\bm{f}\|_0\|\nabla\bm{v}\|_0}
	{\|\nabla\bm{v}\|_0}
	\nonumber\\
	\le&\sup_{\bm 0\neq\bm{v}\in\bm{W}}\frac{(\bm{f},\bm{\pi}_{h}\bm{v})}
	{\|  \bm{\Pi}_{h}\bm{v}\|_{V}  }+Ch\|\bm{f}\|_0\nonumber\\
	\le&\sup_{\bm 0\neq\bm{v}_h\in\color{black}\bm W_h}\frac{(\bm{f},\bm{v}_{h})}
	{\|\mathcal{V}_h\|_V}+Ch\|\bm{f}\|_0\nonumber\\
	\leq&\|\bm{f}\|_{*,h}+Ch\|\bm{f}\|_0.\label{fh2}
	\end{align}
	Using the approximation properties of interpolations and \eqref{bound1},  we get
	\begin{eqnarray}
	\|\bm{\Pi}_h\bm{v}\|_V\lesssim\|\nabla\bm{v}\|_0.\label{548}
	\end{eqnarray}
	In light of \eqref{est2} and the approximation property of $\bm{P}_k^{RT}$, it holds
	\begin{align}
	\|\bm{\pi}_{h}\bm{v}-\bm{P}^{RT}_k\bm{v}\|_{0,3}&\le
	\|\bm{\pi}_{h}\bm{v}-\bm{v}\|_{0,3}
	+\|\bm{v}-\bm{P}^{RT}_k\bm{v}\|_{0,3}\lesssim h^{s-d/6}\|\nabla\bm{v}\|_0.\label{549}
	\end{align}
	Then, from \eqref{B2}, \eqref{B3}, \eqref{B4}, \eqref{548}, and \eqref{549} it follows
	\begin{align*}
	&|b_h(\bm{\Pi}_h\bm{w};\bm{\Pi}_h\bm{u},\bm{\Pi}_h\bm{v})-b_h(\mathcal{W}_I;\mathcal{U}_I,\mathcal{V}_I)|\le
	|b_h(\bm{\Pi}_h\bm{w}-\mathcal{W}_I;\bm{\Pi}_h\bm{u},\bm{\Pi}_h\bm{v})|\nonumber\\
	&\quad\quad\quad+|b_h(\mathcal{W}_I;\bm{\Pi}_h\bm{u}-\mathcal{U}_I,\bm{\Pi}_h\bm{v})|
	+|b_h(\mathcal{W}_I;\mathcal{U}_I,\bm{\Pi}_h\bm{v}-\mathcal{V}_I)|\nonumber\\
	&\quad\quad\lesssim h^{s-d/6}\|\nabla\bm{w}\|_0\|\nabla\bm{u}\|_0\|\nabla\bm{v}\|_0.
	\end{align*}
	Note that \eqref{B5} means
	\begin{eqnarray*}
		|b_h(\mathcal{W}_I;\mathcal{U}_I,\mathcal{V}_I)-b(\bm{w};\bm{u},\bm{v})|\le
		h^{1-\frac d 6}\|\nabla\bm{w}\|_0\|\nabla\bm{u}\|_0\|\nabla\bm{v}\|_0.
	\end{eqnarray*}
	Therefore,
	\begin{align*}
	b(\bm{w};\bm{u},\bm{v})
	=&b_h(\bm{\Pi}_h\bm{w};\bm{\Pi}_h\bm{u},\bm{\Pi}_h\bm{v})
	\nonumber\\
	&-(b_h(\bm{\Pi}_h\bm{w};\bm{\Pi}_h\bm{u},\bm{\Pi}_h\bm{v})-b_h(\mathcal{W}_I;\mathcal{U}_I,\mathcal{V}_I)) \nonumber\\
	&-(b_h(\mathcal{W}_I;\mathcal{U}_I,\mathcal{V}_I)-b(\bm{w};\bm{u},\bm{v})) \nonumber\\
	\le&b_h(\bm{\Pi}_h\bm{w};\bm{\Pi}_h\bm{u},\bm{\Pi}_h\bm{v})+C h^{s-d/6}\|\nabla\bm{w}\|_0\|\nabla\bm{u}\|_0\|\nabla\bm{v}\|_0,
	\end{align*}
	which implies
	\begin{align}
	\mathcal{N}=&\sup_{\bm 0\neq\bm{w},\bm{u},\bm{v}\in\bm{W}}\frac{b(\bm{w};\bm{u},\bm{v})
	}
	{\|\nabla\bm{w}\|_0\|\nabla\bm{u}\|_0\|\nabla\bm{v}\|_0}\nonumber\\
	\le&\sup_{\bm 0\neq \bm{w},\bm{u},\bm{v}\in\bm{W}}\frac{b_h(\bm{\Pi}\mathcal{W}_h;\bm{\Pi}\mathcal{U}_h,\bm{\Pi}\mathcal{V}_h)
		+Ch^{1-\frac d 6}\|\nabla\bm{w}\|_0\|\nabla\bm{u}\|_0\|\nabla\bm{v}\|_0}
	{\|\nabla\bm{w}\|_0\|\nabla\bm{u}\|_0\|\nabla\bm{v}\|_0}
	\nonumber\\
	\le&\sup_{\bm 0\neq\bm{w},\bm{u},\bm{v}\in\bm{W}}
	\frac{b_h(\bm{\Pi}_h\bm{w};\bm{\Pi}_h\bm{u},\bm{\Pi}_h\bm{v})
	}
	{\| \bm{\Pi}_h\bm{w}\|_V \cdot\| \bm{\Pi}_h\bm{u}\|_V \cdot\| \bm{\Pi}_h\bm{v}\|_V }+Ch^{1-\frac d 6}
	\nonumber\\
	\le&\sup_{\bm 0\neq\bm{w}_h,\bm{u}_h,\bm{v}_h\in\bm{W}_h}\frac{b_h(\mathcal{W}_h;\mathcal{U}_h,\mathcal{V}_h)}{\|\mathcal{W}_h\|_V \cdot\|\mathcal{U}_h\|_V \cdot\|\mathcal{V}_h\|_V}
	+Ch^{1-\frac d 6}
	\nonumber\\
	=&\mathcal{N}_h+Ch^{s-d/6}.\label{Nh2}
	\end{align}

	Step 3).
	By \eqref{fh1}, \eqref{fh2}, \eqref{Nh1}, and \eqref{Nh2} we have
	\begin{align*}
	\|\bm{f}\|_{*}-Ch^s\|\bm{f}\|_0&\le \|\bm{f}\|_{*,h}\le \|\bm{f}\|_{*}+Ch^s\|\bm{f}\|_0,\\
	\mathcal{N}-Ch^{s-d/6}&\le \mathcal{N}_h\le\mathcal{N}+Ch^{s-d/6},
	\end{align*}
	where we recall that $s\in(\frac{1}{2},1]$ and $\color{black} d=2,3$. As a result, the desired results follow from the squeeze theorem immediately.

\end{proof}

\begin{remark} In view of of Lemma \ref{limit} and the condition \eqref{uni},  it  holds
	\begin{eqnarray}
	(\mathcal{N}_h/\nu^2)\|\bm{f}\|_{*,h}\le 1-\frac{\delta}{2}\label{uni-condition}
	\end{eqnarray}
	when the mesh size $h$ is sufficiently small.
\end{remark}

\section{A priori error estimates}
\label{section4}

%

By some simple calculations, we have the following lemma.
\begin{lemma} For any $\bm{u},\bm{w}\in \bm{W}$ and $\bm{v}_h\in\bm{V}_h$, it holds
	\begin{eqnarray}
	b_h(\mathcal{W}_I;\mathcal{U}_I,\mathcal{V}_h)
	=(\nabla\cdot(\bm{u}\otimes\bm{w}),\bm{v}_{h})+E_{N}(\bm{w};\bm{u},\mathcal{V}_h),\label{EN1}
	\end{eqnarray}
	where
	\begin{align}
	E_{N}(\bm{w};\bm{u},\mathcal{V}_h)=& -\frac{1}{2}(\bm{P}_k^{RT}\bm{u}\otimes \bm{P}_k^{RT}\bm{w}-\bm{u}\otimes\bm{w},\nabla_h\bm{v}_{h})\nonumber\\
	&+\frac{1}{2}\langle\bm{\Pi}_{k}^{\partial}\bm{u}\otimes \bm{\Pi}_{k}^{\partial}\bm{w}\bm{n}-\bm{u}\otimes\bm{w}\bm{n},\bm{v}_{h}\rangle_{\partial\mathcal{T}_h}\nonumber\\
	&-\frac{1}{2}(\bm{w}\cdot\nabla\bm{u}-\bm{P}_k^{RT}\bm{w}\cdot\nabla_h\bm{P}_k^{RT}\bm{u},\bm{v}_{h})\nonumber\\
	&-\frac{1}{2}\langle\widehat{\bm{v}}_{h}\otimes \bm{\Pi}_{k}^{\partial}\bm{w}\bm{n},\bm{P}_k^{RT}\bm{u}\rangle_{\partial\mathcal{T}_h}.\label{def-EN}
	\end{align}
\end{lemma}
\begin{proof} By integration by parts, it arrives  at
	\begin{eqnarray}
	&&-(\bm{P}_k^{RT}\bm{u}\otimes \bm{P}_k^{RT}\bm{w},\nabla_h\bm{v}_{h})+\langle\bm{\Pi}_{k}^{\partial}\bm{u}\otimes \bm{\Pi}_{k}^{\partial}\bm{w}\bm{n},\bm{v}_{h}\rangle_{\partial\mathcal{T}_h}\nonumber\\
	&&\qquad=(\nabla\cdot(\bm{u}\otimes \bm{w}),\bm{v}_{h})
	-(\bm{P}_k^{RT}\bm{u}\otimes \bm{P}_k^{RT}\bm{w}-\bm{u}\otimes\bm{w},\nabla_h\bm{v}_{h})\nonumber\\
	&&\qquad\quad+\langle\bm{\Pi}_{k}^{\partial}\bm{u}\otimes \bm{\Pi}_{k}^{\partial}\bm{w}\bm{n}-\bm{u}\otimes\bm{w}\bm{n},\bm{v}_{h}\rangle_{\partial\mathcal{T}_h}.\nonumber
	\end{eqnarray}
	From integration by parts and the fact that $\nabla\cdot\bm{w}=0$, it follows
	\begin{eqnarray*}
		&&-(\bm{v}_{h}\otimes \bm{P}_k^{RT}\bm{w},\nabla_h\bm{P}_k^{RT}\bm{u})+\langle\widehat{\bm{v}}_{h}\otimes \bm{\Pi}_{k}^{\partial}\bm{w}\bm{n},\bm{P}_k^{RT}\bm{u}\rangle_{\partial\mathcal{T}_h}\nonumber\\
		&&\qquad=-(\nabla\cdot(\bm{u}\otimes \bm{w}),\bm{v}_{h})+(\bm{v}_{h}\otimes\bm{w},\nabla\bm{u})-(\bm{v}_{h}\otimes \bm{P}_k^{RT}\bm{w},\nabla_h\bm{P}_k^{RT}\bm{u})\nonumber\\
		&&\quad\qquad+\langle\widehat{\bm{v}}_{h}\otimes \bm{\Pi}_{k}^{\partial}\bm{w}\bm{n},\bm{P}_k^{RT}\bm{u}\rangle_{\partial\mathcal{T}_h}\nonumber\\
		&&\qquad=-(\nabla\cdot(\bm{u}\otimes \bm{w}),\bm{v}_{h})
		+(\bm{w}\cdot\nabla\bm{u}-\bm{P}_k^{RT}\bm{w}\cdot\nabla_h\bm{P}_k^{RT}\bm{u},\bm{v}_{h})\nonumber\\
		&&\quad\qquad+\langle\widehat{\bm{v}}_{h}\otimes \bm{\Pi}_{k}^{\partial}\bm{w}\bm{n},\bm{P}_k^{RT}\bm{u}\rangle_{\partial\mathcal{T}_h}.
	\end{eqnarray*}
	\color{black}
	Then the desired results follows from the definition of $b_h$.
	\color{black}
\end{proof}

\begin{lemma}\label{lemma4.1} Let $(\mathbb L,\bm{u}, p) $
 be the solution to \eqref{mixed}. Then, for all $(\mathcal{V}_h,\mathcal{Q}_h)\in [\bm{V}_h\times\widehat{\bm{V}}_h]\times [Q_h\times\widehat{Q}_h]$, it holds the   equations
	\begin{subequations}
		\begin{align}
		\bm{\Pi}_m^o\mathbb{L}-\nu K_h\mathcal{U}_I&=0,\label{error1}\\
		\nu(K_h\mathcal{U}_I,K_h\mathcal{V}_h)-d_h(\mathcal{V}_{h},\mathcal{P}_I)
		&\nonumber\\
		+s_h(\mathcal{U}_I,\mathcal{V}_h)
		+ b_h(\mathcal{U}_I;\mathcal{U}_I,\mathcal{V}_h)&=(\bm{f},\bm{v}_h)+E_L(\bm{u};\mathcal{V}_h) +E_{N}(\bm{u},\bm{u};\mathcal{V}_h)\label{error2}\\
		d_h(\mathcal{U}_I,\mathcal{Q}_h)&=0,\label{error3}
		\end{align}
	\end{subequations}
	where $E_N$ is defined in \eqref{def-EN}, and $E_L$ is defined as
	\begin{align}
	E_L(\bm{u};\mathcal{V}_h):= -\nu\langle \bm{v}_h-\widehat{\bm{v}}_{h},(\bm{\Pi}_m^o-\mathbb{Id} )(\nabla\bm u)\bm{n} \rangle_{\partial\mathcal{T}_h} -\nu\langle \tau(\bm{P}^{RT}_k\bm{u}-{\bm{u}}),\bm{v}_{h}-\widehat{\bm{v}}_{h} \rangle_{\partial\mathcal{T}_h}.\nonumber
	\end{align}
	In addition, it holds
	\begin{eqnarray}
	\bm{P}^{RT}_k\bm{u}|_T\in [\mathcal{P}_{k}(T)]^d,\forall T\in\mathcal{T}_h. \label{4.5}
	\end{eqnarray}
\end{lemma}
\begin{proof} For any $T\in\mathcal{T}_h,\tau_k\in \mathcal{P}_{k}(T)$, by \color{black}the property   \color{black}(\ref{RT4}) we have
	\begin{eqnarray*}
		(\nabla\cdot\bm{P}^{RT}_{k}\bm{u},\tau_k)_T=(\nabla\cdot\bm{u},\tau_k)_T=0,
	\end{eqnarray*}
	which implies that $\nabla\cdot\bm{P}^{RT}_k\bm{u}=0$.  Then the result $(\ref{4.5})$ follows from   Lemma \ref{lemma3.1}.
	
	By the orthogonality of projections, integration by parts,
	and the fact $\mathbb L=\nabla\bm u$ we get
	\color{black}
	\begin{eqnarray}
	&&a_h(\bm{\Pi}_m^o \mathbb{L},\mathbb{G}_h)+ c_h(\mathcal{U}_I,\mathbb{G}_h)\nonumber\\
	&&\quad\quad=\nu^{-1}(\bm{\Pi}_m^o\mathbb L,\mathbb{G}_h)+(\bm{P}^{RT}_k\bm{u},\nabla_h\cdot\mathbb{G}_h)
	-\langle\bm{\Pi}_k^{\partial}{\bm{u}},\mathbb{G}_h\bm{n} \rangle_{\partial\mathcal{T}_h} \nonumber\\
	&&\quad\quad=\nu^{-1}(\mathbb L,\mathbb{G}_h)+(\bm{u},\nabla_h\cdot\mathbb{G}_h)
	-\langle{\bm{u}},\mathbb{G}_h\bm{n} \rangle_{\partial\mathcal{T}_h}  \nonumber\\
	&&\quad\quad=0.\nonumber
	\end{eqnarray}
	\color{black}
	Similarly, by 
	the fact $ -\nabla\cdot \mathbb{L}+\nabla p=\bm f$ we obtain
	\color{black}
	\begin{eqnarray}
	&&c_h(\mathcal{V}_h,\bm{\Pi}_m^o \mathbb{L})
	+d_h(\mathcal{V}_{h},\mathcal{P}_I)-s_h( \mathcal{U}_I,\mathcal{V}_h  ) \nonumber\\
	&&\quad\quad=(\bm{v}_h,\nabla_h\cdot\bm{\Pi}_m^o \mathbb{L})
	-\langle\widehat{\bm{v}}_h,\bm{\Pi}_m^o \mathbb{L}\bm{n} \rangle_{\partial\mathcal{T}_h}\nonumber\\
	&&\qquad\quad+(\nabla_h\cdot\bm{v}_h,\Pi^o_{k-1}p)+\langle \bm{v}_h\bm{n},\Pi_k^{\partial}p \rangle_{\partial\mathcal{T}_h} \nonumber\\
	&&\qquad\quad-\nu\langle \tau(\bm{\Pi}_k^{\partial} \bm{P}^{RT}_k\bm{u}-\bm{\Pi}_k^{\partial}{\bm{u}}),\bm{\Pi}_k^{\partial}\bm{v}_{h}-\widehat{\bm{v}}_{h} \rangle_{\partial\mathcal{T}_h} \nonumber\\
	&&\quad\quad=(\bm{v}_h,\nabla\cdot\mathbb{L})
	-\langle\widehat{\bm{v}}_h-\bm{v}_h,(\bm{\Pi}_m^o\mathbb{L}-\mathbb{L} )\bm{n} \rangle_{\partial\mathcal{T}_h}-(\bm{v}_h,\nabla p)\nonumber\\
	&&\qquad\quad-\nu\langle \tau(\bm{P}^{RT}_k\bm{u}-{\bm{u}}),\bm{\Pi}_k^{\partial}\bm{v}_{h}-\widehat{\bm{v}}_{h} \rangle_{\partial\mathcal{T}_h}\nonumber\\
	&&\quad\quad=-(\bm{f},\bm{v}_h)+(\nabla\cdot(\bm{u}\otimes\bm{u}),\bm{v}_h   )+E_L(\bm{u};\mathcal{V}_h    )\nonumber\\
	&&\quad\quad=-(\bm{f},\bm{v}_h)+E_N(\bm{u};\bm{u},\bm{v}_h   )+E_L(\bm{u};\mathcal{V}_h    )- b_h(\mathcal{U}_I;\mathcal{U}_I,\mathcal{V}_h).\nonumber
	\end{eqnarray}
	Thanks to $\nabla\cdot\bm{P}^{RT}_{k}\bm{u}=0$ and the property   \eqref{RT1},  \color{black} we have
	\begin{eqnarray}
	d_h(\mathcal{U}_I,\mathcal{Q}_h  ) =(\nabla\cdot\bm{P}^{RT}_{k}\bm{u},q_{h})-\langle\bm{P}^{RT}_{k}\bm{u}\cdot\bm{n},\widehat{q}_h \rangle_{\partial\mathcal{T}_h}
	=0. \nonumber
	\end{eqnarray}
	\color{black}
	This completes the proof.
	\color{black}
\end{proof}

\color{black}
Let us estimate the nonlinear error $E_{N}$ and linear error $E_L$ in the following lemma.\color{black}

\begin{lemma} Let $\mathbb L=\nu\nabla \bm u$,  $\bm{u}\in [H^{r_{\bm u}}(\Omega)]^d\cap \bm{W}$, $\bm{w}\in [H^{r_{\bm w}}(\Omega)]^d\cap \bm{W}$, 
	\color{black}
	with $r_{\bm u},r_{\bm w}>\frac{3}{2}$,
	\color{black}
	and $\mathcal{V}_h\in\bm{V}_h\times\widehat{\bm{V}}_h$, then the following estimates holds
	\begin{align}
	|E_{N}(\bm{u},\bm{w};\mathcal{V}_h)|&\lesssim  \left(
	h^{1+s_{\bm{w}}-\frac d 2}|\bm{u}|_{1}
	\|\bm{w}\|_{s_w}
	+h^{1+s_{\bm u}-\frac d 2}|\bm{w}|_{1}
	|\bm{u}|_{s_{\bm u}}\right)
	\|\mathcal{V}_h\|_V,\nonumber\\
	|E_L(\bm{u};\mathcal{V}_h)|&\lesssim \nu h^{s_{\bm u}-1}\|\bm{u}\|_{s_{\bm u}}\|\mathcal{V}_h\|_V,\nonumber
	\end{align}
	where 
$ s_{\bm{u}}:=\min\{r_{\bm u},k+1\}$ and $ s_{\bm w}:=\min\{r_{\bm w},k+1\}. 
$
\end{lemma}

\begin{proof} We only prove the first estimate, since the second one follows similarly. 
	
	Set $E_{N}(\bm{u};\bm{u},\bm{v}_h)=\sum_{i=1}^3E_i$, where
	\begin{align*}
	E_1:=& -\frac{1}{2}(\bm{P}_k^{RT}\bm{u}\otimes \bm{P}_k^{RT}\bm{w}-\bm{u}\otimes\bm{w},\nabla_h\bm{v}_{h}),\\
	E_2:=&\frac{1}{2}\langle\bm{\Pi}_{k}^{\partial}\bm{u}\otimes \bm{\Pi}_{k}^{\partial}\bm{w}\bm{n}-\bm{u}\otimes\bm{w}\bm{n},\bm{v}_{h}\rangle_{\partial\mathcal{T}_h},\\
	E_3:=&-\frac{1}{2}(\bm{w}\cdot\nabla\bm{u}-\bm{P}_k^{RT}\bm{w}\cdot\nabla_h\bm{P}_k^{RT}\bm{u},\bm{v}_{h}) -\frac{1}{2}\langle\widehat{\bm{v}}_{h}\otimes \bm{\Pi}_{k}^{\partial}\bm{w}\bm{n},\bm{P}_k^{RT}\bm{u}\rangle_{\partial\mathcal{T}_h}.
	\end{align*}
	\color{black}
	The orthogonality \eqref{RT2} yields
	\begin{align}
	|(\bm{P}_k^{RT}\bm{u}\otimes (\bm{P}_k^{RT}\bm{w}-\bm{w}),\nabla_h\bm{v}_{h})|
	&=|((\bm{P}_k^{RT}\bm{u}-\bm{\Pi}_0^o\bm u)\otimes (\bm{P}_k^{RT}\bm{w}-\bm{w}),\nabla_h\bm{v}_{h})|\nonumber\\
	&\lesssim h^{1+s_w-\frac d 2} |\bm u|_{1}\|\bm w\|_{s_w}
	\|\mathcal{V}_h\|_V.\nonumber
	\end{align} 
	Similarity, we have
	\begin{align}
	|((\bm{P}_k^{RT}\bm{u}-\bm{u})\otimes \bm{w},\nabla_h\bm{v}_{h})|
	&=|((\bm{P}_k^{RT}\bm{u}-\bm u)\otimes (\bm{P}_k^{RT}\bm{w}-\bm{\Pi}_0^o\bm{w}),\nabla_h\bm{v}_{h})|\nonumber\\
	&\lesssim h^{1+s_u-\frac d 2} |\bm w|_{1}\|\bm u\|_{s_u}
	\|\mathcal{V}_h\|_V.\nonumber
	\end{align} 
	Therefore, from the triangle inequality  it follows
	\begin{align}
	|E_1| 
	\lesssim& \left(
	h^{1+s_{\bm{w}}-\frac d 2}|\bm{u}|_{1}
	\|\bm{w}\|_{s_{\bm{w}}}
	+h^{1+s_{\bm{u}}-\frac d 2}|\bm{w}|_{1}
	|\bm{u}|_{s_{\bm{u}}}\right)
	\|\mathcal{V}_h\|_V.\nonumber
	\end{align}
	Since
	\begin{align}
	&|\langle\bm{\Pi}_{k}^{\partial}\bm{u}\otimes (\bm{\Pi}_{k}^{\partial}\bm{w}-\bm{w})\bm{n},\bm{v}_{h}-\widehat{\bm{v}}_{h}\rangle_{\partial\mathcal{T}_h}|\nonumber\\
	&\qquad=|\langle(\bm{\Pi}_{k}^{\partial}\bm{u}-\bm{\Pi}_{0}^{\partial}\bm{u})\otimes (\bm{\Pi}_{k}^{\partial}\bm{w}-\bm{w})\bm{n},\bm{v}_{h}-\widehat{\bm{v}}_{h}\rangle_{\partial\mathcal{T}_h}|\nonumber\\
	&\qquad\lesssim h^{1+s_w-\frac d 2}|\bm u|_1\|\bm w\|_{s_w}\|\mathcal{V}\|_V\nonumber
	\end{align}
	and
	\begin{align}
	|\langle(\bm{\Pi}_{k}^{\partial}\bm{u}-\bm{u})\otimes \bm{w}\bm{n},\bm{v}_{h}-\widehat{\bm{v}}_{h}\rangle_{\partial\mathcal{T}_h}| \lesssim
	h^{1+s_u-\frac d 2}|\bm w|_1\|\bm u\|_{s_u}\|\mathcal{V}\|_V,\nonumber
	\end{align}
	by triangle inequality we obtain
	\begin{align}
	|E_2|
	\lesssim& \left(
	h^{1+s_{\bm{w}}-\frac d 2}|\bm{u}|_{1}
	\|\bm{w}\|_{s_{\bm{w}}}
	+h^{1+s_{\bm{u}}-\frac d 2}|\bm{w}|_{1}
	|\bm{u}|_{s_{\bm{u}},T}\right)
	\|\mathcal{V}_h\|_V.\nonumber
	\end{align}
	It is easy to see that
	\begin{align}              
	-2E_3=&((\bm{u}-\bm{P}_k^{RT}\bm{u})\cdot\nabla\bm{w},\bm{v}_{h})
	+(\bm{P}_k^{RT}\bm{u}\cdot\nabla_h(\bm{w}-\bm{P}_k^{RT}\bm{w}),\bm{v}_{h})\nonumber\\
	&+\langle\widehat{\bm{v}}_{h}\otimes \bm{\Pi}_{k}^{\partial}\bm{w}\bm{n},\bm{P}_k^{RT}\bm{u}-\bm{u}\rangle_{\partial\mathcal{T}_h}\nonumber\\
	=&((\bm{u}-\bm{P}_k^{RT}\bm{u})\cdot\nabla\bm{w},\bm{v}_{h})+((\bm{P}_k^{RT}\bm{u}-\bm{u})\cdot\nabla_h(\bm{w}-\bm{P}_k^{RT}\bm{w}),\bm{v}_{h})\nonumber\\
	&+(\bm{u}\cdot\nabla_h(\bm{w}-\bm{P}_k^{RT}\bm{w}),\bm{v}_{h})+\langle(\widehat{\bm{v}}_{h}-\bm{v}_{h})\otimes \bm{\Pi}_{k}^{\partial}\bm{w}\bm{n},\bm{P}_k^{RT}\bm{u}-\bm{u}\rangle_{\partial\mathcal{T}_h}\nonumber\\
	&+\langle\bm{v}_{h}\otimes \bm{\Pi}_{k}^{\partial}\bm{w}\bm{n},\bm{P}_k^{RT}\bm{u}-\bm{u}\rangle_{\partial\mathcal{T}_h}\nonumber\\
	=&-((\bm{w}-\bm{P}_k^{RT}\bm{w})\cdot\nabla_h\bm{v}_{h},\bm{u}-\bm{\Pi}_{0}^o\bm{u})\nonumber\\
	&+\langle(\bm{w}-\bm{P}_k^{RT}\bm{w})\cdot\bm{n},(\bm{v}_h-\widehat{\bm{v}}_h)(\bm{u}-\bm{\Pi}^o_0\bm{u}) \rangle_{\partial\mathcal{T}_h} \nonumber\\
	&+((\bm{P}_k^{RT}\bm{w}-\bm{w})\cdot\nabla_h(\bm{u}-\bm{P}_k^{RT}\bm{u}),\bm{v}_{h})\nonumber\\
	&-((\bm{w}-\bm{\Pi}_{0}^o\bm{w})\cdot\nabla_h\bm{v}_{h},\bm{u}-\bm{P}_k^{RT}\bm{u})\nonumber\\
	&+\langle(\widehat{\bm{v}}_{h}-\bm{v}_{h})\otimes \bm{\Pi}_{k}^{\partial}\bm{w}\bm{n},\bm{P}_k^{RT}\bm{u}-\bm{u}\rangle_{\partial\mathcal{T}_h}\nonumber\\
	&+\langle\bm{v}_{h}\otimes (\bm{\Pi}_{k}^{\partial}\bm{w}-\bm{w})\bm{n},\bm{P}_k^{RT}\bm{u}-\bm{u}\rangle_{\partial\mathcal{T}_h}.\nonumber
	\end{align}
	So
	\begin{eqnarray*}
		|E_3|\lesssim \left(
		h^{1+s_{\bm{w}}-\frac d 2}|\bm{u}|_{1}
		\|\bm{w}\|_{s_{\bm{w}}}
		+h^{1+s_{\bm{u}}-\frac d 2}|\bm{w}|_{1}
		|\bm{u}|_{s_{\bm{u}},T}\right)
		\|\mathcal{V}_h\|_V.\nonumber
	\end{eqnarray*}
	Combining the above estimates of $E_1, E_2,E_3$ yields the desired conclusion.
\end{proof}

\begin{theorem} \label{theorem4.4} 
	Let $(\bm{u},p)\in [H^{r_{\bm u}}(\Omega)]^d\times H^{r_p}(\Omega)$, with $r_{\bm u}>\frac{3}{2}$ and $r_p>\frac{1}{2}$, and $(\mathbb{L}_h,\bm{u}_h,\widehat{\bm{u}}_h,p_h,\widehat{p}_h)\in \mathbb{K}_h\color{black}\times[\color{black}\bm{V}_h\times\bm{V}_h]\times[ Q^0_h\times \widehat{Q}_h]$  be the solutions to $(\ref{Or-NS})$ and  $(\ref{HDG11})$-\eqref{HDG33}, respectively. Then it holds the following  error estimates:
	\begin{align}
	\| \mathcal{U}_I-\mathcal{U}_h\|_V \lesssim&  \delta^{-1}{\nu}^{-1}h^{s_{\bm{u}}-1}(h^{s_{\bm{u}}+1-\frac d 2}\|\bm{u}\|_{s_{\bm{u}}}|\bm{u}|_{1}+\nu \|\bm{u}\|_{s_{\bm{u}}}),\label{78}\\
	\|J_hp-p_h\|_Q\lesssim&  \delta^{-1}{\nu}^{-1}h^{s_{\bm{u}}-1}(h^{s_{\bm{u}}+1-\frac d 2}\|\bm{u}\|_{s_{\bm{u}}}|\bm{u}|_{1}+\nu \|\bm{u}\|_{s_{\bm{u}}}),\label{79}
	\end{align}
	where $s_{\bm{u}}=\min\{r_{\bm u},k+1\}$.
\end{theorem}
\begin{proof} Set 
	\begin{align*}
	&e_h^{\mathbb L}:=\bm{\Pi}_m^o\mathbb L-\mathbb L_h,\\ &e_h^{\bm{u}}:=\bm{P}^{RT}_k\bm{u}-\bm{u}_h,\quad e_h^{\widehat{\bm{u}}}:=\bm{\Pi}_k^{\partial}\bm{u}-\widehat{\bm{u}}_h,
	\quad \mathcal{E}_h^{\bm{u}}:=( e_h^{\bm{u}},e_h^{\widehat{\bm{u}}}  ),\\
	& e_h^{p}:=\Pi^o_{k-1}p-p_h, \quad e_h^{\widehat{p}}:=\Pi_k^{\partial}p-p_h,\quad \mathcal{E}_h^{p}:=( e_h^{p},e_h^{\widehat{p}}  ).
	\end{align*}
\color{black}  Subtracting \eqref{HDG11}-\eqref{HDG33} from \eqref{error1}-\eqref{error3}, respectively,   we get
	\begin{subequations}\label{ee}
		\begin{align}
		e_h^{\mathbb L}-\nu K_h\mathcal{E}_h^{\bm{u}} &=0,\\
		\nu(K_h\mathcal{E}_h^{\bm{u}} ,K_h\mathcal{V}_h)
		+s_h( \mathcal{E}_h^{\bm{u}};\mathcal{V}_h)-d_h(\mathcal{V}_{h},  \color{black} \mathcal{E}_h^p \color{black})&\nonumber\\
		+ b_h(\mathcal{U}_I;\mathcal{U}_I,\mathcal{V}_h)-b_h(\mathcal{U}_h;\mathcal{U}_h,\mathcal{V}_h)&=E_L(\bm{u};\mathcal{V}_h) +E_{N}(\bm{u},\bm{u},\mathcal{V}_h)
		\\
		d_h( \mathcal{E}_h^{\bm{u}},\mathcal{Q}_h)&=0.
		\end{align}
	\end{subequations}
	It is easy to verify that $(\mathcal{U}_I-\mathcal{U}_h,\mathcal{P}_I-\mathcal{P}_h)\in [\bm{V}_h\times\widehat{\bm{V}}_h]\times [Q^0_h\times \widehat{Q}_h]$. Then we take $(\mathcal{V}_h,\mathcal{Q}_h)=(\mathcal{U}_I-\mathcal{U}_h,\mathcal{P}_I-\mathcal{P}_h)$ in  \eqref{ee} to get
	\begin{align*}
	\nu\|\mathcal{E}_h^{\bm{u}} \|_V^2=&E_L(\bm{u}; \mathcal{E}_h^{\bm{u}})
	+E_{N}(\bm{u},\bm{u}; \mathcal{E}_h^{\bm{u}}) -\left\{b_h( \mathcal{U}_I;\mathcal{U}_I, \mathcal{E}_h^{\bm{u}}  )
	-b_h( \mathcal{U}_h;\mathcal{U}_h, \mathcal{E}_h^{\bm{u}} )\right\}.
	\end{align*}
	By the triangle inequity we have
	\begin{align*}
	&b_h( \mathcal{U}_I;\mathcal{U}_I,\mathcal{E}_h^{\bm{u}}    )
	-b_h( \mathcal{U}_h;\mathcal{U}_h,\mathcal{E}_h^{\bm{u}}   )=b_h( \mathcal{U}_I; \mathcal{E}_h^{\bm{u}},\mathcal{E}_h^{\bm{u}})+b_h(\mathcal{E}_h^{\bm{u}};\mathcal{U}_h,\mathcal{E}_h^{\bm{u}})\nonumber\\
	&\qquad
	\le \mathcal{N}_h\|\mathcal{U}_h \|_V\|\mathcal{E}_h^{\bm{u}} \|_V^2
	\le \mathcal{N}_h\nu^{-1}\|\bm{f}\|_{*,h}\|\mathcal{E}_h^{\bm{u}} \|_V^2.
	\end{align*}
	Then from \eqref{uni-condition}, we get
	\begin{align*}
	\frac{\delta\nu}{2}\|\mathcal{E}_h^{\bm{u}} \|_V^2
	&\le (\nu-\mathcal{N}_h\nu^{-1}\|\bm{f}\|_{*,h})\|\mathcal{E}_h^{\bm{u}} \|_V^2\nonumber\\
	&\le E_L(\bm{u};\mathcal{E}_h^{\bm{u}})+E_{N}(\bm{u};\bm{u},\mathcal{E}_h^{\bm{u}})\nonumber\\
	&\lesssim 
	(h^{s_{\bm{u}}+1-\frac d 2}\|\bm{u}\|_{s_{\bm{u}}}|\bm{u}|_{1}+\nu h^{s_{\bm{u}}-1}\|\bm{u}\|_{s_{\bm{u}}})\|\mathcal{E}_h^{\bm{u}}\|_V,
	\end{align*}
	which leads to
	\begin{align*}
	\|\mathcal{E}_h^{\bm{u}} \|_V \lesssim \delta^{-1}{\nu}^{-1}h^{s_{\bm{u}}-1}(h^{s_{\bm{u}}+1-\frac d 2}\|\bm{u}\|_{s_{\bm{u}}}|\bm{u}|_{1}+\nu \|\bm{u}\|_{s_{\bm{u}}}).
	\end{align*}
	This yields \eqref{78}.
	
	From \eqref{HDG22} it follows
	\begin{align*}
	d_h(\mathcal{V}_{h}, \mathcal{E}^p_h)=&
	(K_h\mathcal{E}_h^{\bm{u}} ,K_h\mathcal{V}_h)
	+s_h( \mathcal{E}_h^{\bm{u}};\mathcal{V}_h)+b_h(\mathcal{U}_I;\mathcal{U}_I,\mathcal{V}_h)\nonumber\\
	&-b_h(\mathcal{U}_h;\mathcal{U}_h,\mathcal{V}_h)
	-E_L(\bm{u};\mathcal{V}_h)
	-E_{N}(\bm{u},\bm{u};\mathcal{V}_h)
	\end{align*}
	Thanks to Theorem \ref{theoremLBB}, it holds
	\begin{align*}
	\|\mathcal{E}^p_h\|_Q\lesssim\sup_{\bm 0\neq \mathcal{V}_h\in \bm{V}_h\times\widehat{\bm{V}}_h}\frac{d_h(\mathcal{V}_{h}, \mathcal{E}^p_h)}{\|\mathcal{V}_h\|_V}
	\lesssim\delta^{-1}{\nu}^{-1}h^{s_{\bm{u}}-1}(h^{s_{\bm{u}}+1-\frac d 2}\|\bm{u}\|_{s_{\bm{u}}}|\bm{u}|_{1}+\nu \|\bm{u}\|_{s_{\bm{u}}}).
	\end{align*}
	This yields \eqref{79}.

\end{proof}

Furthermore,  Theorem \ref{theorem4.4} leads to  the following a priori error estimates.
\begin{theorem} \label{th44} Under the same conditions of  Theorem \ref{theorem4.4}, 
 it holds the  error estimates
	\begin{align}
	\|\nabla\bm{u}-\nabla_h\bm{u}_{h}\|_0&\lesssim  \delta^{-1}{\nu}^{-1}h^{s_{\bm{u}}-1}(h^{s_{\bm{u}}+1-\frac d 2}\|\bm{u}\|_{s_{\bm{u}}}|\bm{u}|_{1}+\nu \|\bm{u}\|_{s_{\bm{u}}})
	+h^{s_{\bm{u}}-1}\|\bm{u}\|_{s_{\bm{u}}}\label{HU}
	,\\
	\nu^{-1}\|\mathbb L-\mathbb L_{h}\|_0&\lesssim\delta^{-1}{\nu}^{-1}h^{s_{\bm{u}}-1}(h^{s_{\bm{u}}+1-\frac d 2}\|\bm{u}\|_{s_{\bm{u}}}|\bm{u}|_{1}+\nu \|\bm{u}\|_{s_{\bm{u}}})
	+h^{s_{\bm{u}}-1}\|\bm{u}\|_{s_{\bm{u}}},\label{LI}\\
	\|p-p_{h}\|_0&\lesssim \delta^{-1}{\nu}^{-1}h^{s_{\bm{u}}-1}(h^{s_{\bm{u}}+1-\frac d 2}\|\bm{u}\|_{s_{\bm{u}}}|\bm{u}|_{1}+\nu \|\bm{u}\|_{s_{\bm{u}}})
	+h^{s_{p}}\|p\|_{s_{p}}, \label{LP}
	\end{align}
	where $\mathbb L=\nu\nabla\bm u$ and $s_p:=\min\{r_p,k\}$.
\end{theorem}
\begin{proof} Using  \eqref{sim}  and  Theorem \ref{theorem4.4}, we have
	\begin{eqnarray*}
		\|\nabla_h (\bm{P}^{RT}_k\bm{u}-\bm{u}_{h})\|_0\lesssim \| \mathcal{U}_I-\mathcal{U}_h\|_V.
	\end{eqnarray*}
	Then, from the triangle inequality it follows
	\begin{align}
	\|\nabla \bm{u}-\nabla_h\bm{u}_{h}\|_0\lesssim& \|\nabla \bm{u}-\nabla_h\bm{P}^{RT}_k\bm{u}\|_0+ \|\nabla_h (\bm{P}^{RT}_k\bm{u}-\bm{u}_{h})\|_0\nonumber\\
	\lesssim& \delta^{-1}{\nu}^{-1}h^{s_{\bm{u}}-1}(h^{s_{\bm{u}}+1-\frac d 2}\|\bm{u}\|_{s_{\bm{u}}}|\bm{u}|_{1}+\nu \|\bm{u}\|_{s_{\bm{u}}})
	+h^{s_{\bm{u}}-1}\|\bm{u}\|_{s_{\bm{u}}},\nonumber
	\end{align}
	i.e. \eqref{HU} holds.
	
	The estimate \eqref{LI} follows from  the estimate \eqref{79}   and the definition of the norm $\|\cdot\|_V$. 
	Similarly, Theorem $\ref{theorem4.4}$ and a triangle inequality give \eqref{LP}.
\end{proof}
\begin{remark}
	From \eqref{HU} we see   that the velocity error is independent of the pressure. This means that our HDG scheme is  pressure-robust.
\end{remark}

\section{$L^2$ error estimation for velocity by dual arguments}

\label{section5}

%

\color{black}
We follow standard dual arguments to derive an $L^2$ error estimate for velocity.  \color{black} To this end, introduce the following dual problem:  find   $(\bm{\Phi},\Psi)$ satisfying
\begin{eqnarray}
\left\{
\begin{aligned}
-\nu\Delta\bm{\Phi}-\bm{u}\cdot\nabla\bm{\Phi}+(\nabla\bm{u })^T\bm{\Phi}+\nabla \Psi&=e_h^{\bm{u}} \text{ in }\ \Omega, \\
\nabla\cdot \bm{\Phi}&=0\ \text{ in } \ \Omega,  \\
\bm{\Phi}&=\bm{0}\ \text{ on } \ \partial \Omega, \\
\end{aligned}
\right.
\end{eqnarray}
\color{black}
where $\bm{u}$ and $\bm{u}_{h}$ are the solutions of (\ref{Or-NS}) and (\ref{HDG1}), respectively.
\color{black}
We assume the following \color{black}regularity \color{black} holds:
\begin{eqnarray}\label{sta-assum}
\|\bm{\Phi}\|_{1+\alpha}+\|\Psi\|_{\alpha}\lesssim\|e_h^{\bm{u}}\|_0\color{black},\color{black}
\end{eqnarray}
with $\alpha\in (\frac{1}{2},1]$  depending on $\Omega$. 

\begin{lemma} For any $\bm{\Phi}$, $\bm{u}\in \bm{V}$ and $\mathcal{V}_h\in\bm{V}_h\times\widehat{\bm{V}}_h$ we have
	\begin{eqnarray}
	b_h(\mathcal{V}_h;\mathcal{U}_I,\bm{\Phi}_I)=((\nabla\bm{u})^T \bm{\Phi},\mathcal{V}_{h})+E_{N}^{\color{black}\star}(\mathcal{V}_h;\bm{u},\bm{\Phi}),
	\label{EN2}
	\end{eqnarray}
	where
	\begin{align*}
	E_{N}^{\color{black}\star}(\mathcal{V}_h;\bm{u},\bm{\Phi})&= ( (\nabla_h\bm{P}_k^{RT}\bm{u})^T\bm{P}_k^{RT}\bm{\Phi}-(\nabla\bm{u})^T\bm{\Phi},\bm{v}_{h})-\frac{1}{2}\langle\bm{\Pi}_k^{\partial}\bm{\Phi}\otimes\widehat{\bm{v}}_{h} \bm{n},\bm{P}_k^{RT}\bm{u}\rangle_{\partial\mathcal{T}_h}\nonumber\\
	&\quad+\frac{1}{2}\langle \bm{\Pi}_k^{\partial}\bm{u}\otimes\widehat{\bm{v}}_{h}\bm{n},\bm{P}_k^{RT}\bm{\Phi}\rangle_{\partial\mathcal{T}_h}
	-\frac{1}{2}\langle\bm{P}_k^{RT} \bm{u}\otimes\bm{v}_{h}\bm{n},\bm{P}_k^{RT}\bm{\Phi}\rangle_{\partial\mathcal{T}_h}.
	\end{align*}
\end{lemma}
\begin{proof} 
	\color{black}
	By direct calculations, we have
	\color{black}
	\begin{align}
	&-(\bm{P}_k^{RT}\bm{u}\otimes\bm{v}_{h}, \nabla_h\bm{P}_k^{RT}\bm{\Phi})+\langle\bm{\Pi}_k^{\partial}\bm{u}\otimes \widehat{\bm{v}}_{h} \bm{n},\bm{P}_k^{RT}\bm{\Phi}\rangle_{\partial\mathcal{T}_h}\nonumber\\
	&\quad=((\nabla\bm{u})^T \bm{\Phi},\bm{v}_{h})+( (\nabla_h\bm{P}_k^{RT}\bm{u})^T\bm{P}_k^{RT}\bm{\Phi}-(\nabla\bm{u})^T\bm{\Phi},\bm{v}_{h})\nonumber\\
	&\qquad+\langle\bm{\Pi}_k^{\partial}\bm{u}\otimes\widehat{\bm{v}}_{h} \bm{n},\bm{P}_k^{RT}\bm{\Phi}\rangle_{\partial\mathcal{T}_h}
	-\langle\bm{P}_k^{RT} \bm{u}\otimes\bm{v}_{h}\bm{n},\bm{P}_k^{RT}\bm{\Phi}\rangle_{\partial\mathcal{T}_h},\nonumber
	\end{align}
	ajd
	\begin{align}
	&-( \bm{P}_k^{RT}\bm{\Phi}\otimes\bm{v}_{h},\nabla_h\bm{P}_k^{RT}\bm{u})+\langle\bm{\Pi}_k^{\partial}\bm{\Phi}\otimes\widehat{\bm{v}}_{h}\bm{n},\bm{P}_k^{RT}\bm{u}\rangle_{\partial\mathcal{T}_h}\nonumber\\
	&\quad=((\nabla\bm{u})^T\bm{\Phi},\bm{v}_{h})+((\nabla_h\bm{P}_k^{RT}\bm{u})^T\bm{P}_k^{RT}\bm{\Phi}-(\nabla\bm{u})^T\bm{\Phi},\bm{v}_{h})\nonumber\\
	&\qquad+\langle \bm{\Pi}_k^{\partial}\bm{\Phi}\otimes\widehat{\bm{v}}_{h}\bm{n},\bm{P}_k^{RT}\bm{u}\rangle_{\partial\mathcal{T}_h}.\nonumber
	\end{align}
	Then the desired result follows immediately.

\end{proof}

\begin{lemma} For all $\bm{u}\in\bm{W}\cap [{H}^{s_{\bm{u}}}(\Omega)]^d$, $\bm{\Phi}\in\bm{W}\cap [{H}^{1+\alpha}(\Omega)]^d$ and $\bm{v}_h\in\bm{V}_h\times\widehat{\bm{V}}_h$, there holds
	\begin{align}
	|E_{N}(\bm{u};\bm{u},\bm{\Psi}_I)|\lesssim& h^{1+s_{\bm{u}}-\frac d 2+\alpha}|\bm{u}|_{1}
	\|\bm{u}\|_{s_{\bm{u}}}
	\|\bm{\Psi}\|_{1+\alpha},\label{e71}\\
	|E_{N}^{\color{black}\star}(\mathcal{V}_h;\bm{u},\bm{\Phi})|\lesssim& h^{\alpha}|\bm{u}|_{1}\|\bm{\Phi}\|_{1+\alpha}\|\mathcal{V}_h\|_V\label{e73},\\
	|E_L(\bm{u};\bm{\Phi}_I)|+
	|E_L(\bm{\Phi};\mathcal{E}^{\bm{u}}_{h})|\lesssim& h^{s_{\bm{u}}-1+\alpha}\|\bm{u}\|_{s_{\bm{u}}}|\bm{\Phi}|_{1+\alpha}.\label{e77}
	\end{align}
\end{lemma}
\begin{proof}
	
	It is easy to obtain
	\begin{align}
	E_{N}(\bm{u};\bm{u},\bm{\Psi}_I)\lesssim&
	h^{1+s_{\bm{u}}-\frac d 2}|\bm{u}|_{1}
	\|\bm{u}\|_{s_{\bm{u}}}
	\|\bm{\Psi}_I\|_V
	\lesssim
	h^{1+s_{\bm{u}}-\frac d 2+\alpha}|\bm{u}|_{1}
	\|\bm{u}\|_{s_{\bm{u}}}
	\|\bm{\Psi}\|_{1+\alpha}\nonumber
	\end{align}
	and	
	\begin{align*}
	E_{N}^{\color{black}\star}(\mathcal{V}_h;\bm{u},\bm{\Phi})
	=& -( \nabla_h\cdot(\bm{P}_k^{RT}\bm{\Phi}\otimes \bm{v}_{h}), \bm{P}_k^{RT}\bm{u}-\bm{u})\nonumber\\
	&+\langle \bm{P}_k^{RT}\bm{\Phi}\bm{n},\bm{v}_{h}\cdot( \bm{P}_k^{RT}\bm{u}-\bm{u})\rangle_{\partial\mathcal{T}_h}\nonumber\\
	&  +( \bm{P}_k^{RT}\bm{\Phi}-\bm{\Phi},\bm{v}_{h}\cdot\nabla\bm{u})\nonumber\\
	&-\frac{1}{2}\langle(\widehat{\bm{v}}_{h}-\bm{v}_{h}) \cdot \bm{n},\bm{\Pi}_k^{\partial}\bm{\Phi}\cdot(\bm{P}_k^{RT}\bm{u}-\bm{\Pi}_k^{\partial}\bm{u})\rangle_{\partial\mathcal{T}_h}\nonumber\\
	&-\frac{1}{2}\langle (\widehat{\bm{v}}_{h}-\bm{v}_{h}) \cdot \bm{n},(\bm{\Pi}_k^{\partial}\bm{\Phi}-\bm{P}_k^{RT}\bm{\Phi})\cdot\bm{\Pi}_k^{\partial}\bm{u}\rangle_{\partial\mathcal{T}_h}
	\nonumber\\
	&-\frac{1}{2}\langle\bm{v}_{h}\cdot\bm{n},\bm{P}_k^{RT}\bm{\Phi}\cdot(\bm{P}_k^{RT}\bm{u}
	-\bm{\Pi}_k^{\partial}\bm{u}) \rangle_{\partial\mathcal{T}_h}
	\nonumber\\
	&-\frac{1}{2}\langle\bm{v}_{h}\cdot\bm{n},\bm{\Pi}_k^{\partial}\bm{\Phi}\cdot(\bm{P}_k^{RT}\bm{u}-\bm{\Pi}_k^{\partial}\bm{u}) \rangle_{\partial\mathcal{T}_h}
	\nonumber\\
	\lesssim& h^{\alpha}\|\bm{u}\|_{1}\|\bm{\Phi}\|_{1+\alpha}\|\mathcal{V}_h\|_V.
	\end{align*}
	Then the thing left is to show   \eqref{e77}. Since $\bm{\Phi}\in \bm{V}$, \color{black} $\bm{u}\in \bm{H}^{1+s}(\Omega)$ $s>\frac{1}{2}$ and $m\le \color{black}k$\color{black}, we have
	\begin{align}
	\langle \bm{n} \cdot(\nabla \bm{u}-\bm{\Pi}_m^o\nabla \bm{u}),\bm{\Pi}_k^{\partial}\bm{\Phi}\rangle_{\partial \mathcal{T}_h}
	=&\langle \bm{n} \cdot(-\bm{\Pi}_m^o\nabla \bm{u}),\bm{\Pi}_k^{\partial}\bm{\Phi}\rangle_{\partial \mathcal{T}_h}\nonumber\\
	=&\langle \bm{n} \cdot(-\bm{\Pi}_m^o\nabla \bm{u}),\bm{\Phi}\rangle_{\partial \mathcal{T}_h}\nonumber\\
	=&\langle \bm{n} \cdot(\nabla \bm{u}-\bm{\Pi}_m^o\nabla \bm{u}),\bm{\Phi}\rangle_{\partial \mathcal{T}_h}.\nonumber
	\end{align} \color{black}
	By \color{black} the approximation properties of $\bm{P}^{RT}_{k}$ and $\bm{\Pi}_m^o$, the approximation property and stability of $L^2$ projection $\bm{\Pi}_k^{\partial}$, we  easily get\color{black}
	\begin{align*}
	E_L(\bm{u};\bm{\Phi}_I)=& \nu\langle \bm{n} \cdot(\nabla \bm{u}-\bm{\Pi}_m^o\nabla \bm{u}),\bm{P}^{RT}_k\bm{\Phi}-\bm{\Pi}_k^{\partial}\bm{\Phi}\rangle_{\partial \mathcal{T}_h}\nonumber\\
	&+\nu\langle\tau(\bm{P}^{RT}_{k}\bm{u}-\bm{u}),\bm{\Pi}_k^{\partial}\bm{P}^{RT}_k\bm{\Phi}-\bm{\Pi}_k^{\partial}\bm{\Phi} \rangle_{\partial\mathcal{T}_h}\nonumber\\
	=&\nu\langle \bm{n} \cdot(\nabla \bm{u}-\bm{\Pi}_m^o\nabla \bm{u}),\bm{P}^{RT}_k\bm{\Phi}-\bm{\Phi}\rangle_{\partial \mathcal{T}_h}\nonumber\\
	&+\nu\langle\tau(\bm{P}^{RT}_{k}\bm{u}-\bm{u}),\color{black}\bm{\Pi}_k^{\partial}\color{black}(\bm{P}^{RT}_k\bm{\Phi}-\bm{\Phi}) \rangle_{\partial\mathcal{T}_h}\nonumber\\
	\lesssim&\nu h^{s_{\bm{u}}-1+\alpha}\|\bm{u}\|_{s_{\bm{u}}}|\bm{\Phi}|_{1+\alpha},
	\end{align*}
	and
	\begin{align*}
	E_L(\bm{\Phi};\mathcal{E}^{\bm{u}}_{h})&=\nu\langle \bm{n} \cdot(\nabla \bm{\Phi}-\bm{\Pi}_m^o\nabla \bm{\Phi}),(\bm{P}^{RT}_k\bm{u}-\bm{u}_{h})-(\bm{\Pi}_k^{\partial}\bm{u}-\widehat{\bm{u}}_h)\rangle_{\partial \mathcal{T}_h}\nonumber\\
	&\quad+\nu\langle\tau(\bm{P}^{RT}_{k}\bm{\Phi}-\bm{\Phi}),(\bm{P}^{RT}_k\bm{u}-\bm{u}_{h})-(\bm{\Pi}_k^{\partial}\bm{u}-\widehat{\bm{u}}_h) \rangle_{\partial\mathcal{T}_h}\nonumber\\
	&\lesssim \nu h^{\alpha}|\bm{\Phi}|_{1+\alpha}\| \mathcal{U}_I-\mathcal{U}_{h}\|_V \nonumber\\
	&\lesssim \nu h^{s_{\bm{u}}-1+\alpha}\|\bm{u}\|_{s_{\bm{u}}}|\bm{\Phi}|_{1+\alpha},
	\end{align*}
	\color{black}
	which yield the desired result.
	\color{black}
\end{proof}

\begin{theorem}\label{th51}
	Let $(\bm{u},p)\in [H^{r_{\bm u}}(\Omega)]^d\times H^{r_p}(\Omega)$ and $(\mathbb{L}_h,\bm{u}_h,\widehat{\bm{u}}_h,p_h,\widehat{p}_h)\in \mathbb{K}_h\times[\bm{V}_h\times \widehat{\bm{V}}_h]\times[ Q^0_h\times \widehat{Q}_h]$  be the solutions to $(\ref{Or-NS})$ and  $(\ref{HDG11})$-\eqref{HDG33}, respectively. Then it holds   the error estimate
	\begin{align}
	&\|\bm{u}-\bm{u}_{h}\|_{0}\lesssim  \delta^{-1}{\nu}^{-1}h^{s_{\bm{u}}-1+\alpha}(h^{s_{\bm{u}}+1-\frac d 2}\|\bm{u}\|_{s_{\bm{u}}}|\bm{u}|_{1}+\nu \|\bm{u}\|_{s_{\bm{u}}})
	+h^{s_{\bm{u}}-1+\alpha}\|\bm{u}\|_{s_{\bm{u}}}, \label{LU}
	\end{align}
	where $\alpha\in (\frac{1}{2},1].$ 
\end{theorem}
\begin{proof}
	Similar to Lemma $\ref{lemma4.1}$, for all $(\mathcal{V}_h,\mathcal{Q}_h)\in [\bm{V}_h\times\widehat{\bm{V}}_h]\times [Q_h\times\widehat{Q}_h]$ it holds
	\begin{align*}
	&\nu(K_h{\bm{\Phi}}_I,K_h\mathcal{V}_h)	-d_h(\mathcal{V}_{h},{\Psi}_I)+s_h({\bm{\Phi}}_I;\mathcal{V}_h)
	- b_h(\mathcal{U}_I;{\bm{\Phi}}_I,\mathcal{V}_h)
	+b_h(\mathcal{V}_h;\mathcal{U}_I,{\bm{\Phi}}_I)\nonumber\\
	&\qquad=(e_h^{\bm{u}},\bm{v}_h)+E_L(\bm{u};\mathcal{V}_h) +E_{N}(\bm{u},\bm{u};\mathcal{V}_h)+E_{N}^{\color{black}\star}(\mathcal{V}_h;\bm{u},\bm{\Phi})\\
	&d_h(\bm{\Phi}_I;\mathcal{Q}_h)=0.
	\end{align*}
	Then there holds
	\begin{align*} 
	\|e_h^{\bm{u}}\|^2_0=&\nu(K_h{\bm{\Phi}}_I,K_h\mathcal{E}_h^{\bm{u}})-d_h(\mathcal{E}_h^{\bm{u}};{\Psi}_I)+s_h({\bm{\Phi}}_I;\mathcal{E}_h^{\bm{u}})
	- b_h(\mathcal{U}_I;{\bm{\Phi}}_I,\mathcal{E}_h^{\bm{u}})
	\nonumber\\
	&	+b_h(\mathcal{E}_h^{\bm{u}};\mathcal{U}_I,{\bm{\Phi}}_I)-E_L(\bm{\Phi};\mathcal{E}_h^{\bm{u}})
	-E_{N}(\bm{u};\bm{\Phi},\mathcal{E}_h^{\bm{u}})
	-E_{N}^{\color{black}\star}(\mathcal{E}_h^{\bm{u}};\bm{u},\bm{\Phi}) \nonumber\\
	=&b_h(\mathcal{E}_h^{\bm{u}};\mathcal{E}_h^{\bm{u}},\bm{\Phi}_I)
	+E_L(\bm{u};\bm{\Phi}_I)+E_{N}(\bm{u};\bm{u},\bm{\Phi}_I)
	\nonumber\\
	&-E_L(\bm{\Phi};\mathcal{E}_h^{\bm{u}})
	-E_{N}(\bm{u};\bm{\Phi},\mathcal{E}_h^{\bm{u}})
	-E_{N}^{\color{black}\star}(\mathcal{E}_h^{\bm{u}};\bm{u},\bm{\Phi}). 
	\end{align*}
	which, together with the regularity \color{black}\eqref{sta-assum}, indicate\color{black}
	\begin{eqnarray*}
		\|e_h^{\bm{u}}\|_0\lesssim \delta^{-1}{\nu}^{-1}h^{s_{\bm{u}}-1+\alpha}(h^{s_{\bm{u}}+1-\frac d 2}\|\bm{u}\|_{s_{\bm{u}}}|\bm{u}|_{1}+\nu \|\bm{u}\|_{s_{\bm{u}}}).
	\end{eqnarray*}
	This estimate, together with  triangle inequality  and the approximation property of the operator $\bm{P}^{RT}_{k}$, leads to the desired conclusion.
\end{proof}

\begin{corollary} If $\Omega$ is convex, $\delta^{-1}{\nu}^{-1}(h^{s_{\bm{u}}+1-\frac d 2}|\bm{u}|_{1}+\nu )\lesssim 1$, and the true solution $(\bm{u},p)$ is smooth enough, we will have the following optimal estimates:
	\begin{subequations}
	\begin{align}
	\|\bm{u}-\bm{u}_{h}\|_{0}&\lesssim h^{k+1}\|\bm{u}\|_{k+1},\label{opt-u0}\\  %
	\|\mathbb L-\mathbb L_{h}\|_{0}&\lesssim h^{k}\|\bm{u}\|_{k+1},\label{LL} \\ %
	\|p-p_{h}\|_0&\lesssim h^{k}(\|\bm{u}\|_{k+1}+\|p\|_k).\label{opt-p0}
	\end{align}
\end{subequations}
\end{corollary}

\section{Numerical experiments}

\color{black}
In this section, we provide some numerical results to verify the   performance  of the proposed  HDG finite element methods 
for the Navier-Stokes model \eqref{Or-NS} in two dimensional. All tests  are programmed in C++ using the Eigen \cite{EIGEN} library.  

We take $\Omega=\color{black}(0,1)\times(0,1)$ and the \text{Re}ynolds number $\text{Re}=1$ in the two examples below.

\subsection*{Example 7.1}
The forcing term $\bm{f}$ is chosen so that the analytical solution to $(\ref{Or-NS})$, with \color{black} the homogeneous \color{black} Dirichlet boundary condition, is given by \color{black}
\begin{eqnarray}
\left\{
\begin{aligned}
u_1&=-x^2(x-1)^2y(y-1)(2y-1),\\
u_2&=x(x-1)(2x-1)y^2(y-1)^2,\\
\color{black}
p&=\color{black}
10\left(\left(x-\frac 1 2\right)^3y^2+(1-x)^3\left(y-\frac 1 2\right)^3\right).
\end{aligned}
\right.
\end{eqnarray}

\color{black}The computational mesh is a regular triangulation of $\Omega$. The results of the relative errors of the velocity and pressure approximations for $k=1,2,3$ are presented in Table \ref{tab2}.
\color{black}
The obtained optimal convergence rates  agree with our theoretical results \eqref{opt-u0}, \eqref{LL} and   \eqref{opt-p0}. In addition, 
the velocities are truly divergence-free.

\begin{table}[h]
	\small
	\caption{\label{tab2}\color{black}
		History of convergence for Example 7.1
		\color{black}}
	\centering
	
	\begin{tabular}{c|c|c|c|c|c|c|c|c}
		\Xhline{1pt}
		
		\multirow{2}{*}{$k$} &
		\multirow{2}{*}{$\frac{\sqrt 2}{h}$} &
		
		\multicolumn{2}{c|}{$\frac{\|\bm{u}-\bm{u}_{h}\|_0}{\|\bm{u}\|_0}$} &
		\multicolumn{2}{c|}{$\frac{\|\mathbb L-\mathbb L_h\|_0}{\|\mathbb L\|_0}$} &
		
		\multicolumn{2}{c| }{$\frac{\|p-p_{h}\|_0}{\|p\|_0}$} &
		\multicolumn{1}{c }{
			\footnotesize  $\int_{\Omega}|\nabla\cdot\bm{u}_{h}|dx$
		} \\
		\cline{3-9}
		
		& &Error &Rate  &Error &Rate  &Error &Rate &Error \\
		\hline
		
		\multirow{5}{*}{$1$}
		&4	&1.6593E-01	&	    &1.5900E-01	&	    &5.1047E-01	&	    &3.94E-16\\
		&8	&4.5065E-02	&1.88 	&5.2372E-02	&1.60 	&2.7139E-01	&0.91 	&7.72E-17\\
		&16	&1.1561E-02	&1.96 	&2.0208E-02	&1.37 	&1.3775E-01	&0.98 	&1.12E-15\\
		&32	&2.9109E-03	&1.99 	&9.1087E-03	&1.15 	&6.9135E-02	&0.99 	&3.30E-14\\
		&64	&7.2931E-04	&2.00 	&4.4150E-03	&1.04 	&3.4600E-02	&1.00 	&4.70E-15\\

		\Xhline{1pt}
		\multirow{5}{*}{$2$}
		&4	&2.8355E-02	&	    &2.9255E-02	&	    &1.0923E-01	&	    &1.15E-14\\
		&8	&3.6823E-03	&2.94 	&5.1938E-03	&2.49 	&2.8525E-02	&1.94 	&7.08E-15\\
		&16	&4.6102E-04	&3.00 	&1.0012E-03	&2.37 	&7.2087E-03	&1.98 	&1.83E-14\\
		&32	&5.7322E-05	&3.01 	&2.1915E-04	&2.19 	&1.8070E-03	&2.00 	&6.67E-15\\
		&64	&7.1330E-06	&3.01 	&5.1758E-05	&2.08 	&4.5206E-04	&2.00 	&2.28E-13\\

		\Xhline{1pt}
		\multirow{5}{*}{$3$}
		&4	&3.8396E-03	&	    &4.4818E-03	&	    &1.3945E-02	&	    &3.40E-15\\
		&8	&2.6811E-04	&3.84 	&4.4213E-04	&3.34 	&1.7949E-03	&2.96 	&1.75E-15\\
		&16	&1.7149E-05	&3.97 	&5.0968E-05	&3.12 	&2.2600E-04	&2.99 	&1.00E-14\\
		&32	&1.0762E-06	&3.99 	&6.3021E-06	&3.02 	&2.8301E-05	&3.00 	&1.48E-14\\
		&64	&6.7274E-08	&4.00 	&7.9099E-07	&2.99 	&3.5393E-06	&3.00 	&1.17E-12\\

		\Xhline{1pt}
	\end{tabular}

\end{table}

\color{black}
\subsection*{Example 7.2}
This numerical example  is from \cite{MR3683678}, 
forcing term  $\bm f=(0,10^6(3y^2$ $-y+1))^T$, and the  analytical solution to $(\ref{Or-NS})$, with \color{black} the homogeneous \color{black} Dirichlet boundary condition, is given by
$$\bm u=0,\quad  p=10^6(y^3-y^2/2+y+7/12).$$ 
In \cite{MR3683678}, it has been shown that
the discrete velocity solved by Taylor-Hood space $\mathcal P_2-\mathcal P_1$ is far from being equal to zero for Stokes equations. We present the errors of the velocity and pressure approximations by our HDG method with $k=1,2,3$ and the Taylor-Hood $\mathcal P_2-\mathcal P_1$ method in Tables \ref{tab3}-\ref{tab4}. We can see that   the discrete velocity obtained by our HDG method is very close to zero, while the velocity obtained by Taylor-Hood $\mathcal P_2-\mathcal P_1$ method is not so accurate, which is influenced by  the  pressure approximation. This confirms that our proposed HDG method is indeed pressure-robust.

\begin{table}[H]
	\caption{\label{tab3}\color{black}
		Results  of $\|u-u_h\|_0$ for Example 7.2
		\color{black}}
	\centering
	
	\begin{tabular}{c|c|c|c|c}
		\Xhline{1pt}
		
		&
		\multicolumn{3}{c|} {HDG} & \multicolumn{1}{c}{Taylor-Hood}	\\
		\cline{2-4}
		$\frac{\sqrt{2}}{h}$	  &$k=1$ &$k=2$ &$k=3$ &  $\mathcal P_2-\mathcal P_1$ \\ 
		\hline
		
		10	&3.1264E-10	&1.6895E-10	&5.4894E-10	&1.4948E+00\\
		20	&4.7678E-10	&5.1265E-10	&5.2989E-10	&9.4042E-02\\
		40	&1.1007E-09	&1.1615E-09	&1.0966E-08	&5.8943E-03\\
		80	&3.1464E-09	&1.9250E-09	&1.5895E-09	&3.6885E-04\\

		\Xhline{1pt}
	\end{tabular}
\end{table}
\begin{table}[h]
	\footnotesize 
	\caption{\label{tab4}
		History of convergence for $\|p-p_h\|_0/\|p\|_0$ Example 7.2
		\color{black}}
	\centering
	
	\begin{tabular}{c|c|c|c|c|c|c|c|c}
		\Xhline{1pt}

		\multirow{2}{*}{ $\frac{\sqrt 2}{h}$} &
		
		\multicolumn{2}{c|}{HDG: $k=1$} &
		\multicolumn{2}{c|}{HDG: $k=2$} &
		\multicolumn{2}{c|}{HDG: $k=3$} &
		\multicolumn{2}{c }{$\mathcal P_2-\mathcal P_1$
		} \\
		\cline{2-9}
		
		&Error &Rate  &Error &Rate  &Error &Rate &Error &Rate \\
		\hline
		
		10	&9.2848E-02	&	    &1.8531E-03	&	    &3.4725E-05	&	    &2.4070E-03	&\\
		20	&4.6471E-02	&1.00 	&4.6392E-04	&2.00 	&4.3406E-06	&3.00 	&5.9988E-04	&2.00\\ 
		40	&2.3241E-02	&1.00 	&1.1602E-04	&2.00 	&5.4258E-07	&3.00 	&1.4984E-04	&2.00 \\
		80	&1.1621E-02	&1.00 	&2.9007E-05	&2.00 	&6.7822E-08	&3.00 	&3.7452E-05	&2.00 \\

		\Xhline{1pt}
	\end{tabular}

\end{table}
%





\begin{appendices}
	
	\section{  Interpolation properties}
	\color{black}
	This section is devoted to some standard interpolation properties  used in this paper.  
	
	\color{black}
	\subsection{$L^2$ projection}
	For any $T\in\mathcal{T}_h$, $ E\in\mathcal{E}_h$ and any   integer $r\ge 0$, let $\Pi_r^o: L^2(T)\rightarrow \mathcal{P}_{r}(T)$ and $\Pi_r^{\partial}: L^2(E)\rightarrow \mathcal{P}_{r}(E)$ be the usual $L^2$-projection operators. From \cite{book2} we have Lemmas \ref{lem3.1}  and \ref{lem3.2}.
	\begin{lemma}\label{lem3.1}  Let $r$ be a nonnegative integer and $\rho\in [1,\infty]$. Assume that $2\le \frac{d\rho}{d-(r+1)\rho}$ when $(r+1)\rho<d$. Then, for $j\in \{0,1,\cdots,r+1\}$,  there exists a constant $C$ independent of $T$ such that the estimate
		\begin{eqnarray}
		|v-\Pi_r^ov|_{j,\mu,T}\le Ch_T^{r+1-j+\frac{d}{\mu}-\frac{d}{\rho}}|v|_{r+1,\rho,T},\qquad \forall v\in W^{r+1,\rho}(T)
		\end{eqnarray}
		holds for
		$\mu$ satisfying
		\begin{eqnarray}
		\left\{
		\begin{aligned}
		&\rho\le \mu\le \frac{d\rho}{d-(r+1-j)\rho},&\text{ if } \ (r+1-j)\rho<d,\\
		&\rho\le \mu< \infty,&\text{ if } \ (r+1-j)\rho=d,\\
		&\rho\le \mu\le \infty,&\text{ if } \ (r+1-j)\rho>d,
		\end{aligned}
		\right.
		\end{eqnarray}
		and the estimate
		\begin{eqnarray}
		|\nabla^j(v-\Pi_r^ov)|_{0,\mu,\partial T}\le Ch_T^{r+1-j+\frac{d-1}{\mu}-\frac{d}{\rho}}|v|_{r+1,\rho,T},\forall v\in W^{r+1,\rho}(T)
		\end{eqnarray}
		holds
		for $\mu$ satisfying
		\begin{eqnarray}
		\left\{
		\begin{aligned}
		&\rho\le \mu\le \frac{(d-1)\rho}{d-(r+1-j)\rho},&\text{ if } \ (r+1-j)\rho<d,\\
		&\rho\le \mu< \infty,&\text{ if } \ (r+1-j)\rho=d,\\
		&\rho\le \mu\le \infty,&\text{ if } \ (r+1-j)\rho>d.
		\end{aligned}
		\right.
		\end{eqnarray}
		
	\end{lemma}

	\begin{lemma}\label{lem3.2}  Let $r$ be a nonnegative integer and $\rho\in [1,\infty]$. Assume that $2\le \frac{(d-1)\rho}{d-(r+1)\rho}$ when $(r+1)\rho<d$. Then there exists a constant $C$ independent of $T$ such that
		\begin{eqnarray}
		\|v-\Pi_r^{\partial}v\|_{0,\mu,\partial T}\le Ch_T^{r+1+\frac{d-1}{\mu}-\frac{d}{\rho}}|v|_{r+1,\rho, T},\forall v\in W^{r+1,\rho}(T)
		\end{eqnarray}
		holds for  $\mu$ satisfying
		\begin{eqnarray}
		\left\{
		\begin{aligned}
		&\rho\le \mu\le \frac{(d-1)\rho}{d-(r+1)\rho},&\text{ if } \ (r+1)\rho<d,\\
		&\rho\le \mu< \infty,&\text{ if } \ (r+1)\rho=d,\\
		&\rho\le \mu\le \infty,&\text{ if } \ (r+1)\rho>d.
		\end{aligned}
		\right.
		\end{eqnarray}
		
	\end{lemma}
	
	By following the same lines as in the proofs of the above two lemmas in \cite{book2}, it is easy to   get the following more general results, i.e. Lemmas \ref{lem3.3}  and \ref{lem3.4}.
	
	\begin{lemma}[Approximation properties for ${\Pi}_r^o$]\label{lem3.3} Let $r$ be a nonnegative integer and $\rho\in [1,\infty]$. For $j\in \{0,1,\cdots,r+1\}$ and $s\in\{j,j+1,\cdots,r+1\}$, assume that $2\le \frac{d\rho}{d-s\rho}$ when $s\rho<d$. Then there exists a constant $C$ independent of  $T$ such that
		\begin{eqnarray}\label{I1}
		|v-\Pi_r^ov|_{j,\mu,T}\le Ch_T^{s-j+\frac{d}{\mu}-\frac{d}{\rho}}|v|_{s,\rho,T},\forall v\in W^{s,\rho}(T)
		\end{eqnarray}
		holds for   $\mu$ satisfying
		\begin{eqnarray}
		\left\{
		\begin{aligned}
		&1\le \mu\le \frac{d\rho}{d-(s-j)\rho},&\text{ if } \ (s-j)\rho<d,\\
		&1\le \mu< \infty,&\text{ if } \ (s-j)\rho=d,\\
		&1\le \mu\le \infty,&\text{ if } \ (s-j)\rho>d.
		\end{aligned}
		\right.
		\end{eqnarray}
		In addition,
		for $j\in \{0,1,\cdots,r+1\}$ and $s\in\{j+1,j+2,\cdots,r+1\}$, assume that $2\le \frac{d\rho}{d-s\rho}$ when $s\rho<d$. Then there exists a constant $C$ independent of $T$ such that
		\begin{eqnarray}
		|\nabla^j(v-\Pi_r^ov)|_{0,\mu,\partial T}\le Ch_T^{s-j+\frac{d-1}{\mu}-\frac{d}{\rho}}|v|_{s,\rho,T},\forall v\in W^{s,\rho}(T)
		\label{Qik2}
		\end{eqnarray}
		holds for   $\mu$ satisfying
		\begin{eqnarray}
		\left\{
		\begin{aligned}
		&1\le \mu\le \frac{(d-1)\rho}{d-(s-j)\rho},&\text{ if } \ (s-j)\rho<d,\\
		&1\le \mu< \infty,&\text{ if } \ (s-j)\rho=d,\\
		&1\le \mu\le \infty,&\text{ if } \ (s-j)\rho>d.
		\end{aligned}
		\right.
		\end{eqnarray}
		
	\end{lemma}

	\begin{lemma}[Approximation properties for ${\Pi}_r^{\partial}$]\label{lem3.4} Let $r$ be a nonnegative integer and $\rho\in [1,\infty]$. For $s\in\{1,2,\cdots,r+1\}$, assume that $2\le \frac{d\rho}{d-s\rho}$ when $s\rho<d$. Then there exists a constant $C$ independent of $T$ such that
		\begin{eqnarray}
		\|v-\Pi_r^{\partial}v\|_{0,\mu,\partial T}\le Ch_T^{s+\frac{d-1}{\mu}-\frac{d}{\rho}}|v|_{s,\rho, T},\forall v\in W^{s,\rho}(T)
		\label{Qbk}
		\end{eqnarray}
		holds for  $\mu$ satisfying
		\begin{eqnarray}
		\left\{
		\begin{aligned}
		&1\le \mu\le \frac{(d-1)\rho}{d-s\rho},&\text{ if } \ s\rho<d,\\
		&1\le \mu< \infty,&\text{ if } \  s\rho=d,\\
		&1\le \mu\le \infty,&\text{ if } \ s\rho>d.
		\end{aligned}
		\right.
		\end{eqnarray}
		
	\end{lemma}

	We also need the following boundeness results for the projections.

	\begin{lemma}[Boundedness properties for ${\Pi}_r^o$ and $\Pi_r^{\partial}$]\label{lemma2.5} Let $r$ be a nonnegative integer. Then, for any $\mu\in [2,\infty]$
		and $j\in\{0,1,\cdots,r+1\}$, it holds
		\begin{align}
		|\Pi_r^ov|_{j,\mu,T}\lesssim&|v|_{j,\mu,T},\forall v\in L^{2}(T),\label{bq1}\\
		\|\Pi_r^{\partial}v\|_{0,\mu,\partial T}\lesssim&\|v\|_{0,\mu,\partial T},\forall v\in L^{2}(\partial T),\label{bq2}\\
		\|\Pi_r^{\partial}v\|_{0,6,\partial T}
		\lesssim&h_T^{-\frac 1 6}(|v|_{1, T}+\|v\|_{0,6, T}),\forall v\in H^1(T),\label{bq3}\\
		\|v\|_{0,6,\partial T}
		\lesssim&h_T^{-\frac 1 6}(|v|_{1, T}+\|v\|_{0,6, T}),\forall v\in H^1(T),\label{bq4}\\
		\|\Pi_r^ov\|_{0,6,\partial T}
		\lesssim&h_T^{-\frac 1 6}(|v|_{1, T}+\|v\|_{0,6, T}),\forall v\in H^1(T).\label{bq5}
		\end{align}
	\end{lemma}
	\begin{proof}
		We first show \eqref{bq1}.
		For $j=0$, by an scaling argument we have
		\begin{align}\label{341}
		\|\Pi_r^ov\|_{0,\mu, T}\lesssim&h_T^{-d(\frac{1}{2}-\frac{1}{\mu})} \|\Pi_r^ov\|_{0, T} \nonumber\\
		\le&h_T^{-d(\frac{1}{2}-\frac{1}{\mu})} \|v\|_{0,T} \nonumber\\
		\lesssim&h_T^{-d(\frac{1}{2}-\frac{1}{\mu})} h_T^{\frac{d}{2}}\|\hat{v}\|_{0,\hat{T}} \nonumber\\
		\lesssim&h_T^{-d(\frac{1}{2}-\frac{1}{\mu})} h_T^{\frac{d}{2}}\|\hat{v}\|_{0,\mu, \hat{T}} \nonumber\\
		\lesssim&h_T^{-d(\frac{1}{2}-\frac{1}{\mu})} h_T^{\frac{d}{2}-\frac{d}{\mu}}\|{v}\|_{0,\mu, T}
		=\|{v}\|_{0,\mu, T},
		\end{align}
		For $j\ge 1$, we obtain
		\begin{align*}
		|\Pi_r^ov|_{j,\mu, T}=&|\Pi_r^ov-\Pi_{j-1}^ov|_{j,\mu, T} \nonumber\\
		=&|\Pi_r^o(v-\Pi_{j-1}^ov)|_{j,\mu, T} \nonumber\\
		\lesssim&h_T^{-j-d(\frac{1}{2}-\frac{1}{\mu})}\|\Pi_r^o(v-\Pi_{j-1}^ov)\|_{0, T} \nonumber\\
		\le&h_T^{-j-d(\frac{1}{2}-\frac{1}{\mu})}\|v-\Pi_{j-1}^ov\|_{0, T} \nonumber\\
		\le&|v|_{j,\mu, T}.
		\end{align*}
		Combining the above two estimates yields  \eqref{bq1}.
		
		Secondly, let's prove  \eqref{bq2}.
		Similar to \eqref{341}, it holds
		\begin{align*}
		\|\Pi_r^{\partial}v\|_{0,\mu,\partial T}\lesssim&h_T^{-(d-1)(\frac{1}{2}-\frac{1}{\mu})} \|\Pi_r^{\partial}v\|_{0,\partial T} \nonumber\\
		\le&h_T^{-(d-1)(\frac{1}{2}-\frac{1}{\mu})} \|v\|_{0,\partial T} \nonumber\\
		\lesssim&h_T^{-(d-1)(\frac{1}{2}-\frac{1}{\mu})} h_T^{\frac{d-1}{2}}\|\hat{v}\|_{0,\partial \hat{T}} \nonumber\\
		\lesssim&h_T^{-(d-1)(\frac{1}{2}-\frac{1}{\mu})} h_T^{\frac{d-1}{2}}\|\hat{v}\|_{0,\mu,\partial \hat{T}} \nonumber\\
		\lesssim&h_T^{-(d-1)(\frac{1}{2}-\frac{1}{\mu})} h_T^{\frac{d-1}{2}}h_T^{\frac{1-d}{\mu}}\|{v}\|_{0,\mu,\partial T} 
		=\|{v}\|_{0,\mu,\partial T},
		\end{align*}
		which means   \eqref{bq2}.
		
		Next we  derive 
		\eqref{bq3}. We have
		\begin{align*}
		\|\Pi_r^{\partial}v\|_{0,6,\partial T}\le&\|\Pi_r^{\partial}(v-\Pi_r^ov)\|_{0,6,\partial T}+\|\Pi_r^ov\|_{0,6,\partial T}\nonumber\\
		\lesssim&\|v-\Pi_r^ov\|_{0,6,\partial T}+h_T^{-\frac 1 6}\|\Pi_r^ov\|_{0,6, T}\nonumber\\
		\lesssim&h_T^{5/6-d/3}|v|_{1, T}+h_T^{-\frac 1 6}\|\Pi_r^ov\|_{0,6, T}\nonumber\\
		\lesssim&h_T^{-\frac 1 6}(|v|_{1, T}+\|v\|_{0,6, T}),
		\end{align*}
		i.e. \eqref{bq3} holds.
		
		Now we turn to  show \eqref{bq4}. It is easy to see
		\begin{align*}
		\|v\|_{0,6,\partial T}\le&\|v-\Pi_r^{\partial}v\|_{0,6,\partial T}+\|\Pi_r^{\partial}v\|_{0,6,\partial T} \nonumber\\
		\lesssim&h_T^{5/6-d/3}|v|_{1, T}+h_T^{-\frac 1 6}(|v|_{1, T}+\|v\|_{0,6, T})\nonumber\\
		\lesssim&h_T^{-\frac 1 6}(|v|_{1, T}+\|v\|_{0,6, T}),
		\end{align*}
		which yields \eqref{bq4}.
		
		Finally, it remains to prove
		\eqref{bq5}. We easily know
		\begin{align*}
		\|\Pi_r^ov\|_{0,6,\partial T}\le&\|\Pi_r^ov-\Pi_r^{\partial}v\|_{0,6,\partial T}+\|\Pi_r^{\partial}v\|_{0,6,\partial T} \nonumber\\
		\lesssim&\|\Pi_r^ov-v\|_{0,6,\partial T}+\|\Pi_r^{\partial}v\|_{0,6,\partial T} \nonumber\\
		\lesssim&h_T^{5/6-d/3}|v|_{1, T}+h_T^{-\frac 1 6}(|v|_{1, T}+\|v\|_{0,6, T})\nonumber\\
		\lesssim&h_T^{-\frac 1 6}(|v|_{1, T}+\|v\|_{0,6, T}),
		\end{align*}
		which finishes the proof.
	\end{proof}
	
	\subsection{RT interpolation}

	For any nonnegative integer $r$, we introduce the local Raviart-Thomas (RT) element
	\begin{eqnarray}
	\bm{RT}_r(T)=[\mathcal{P}_r(T)]^d+\bm{x}\mathcal{P}_r(T).\nonumber
	\end{eqnarray}
	
	Lemmas \ref{lemma3.1}-\ref{lemma33} show   some properties of the RT projection  \cite{BDM}.  
	\begin{lemma}\label{lemma3.1} For any $\bm{v}_{h}\in \bm{RT}_r(T)$,  $\nabla\cdot\bm{v}_{h}|_T=0$ implies $\bm{v}_{h}\in [\mathcal{P}_r(T)]^d$\color{black}.\color{black}
		
	\end{lemma}
	\begin{lemma} \label{lemma32} For any $T\in\mathcal{T}_h$ and $\bm{v}\in [H^1(T)]^d$, there exists a unique $\bm{P}^{RT}_r\bm{v}\in \bm{RT}_r(T)$ such that
		\begin{subequations}
			\begin{align}
			\langle\bm{P}^{RT}_r\bm{v}\cdot \bm{n}_E,w_r \rangle_E=&\langle \bm{v}\cdot \bm{n}_E,w_r\rangle_E& \forall w_r\in \mathcal{P}_r(E),   E\subset\partial T,\label{RT1}\\
			(\bm{P}^{RT}_r\bm{v},\bm{w}_{r-1})_T=&(\bm{v},\bm{w}_{r-1})_T& \forall \bm{w}_{r-1}\in [P_{r-1}(T)]^d.\label{RT2}
			\end{align}
		\end{subequations}
		If $r=0$, then $\bm{P}^{RT}_r\bm{v}$ is  determined only by $(\ref{RT1})$.
		Moreover, for any nonnegative integer $s$, the following interpolation \color{black} approximation \color{black}property holds:
		\begin{align}
		\|\bm{v}-\bm{P}^{RT}_r\bm{v}\|_{0,T}&\lesssim h_T^{s}|\bm{v}|_{s,T}&\forall 1\le s\le r+1,\bm{v}\in [H^s(T)]^d.\label{RT3}
		\end{align}
		
	\end{lemma}

	\begin{lemma}\label{lemma33} The operator $\bm{P}^{RT}_r$ defined in Lemma \ref{lemma32} satisfies
		\begin{align}
		(\nabla\cdot\bm{P}^{RT}_r\bm{v},q_h)_T&=(\nabla\cdot\bm{v},q_h)_T& \forall q_h\in \mathcal{P}_r(T),T\in\mathcal{T}_h. \label{RT4}
		\end{align}
		
	\end{lemma}

	The following result is very useful to our analysis.
	
	\begin{lemma} \label{tr} For all $T\in\mathcal{T}_h$, $w\in H^1(T)$, and $\mu\in [1,\infty]$, it holds
		\begin{align}
		\|w\|_{0,\mu,\partial T}\lesssim h_T^{-1/\mu}\|w\|_{0,\mu,T}+h_T^{(\mu-1)/\mu}|{w}|_{1,\mu,{T}}.\label{partial1}
		\end{align}
		In addition, for all   $w\in  \mathcal{P}_{k}(T)$, it holds
		\begin{eqnarray}
		\|w\|_{0,\mu,\partial T}\lesssim h_T^{-1/\mu}\|w\|_{0,T}.\label{partial2}
		\end{eqnarray}
	\end{lemma}
	\begin{proof} For any $w\in {H}^1(T)$, we use the trace theorem on the reference element $\hat{T}$ to get
		\begin{align*}
		\|w\|_{0,\mu,\partial T}^{\mu}=&\int_{\partial \hat{T}}|\hat{w}|^{\mu}\frac{|\partial T|}{|\partial\hat{T}|}d\hat{s}\lesssim h_T^{d-1}\|\hat{w}\|^{\mu}_{0,\mu,\partial\hat{T}}\lesssim h_T^{d-1}\|\hat{w}\|^{\mu}_{1,\mu,\hat{T}} \nonumber\\
		=&h_T^{d-1}(\|\hat{w}\|^{\mu}_{0,\mu,\hat{T}}+|\hat{w}|^{\mu}_{1,\mu,\hat{T}})
		\nonumber\\
		\lesssim&h_T^{-1}(\|w\|^{\mu}_{0,\mu,T}+h_T^{\mu}|{w}|^{\mu}_{1,\mu,{T}})
		\nonumber\\
		\lesssim&h_T^{-1}\|w\|^{\mu}_{0,\mu,T}+h_T^{\mu-1}|{w}|^{\mu}_{1,\mu,{T}},
		\end{align*}
		\color{black} which \color{black} indicates \eqref{partial1}.
		The   result \eqref{partial2}  follows from \eqref{partial1} and the inverse inequality.
	\end{proof}

	By Lemma \ref{tr} we can get the following approximation properties of $\bm{P}_{r}^{RT}$.
	
	\begin{lemma} \label{lem3.10} For $s\in \{1,2,\cdots,r+1\}$ 
		and $j\in \{0,1,\cdots,s\}$,  it  holds
		\begin{align}
		|\bm{v}-\bm{P}_{r}^{RT}\bm{v}|_{j,\mu,T}
		&\lesssim h_T^{s-j-d(\frac{1}{2}-\frac{1}{\mu})}|\bm{v}|_{s,T} &\forall \bm{v}\in [H^s(T)]^d \label{RT-approximation2}
		\end{align}
		for  $\mu$ satisfying
		\begin{eqnarray*}
			\left\{
			\begin{aligned}
				&2\le \mu\le \frac{2d}{d-2(s-j)},&\text{ if } \ 2(s-j)<d,\\
				&2\le \mu< \infty,&\text{ if } \ 2(s-j)=d,\\
				&2\le \mu\le \infty,&\text{ if } \ 2(s-j)>d,
			\end{aligned}
			\right.
		\end{eqnarray*}
		and
		\begin{eqnarray}
		\|\nabla^j(\bm{v}-\bm{P}_{r}^{RT}\bm{v})\|_{0,\mu,\partial T}
		&\lesssim&h_T^{s-j-\frac{1}{\mu}-d(\frac{1}{2}-\frac{1}{\mu})}|\bm{v}|_{s,T}, \forall \bm{v}\in [H^s(T)]^d\label{RT-approximation3}
		\label{RT-addition3}
		\end{eqnarray}
		for $\mu$ satisfying
		\begin{eqnarray*}
			\left\{
			\begin{aligned}
				&2\le \mu\le \frac{2(d-1)}{d-2(s-j)},&\text{ if } \ 2(s-j)<d,\\
				&2\le \mu< \infty,&\text{ if } \ 2(s-j)=d,\\
				&2\le \mu\le \infty,&\text{ if } \ 2(s-j)>d.
			\end{aligned}
			\right.
		\end{eqnarray*}
	\end{lemma}
	\begin{proof} 
		By   a \color{black} triangle inequality,   an \color{black} inverse inequality, the approximation properties of $\bm{\Pi}^o_{s-1}$, and the approximation properties of $\bm{P}^{RT}_r$ in \eqref{RT3}, we  get
		\begin{align*}
		|\bm{v}-\bm{P}_{r}^{RT}\bm{v}|_{j,\mu,T}&\le  |\bm{v}-\bm{\Pi}_{s-1}^o\bm{v}|_{j,\mu,T}+|\bm{\Pi}_{s-1}^o\bm{v}-\bm{P}_{r}^{RT}\bm{v}|_{j,\mu,T} \nonumber\\
		&\lesssim  |\bm{v}-\bm{\Pi}_{s-1}^o\bm{v}|_{j,\mu,T}+h_T^{-j-d(\frac{1}{2}-\frac{1}{\mu})}\|\bm{\Pi}_{s-1}^o\bm{v}-\bm{P}_{r}^{RT}\bm{v}\|_{0,T} \nonumber\\
		&\lesssim  |\bm{v}-\bm{\Pi}_{s-1}^o\bm{v}|_{j,\mu,T}+h_T^{-j-d(\frac{1}{2}-\frac{1}{\mu})}\|\bm{\Pi}_{s-1}^o\bm{v}-\bm{v}\|_{0,T} \nonumber\\
		&\quad+h_T^{-j-d(\frac{1}{2}-\frac{1}{\mu})}\|\bm{v}-\bm{P}_{r}^{RT}\bm{v}\|_{0,T} \nonumber\\
		&\lesssim h_T^{s-j-d(\frac{1}{2}-\frac{1}{\mu})}|\bm{v}|_{s,T},
		\end{align*}
		which yields \eqref{RT-approximation2}.
		
		It remains to show  \eqref{RT-approximation3}.
		By the triangle inequality, \eqref{partial2}, the approximation properties of $\bm{\Pi}^o_{s-1}$, and the approximation property of $\bm{P}^{RT}_r$ in \eqref{RT3}, we have
		\begin{align*}
		\|\nabla^j(\bm{v}-\bm{P}_{r}^{RT}\bm{v})\|_{0,\mu,\partial T}\le&  \|\nabla^j(\bm{v}-\bm{\Pi}_{s-1}^o\bm{v})\|_{0,\mu,\partial T}+\|\nabla^j(\bm{\Pi}_{s-1}^o\bm{v}-\bm{P}_{r}^{RT}\bm{v})\|_{0,\mu,\partial T} \nonumber\\
		\lesssim&  \|\nabla^j(\bm{v}-\bm{\Pi}_{s-1}^o\bm{v})\|_{0,\mu,\partial T}+h_T^{-j-\frac{1}{\mu}}\|\bm{\Pi}_{s-1}^o\bm{v}-\bm{P}_{r}^{RT}\bm{v}\|_{0, \mu,T} \nonumber\\
		\lesssim&  \|\nabla^j(\bm{v}-\bm{\Pi}_{s-1}^o\bm{v})\|_{0,\mu,\partial T}\nonumber\\
		&+h_T^{-j-\frac{1}{\mu}-d(\frac{1}{2}-\frac{1}{\mu})}\|\bm{\Pi}_{s-1}^o\bm{v}-\bm{P}_{r}^{RT}\bm{v}\|_{0,T} \nonumber\\
		\lesssim&  \|\nabla^j(\bm{v}-\bm{\Pi}_{s-1}^o\bm{v})\|_{0,\mu,\partial T}+h_T^{-j-\frac{1}{\mu}-d(\frac{1}{2}-\frac{1}{\mu})}\|\bm{\Pi}_{s-1}^o\bm{v}-\bm{v}\|_{0,T} \nonumber\\
		&+h_T^{-j-\frac{1}{\mu}-d(\frac{1}{2}-\frac{1}{\mu})}\|\bm{v}-\bm{P}_{r}^{RT}\bm{v}\|_{0,T} \nonumber\\
		\lesssim&h_T^{s-j-\frac{1}{\mu}-d(\frac{1}{2}-\frac{1}{\mu})}|\bm{v}|_{s,T}.
		\end{align*}
		This completes the proof.
	\end{proof}
	

	\subsection{BDM interpolation}
	We  recall the definition and properties of the BDM projection $\bm{P}^{BDM}_r$ from \cite{BDM} for 2D and \cite{BDM3D} for 3D. \color{black}
	\begin{lemma}
		For any $T\in\mathcal{T}_h$, $\bm{v}\in [H^1(T)]^d$, and integer $ r\ge 1$, there exists a unique $\bm{P}^{BDM}_r\bm{v}\in [\mathcal{P}_r(T)]^d $ such that
		\begin{subequations}
			\begin{align}
			\langle\bm{P}^{BDM}_r\bm{v}\cdot \bm{n}_E,w_r \rangle_E=&\langle \bm{v}\cdot \bm{n}_E,w_r\rangle_E,\quad &\forall w_r\in \mathcal{P}_r(E),  E\color{black}\subset\color{black}\partial T,\label{BDM1}\\
			(\bm{P}^{BDM}_r\bm{v},\nabla p_{r-1})_T=&(\bm{v},\nabla p_{r-1})_T, &\forall {p}_{r-1}\in P_{r-1}(T), \label{BDM2}\\
			\label{bT}
			(\bm{P}^{BDM}_r\bm{v},{  {\bf{curl}}}(b_T p_{j-2}))_T=&(\bm{v},{  {\bf{curl}}}(b_T p_{r-2}))_T, &\forall {p}_{r-2}\in P_{r-2}(T),
			\end{align}
			when $d=2,3$ and $j\ge 2$, and
			\begin{eqnarray}
			(\bm{P}^{BDM}_r\bm{v},\bm{w})_T=(\bm{v},\bm{w})_T,\quad \forall \bm{w}\in \bm{P}^*_r(T),
			\end{eqnarray}
		\end{subequations}
		when $d=3$, where $b_T$ in \eqref{bT}  is the bubble function on $T$, and
		$$\bm{P}^*_r(T):=\left\{\bm{v}\in [\mathcal{P}_r(T)]^d:  \nabla\cdot\bm{v}=0,\
		\bm{v}\cdot\bm{n}_E=0,\ \  \forall E\color{black}\subset\color{black}\partial T\right\} .
		$$
		\color{black}
		Moreover, for any  integer $s$ with $ 1\le s\le r+1$, and  $\bm{v}\in [H^r(T)]^d$, it holds the following interpolation \color{black}approximation properties:
		\begin{eqnarray}
		&&\|\bm{v}-\bm{P}^{BDM}_r\bm{v}\|_{0,T}\lesssim h_T^{s}|\bm{v}|_{s,T} ,\label{BDM3}\\
		&&\|\bm{v}-\bm{P}^{BDM}_r\bm{v}\|_{0,\partial T}\lesssim h_T^{s-\frac{1}{2}}|\bm{v}|_{s,T} .\label{BDM31}
		\end{eqnarray}
	\end{lemma}
	
	\color{black}
	\subsection{Approximation properties extended to real number}  
	
	In this section, we will extend the approximation properties in 
	\Cref{lem3.2,lem3.3,lem3.10} to real index $s$.
	
	First, we recall the following classical results.
	\begin{theorem} [cf. {\cite[Page 220, Theorem 7.23]{MR2424078}, \cite[Page 373, Proposition 14.1.5]{MR2373954}}]
		\label{lemma366}
		Given $0<\theta<1$ and $1\le p\le \infty$, and given Banach spaces $A_1\hookrightarrow A_0$, $B_1\hookrightarrow B_0$. Let $\mathcal K$ be a bounded linear operator from $A_0+A_1$ into $B_0+B_1$ having the property that $\mathcal K$ is bounded from $A_i$ into $B_i$, with norm at most $M_i$, $i=0,1$; that is
		\begin{align*}
		\|\mathcal K u_i\|_{B_i}\le M_i\|u_i\|_{A_i}\qquad \forall u_i\in A_i, i=1,2.
		\end{align*}
		Then $\mathcal K:A_{\theta,p}\to B_{\theta,p}$ is a bounded linear operator and
		\begin{align*}
		\|\mathcal K\|_{A_{\theta,p}\to B_{\theta,p}}\le  \|\mathcal K\|_{A_{0}\to B_{0}}^{1-\theta}  \|\mathcal K\|^{\theta}_{A_{1}\to B_{1}},
		\end{align*}
		where $A_{\theta,p}:=[A_0,A_1]_{\theta,p}$, $B_{\theta,p}:=[B_0,B_1]_{\theta,p}$, see \cite[Page 372]{MR2373954} for detailed definitions.
	\end{theorem}

	\begin{theorem}[cf. {\cite[Page 375, Theorem 14.2.3]{MR2373954}}] \label{H-space}
		Let $0<s<1$, $1\le p\le \infty$ and $\ell\ge 0$ be an integer, if $\Omega$ has a Lipschitz boundary, then
		\begin{align*}
		[W^{\ell,p}(\Omega),W^{\ell+1,p}(\Omega)]_{s,p}=W^{\ell+s,p}(\Omega).
		\end{align*}
	\end{theorem}

	With the above results, we are ready to prove the following fractional approximation properties of the $L^2$-projection $\Pi_j^o$.

	\begin{lemma} Let $r$ be a nonnegative integer and $\rho\in [1,\infty]$. For $j\in \{0,1,\cdots,r+1\}$ and real number $s\in [j,r+1]$, assume that $2\le \frac{d\rho}{d-[s]\rho}$ when $[s]\rho<d$. Then there exists a constant $C$ independent of  $T$ such that
		\begin{eqnarray}\label{F1}
		|v-\Pi_r^ov|_{j,\mu,T}\le Ch_T^{s-j+\frac{d}{\mu}-\frac{d}{\rho}}\|v\|_{s,\rho,T},\forall v\in W^{s,\rho}(T)
		\end{eqnarray}
		holds for   $\mu$ satisfying
		\begin{eqnarray}
		\left\{
		\begin{aligned}
		&1\le \mu\le \frac{d\rho}{d-([s]-j)\rho},&\text{ if } \ ([s]-j)\rho<d,\\
		&1\le \mu< \infty,&\text{ if } \ ([s]-j)\rho=d,\\
		&1\le \mu\le \infty,&\text{ if } \ ([s]-j)\rho>d.
		\end{aligned}
		\right.
		\end{eqnarray}
		where $[s]$ stands for the integer part of $s$.
		In addition,
		for $j\in \{0,1,\cdots,r+1\}$ and $s\in[j+1,r+1]$, assume that $2\le \frac{d\rho}{d-[s]\rho}$ when $[s]\rho<d$. Then there exists a constant $C$ independent of $T$ such that
		\begin{eqnarray}
		|\nabla^j(v-\Pi_r^ov)|_{0,\mu,\partial T}\le Ch_T^{s-j+\frac{d-1}{\mu}-\frac{d}{\rho}}\|v\|_{s,\rho,T},\forall v\in W^{s,\rho}(T)
		\end{eqnarray}
		holds for   $\mu$ satisfying
		\begin{eqnarray}
		\left\{
		\begin{aligned}
		&1\le \mu\le \frac{(d-1)\rho}{d-([s]-j)\rho},&\text{ if } \ ([s]-j)\rho<d,\\
		&1\le \mu< \infty,&\text{ if } \ ([s]-j)\rho=d,\\
		&1\le \mu\le \infty,&\text{ if } \ ([s]-j)\rho>d.
		\end{aligned}
		\right.
		\end{eqnarray}
	\end{lemma}
	\begin{proof} We only give a proof for \eqref{F1}.
		When $s$ is an integer, the result is followed by \ref{I1}, immediately. Therefore, we assume $m-1<s<m$, where $m\ge 1$ is an integer, and
		we take $A_0=W^{m-1,\rho}(\Omega)$, $A_1=W^{m,\rho}(\Omega)$, $B_0=B_1=W^{m-1,\rho}(\Omega)$,
		$\theta=s-(m-1)$.
		From Theorem \ref{lemma366}, Theorem \ref{H-space}, \eqref{I1} and the fact $[s]=m-1$, we have
		\begin{align*}
		\frac{\|(Id-\Pi_r^o) v\|_{m-1,\rho,T}}{\| v\|_{s,\rho,T}}&\le \|(Id-\Pi_r^o)\|_{W^{s,\rho}(T)\to W^{m-1,\rho}(T)}\nonumber\\
		&\le \|(Id-\Pi_r^o)\|_{W^{m-1,\rho}(T)\to W^{m-1,\rho}(T)}^{1-\theta}\|(Id-\Pi_r^o)\|^{\theta}_{W^{m,\rho}(T)\to W^{m-1,\rho}(T)}\nonumber\\
		&\lesssim h_T^{\theta}\nonumber\\
		&=h_T^{s-(m-1)},
		\end{align*}
		where ${Id}$ is the identity operator. With the above estimate,  \eqref{I1}, and the fact $[s]=m-1$, we can get
		\begin{align*}
		|v-\Pi_r^ov|_{j,\mu,T}&=	|v-\Pi_r^ov-\Pi_r^o(v-\Pi_r^ov)|_{j,\mu,T}\nonumber\\
		&\lesssim h_T^{m-1-j+\frac{d}{\mu}-\frac{d}{\rho}}|v-\Pi_r^ov|_{m-1,\rho,T} \nonumber\\
		&\lesssim h_T^{s-j+\frac{d}{\mu}-\frac{d}{\rho}}\|v\|_{s,\rho,T}
		\end{align*}
		holds for all $v\in W^{s,\rho}(T)$.
		Then we finish our proof.
	\end{proof}
	
	The proofs for the following two lemmas can be done in the same way as above.

	\begin{lemma} Let $r$ be a nonnegative integer and $\rho\in [1,\infty]$. For $s\in\{1,2,\cdots,r+1\}$, assume that $2\le \frac{d\rho}{d-[s]\rho}$ when $[s]\rho<d$. Then there exists a constant $C$ independent of $T$ such that
		\begin{eqnarray}
		\|v-\Pi_r^{\partial}v\|_{0,\mu,\partial T}\le Ch_T^{s+\frac{d-1}{\mu}-\frac{d}{\rho}}\|v\|_{s,\rho, T},\forall v\in W^{s,\rho}(T)
		\label{Qbk}
		\end{eqnarray}
		holds for  $\mu$ satisfying
		\begin{eqnarray}
		\left\{
		\begin{aligned}
		&1\le \mu\le \frac{(d-1)\rho}{d-[s]\rho},&\text{ if } \ [s]\rho<d,\\
		&1\le \mu< \infty,&\text{ if } \  [s]\rho=d,\\
		&1\le \mu\le \infty,&\text{ if } \ [s]\rho>d.
		\end{aligned}
		\right.
		\end{eqnarray}
		
	\end{lemma}

	\begin{lemma}  For any real number $s\in [1,r+1]$ 
		and integer $j\in \{0,1,\cdots,s\}$,  it  holds
		\begin{align}
		|\bm{v}-\bm{P}_{r}^{RT}\bm{v}|_{j,\mu,T}
		&\lesssim h_T^{s-j-d(\frac{1}{2}-\frac{1}{\mu})}\|\bm{v}\|_{s,T} &\forall \bm{v}\in [H^s(T)]^d 
		\end{align}
		for  $\mu$ satisfying
		\begin{eqnarray*}
			\left\{
			\begin{aligned}
				&2\le \mu\le \frac{2d}{d-2(s-j)},&\text{ if } \ 2(s-j)<d,\\
				&2\le \mu< \infty,&\text{ if } \ 2(s-j)=d,\\
				&2\le \mu\le \infty,&\text{ if } \ 2(s-j)>d,
			\end{aligned}
			\right.
		\end{eqnarray*}
		and
		\begin{align}
		\|\nabla^j(\bm{v}-\bm{P}_{r}^{RT}\bm{v})\|_{0,\mu,\partial T}
		&\lesssim h_T^{s-j-\frac{1}{\mu}-d(\frac{1}{2}-\frac{1}{\mu})}\|\bm{v}\|_{s,T}, \forall \bm{v}\in [H^s(T)]^d
		\end{align}
		for $\mu$ satisfying
		\begin{align*}
		\left\{
		\begin{aligned}
		&2\le \mu\le \frac{2(d-1)}{d-2([s]-j)},&\text{ if } \ 2([s]-j)<d,\\
		&2\le \mu< \infty,&\text{ if } \ 2([s]-j)=d,\\
		&2\le \mu\le \infty,&\text{ if } \ 2([s]-j)>d.
		\end{aligned}
		\right.
		\end{align*}
	\end{lemma}

	\color{black}

\end{appendices}



\color{black}
\bibliographystyle{siam}
\bibliography{mybib}{}
\color{black}



\end{document}